\documentclass[11pt,a4paper]{amsart}
\usepackage[a4paper,twoside,centering]{geometry}
\usepackage[arrow,arc,matrix]{xy}
\usepackage{amssymb}

\title{Non-degenerate potentials on the quiver $X_7$}
\author{Sefi Ladkani}
\address{Department of Mathematics, University of Haifa, Mount Carmel,
Haifa 3498838, Israel}
\email{ladkani.math@gmail.com}

\newcommand{\cA}{\mathcal{A}}
\newcommand{\cC}{\mathcal{C}}
\newcommand{\eps}{\varepsilon}
\newcommand{\bF}{\mathbb{F}}
\newcommand{\cJ}{\mathcal{J}}
\newcommand{\gL}{\Lambda}
\newcommand{\fm}{\mathfrak{m}}
\newcommand{\cP}{\mathcal{P}}
\newcommand{\bQ}{\mathbb{Q}}
\newcommand{\bZ}{\mathbb{Z}}
\newcommand{\vphi}{\varphi}
\newcommand{\la}{\leftarrow}
\newcommand{\wh}{\widehat}
\newcommand{\wt}{\widetilde}

\newcommand{\llangle}{\langle \! \langle}
\newcommand{\rrangle}{\rangle \! \rangle}

\DeclareMathOperator{\ch}{char}
\DeclareMathOperator{\HH}{HH}
\DeclareMathOperator{\Hom}{Hom}
\DeclareMathOperator{\id}{id}
\DeclareMathOperator{\rank}{rank}
\DeclareMathOperator{\rot}{rot}
\DeclareMathOperator{\sub}{sub}

\newtheorem{theorem}{Theorem}[section]
\newtheorem{thm}{Theorem}

\newtheorem{prop}[theorem]{Proposition}
\newtheorem{lemma}[theorem]{Lemma}
\newtheorem{cor}[theorem]{Corollary}
\newtheorem*{cor*}{Corollary}

\theoremstyle{definition}
\newtheorem{defn}[theorem]{Definition}
\newtheorem{remark}[theorem]{Remark}
\newtheorem{notat}[theorem]{Notation}
\newtheorem{alg}[theorem]{Algorithm}
\newtheorem{example}[theorem]{Example}

\numberwithin{equation}{section}

\begin{document}

\begin{abstract}
We develop a method to compute certain mutations of quivers with
potentials and use this to construct an explicit family of non-degenerate
potentials on the exceptional quiver $X_7$.

We confirm a conjecture of Geiss-Labardini-Schr\"{o}er
by presenting a computer-assisted proof that over a ground field
of characteristic $2$, the Jacobian algebra of one member $W_0$
of this family is infinite-dimensional, whereas that of another member
$W_1$ is finite-dimensional, implying that these potentials are not
right equivalent.

As a consequence, we draw some conclusions on the associated cluster
categories, and in particular obtain a representation-theoretic proof
that there are no reddening mutation sequences for the quiver $X_7$.

We also show that when the characteristic of the ground field differs 
from $2$, the Jacobian algebras of $W_0$ and $W_1$ are both finite-dimensional.
Thus $W_0$ seems to be the first known non-degenerate potential
with the property that the finite-dimensionality of its
Jacobian algebra depends upon the ground field.
\end{abstract}

\maketitle

\section{Introduction}
%%%%%%%%%%%%%%%%%%%%%%

Quivers without loops and $2$-cycles (cycles of length $2$) form an
important ingredient in the theory of cluster algebras by Fomin and
Zelevinsky~\cite{FZ02}. The notion of 
quiver mutation, which is the process of creating a new such quiver
from a given one at a prescribed vertex, plays a central role in the theory.

A quiver can be enriched by an algebraic datum in the form of a
potential. Roughly speaking, a potential is a (possibly infinite)
linear combination of
cycles in the (complete) path algebra of the quiver. A systematic study of
quivers with potentials and their mutations has been initiated in the
work of Derksen, Weyman and Zelevinsky~\cite{DWZ08}. 
Whereas in quiver mutation the resulting quiver has no 2-cycles by
construction, this may not be the case for mutation of quivers with potentials.
We say that a potential is \emph{non-degenerate} if for any sequence of
mutations, the resulting quiver does not have any 2-cycles.

Given a quiver with potential, one can construct several algebraic objects
(such as the Ginzburg dg-algebra~\cite{Ginzburg06}, the Jacobian algebra,
a cluster category~\cite{Amiot09,Plamondon11}).
Only when the potential is non-degenerate
can these constructions be used to adequately model the corresponding cluster
algebra. This lies at the foundation of the additive categorification of
(skew-symmetric) cluster algebras, see the surveys~\cite{Keller10, Keller12}.
In addition, one can draw significant conclusions on the cluster algebra of a quiver
by investigating the Jacobian algebra of a non-degenerate potential
on it, as done in~\cite{DWZ10}.
For such applications to cluster algebras, it is important to know that
there exists a non-degenerate potential on every quiver
(without loops and $2$-cycles).

The existence of a non-degenerate potential over an uncountable
ground field was established in~\cite{DWZ08}. However, the proof is not
constructive and except for a few classes of quivers,
it is not known in general how to write a non-degenerate potential on a given
quiver.
Among the quivers for which an explicit form of a non-degenerate potential
is known we could mention those belonging to the mutation classes of Dynkin
quivers~\cite{BMR06,DWZ08};
those arising from triangulations of oriented
surfaces~\cite{FST08,Labardini09,Labardini16};
those arising from reduced expressions in Coxeter
groups~\cite{BIRS09,BIRS11,GLS11};
and certain McKay quivers~\cite{dTdVVdB13}.

Moreover, it is not known in general 
whether a non-degenerate potential on a given quiver is unique.
Here, uniqueness is considered up to right equivalence, which is
an automorphism of the complete path algebra fixing the vertices.
We have shown~\cite{Ladkani13} the uniqueness of a non-degenerate
potential for all the quivers belonging to the class $\cP$ of
Kontsevich and Soibelman~\cite{KS08}.
The question of uniqueness has also been settled for many finite mutation
classes of quivers~\cite{GLS16,GLM20}.

The remaining finite mutation class for which the answers to the above
questions were not known was that of the exceptional quiver $X_7$,
discovered by Derksen and Owen in~\cite{DerksenOwen08}
and shown in Figure~\ref{fig:X7intro}.
This mutation class consists of two quivers and it does not belong to the
class $\cP$, so no explicit description of a non-degenerate potential
was known.
A potential on $X_7$ has been suggested in the physics
literature~\cite{ACCERV13}, but without any rigorous analysis.
It has been conjectured by Geiss, Labardini-Fragoso and
Schr\"{o}er~\cite[\S9.5.1]{GLS16} that 
there are at least two non-degenerate potentials on $X_7$
(up to right equivalence).

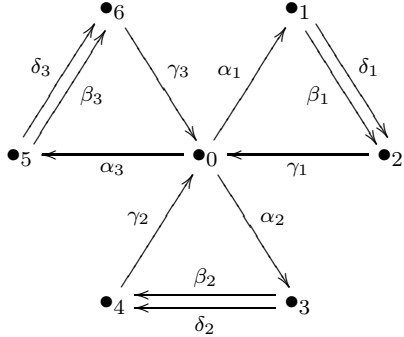
\begin{figure}
\[
\xymatrix@=1.5pc{
& {\bullet_6} \ar[ddr]^{\gamma_3} & &
{\bullet_1} \ar@<0.5ex>[ddr]^{\delta_1} \ar@<-0.5ex>[ddr]_{\beta_1} \\
\\
{\bullet_5} \ar@<0.5ex>[uur]^{\delta_3} \ar@<-0.5ex>[uur]_{\beta_3} & &
{\bullet_0} \ar[ll]^{\alpha_3} \ar[uur]^{\alpha_1} \ar[ddr]^{\alpha_2} & &
{\bullet_2} \ar[ll]^{\gamma_1} \\
\\
& {\bullet_4} \ar[uur]^{\gamma_2} & & {\bullet_3} \ar@<0.5ex>[ll]^{\delta_2}
\ar@<-0.5ex>[ll]_{\beta_2}
}
\]
\caption{The quiver $X_7$.}
\label{fig:X7intro}
\end{figure}

In this paper we construct a family of non-degenerate potentials on $X_7$
which are defined over any ground field $K$
and confirm the conjecture of~\cite{GLS16} when $K$ has characteristic $2$.
In order to describe our first result,
consider the following six $3$-cycles
$B_1, B_2, B_3, \Delta_1, \Delta_2, \Delta_3$ on $X_7$ given by
\begin{align*}
B_i = \alpha_i \beta_i \gamma_i &,&
\Delta_i = \alpha_i \delta_i \gamma_i &,&
(1 \leq i \leq 3).
\end{align*}

\begin{thm}
Let $K$ be a field and
let $P(x) \in K[[x]]$ be a power series without constant term.
Then the potential on the quiver $X_7$ given by
\begin{equation} \label{e:X7:W}
W = B_1 + B_2 + B_3
+ \Delta_1 \Delta_2 + \Delta_2 \Delta_3 + \Delta_3 \Delta_1
+ P(\Delta_1 \Delta_2 \Delta_3)
\end{equation}
is non-degenerate.
\end{thm}

The theorem appears as~Theorem~\ref{t:X7:potential} and its proof
is based on the material in Sections~\ref{sec:QP:mut}
and~\ref{sec:X7:mut}.

Our next result concerns the potentials $W_0$ and $W_1$
corresponding to the choices of power series
$P(x)=0$ and $P(x)=x$ in the theorem.
In explicit terms,
\begin{align*}
W_0 &= B_1 + B_2 + B_3 
+ \Delta_1 \Delta_2 + \Delta_2 \Delta_3 + \Delta_3 \Delta_1, \\
W_1 &= W_0 + \Delta_1 \Delta_2 \Delta_3 .
\end{align*}

Let $\gL_0 = \cP(X_7, W_0)$ and $\gL_1 = \cP(X_7, W_1)$ be
the corresponding Jacobian algebras and
denote by $\ch K$ the characteristic of the ground field $K$.

\begin{thm}
Let $K$ be a field.
\begin{enumerate}
\renewcommand{\theenumi}{\alph{enumi}}
\item
If $\ch K \ne 2$, then the algebras $\gL_0$ and $\gL_1$ are finite-dimensional
over $K$.

\item
If $\ch K = 2$, then $\gL_0$ is infinite-dimensional while $\gL_1$ is
finite-dimensional.
\end{enumerate}
\end{thm}

This theorem is a combination of the statements of
Proposition~\ref{p:W0W1dim}, Proposition~\ref{p:W0W1:Kne2}
concerning the finite-dimensionality of the algebras in question
and~Proposition~\ref{p:W0:Keq2} dealing with the algebra $\gL_0$
over a field of characteristic $2$.

We deduce the following corollary, confirming the conjecture of~\cite{GLS16}
when the ground field has characteristic $2$.

\begin{cor*}
If $\ch K = 2$, then $W_0$ and $W_1$ are two non-degenerate potentials on
$X_7$ that are not right equivalent.
\end{cor*}

Let us describe some cluster-theoretic properties of the quiver $X_7$
implied by these results.
Firstly,
the existence of a non-degenerate potential with infinite-dimensional
Jacobian algebra allows to invoke the results of~\cite{BDP14} and
deduce that there are no reddening mutation sequences for the quiver
$X_7$, see Proposition~\ref{p:X7:reddening}.
Previously, a combinatorial proof that there are no
maximal green sequences for $X_7$ was given by Seven~\cite{Seven14}.

Another consequence is that
for any non-degenerate potential $W$ on $X_7$ whose Jacobian algebra
is finite-dimensional, the exchange graph of the cluster-tilting objects
in the cluster category of $(X_7,W)$ is not connected,
see Proposition~\ref{p:X7:exchange}.

We note that in her Ph.D. thesis~\cite{Shkerat22},
Abeer Shkerat considered the potential $W_0$ and has shown using other
methods that when $\ch K \ne 2$,
the algebra $\gL_0$ is finite-dimensional and symmetric.

The above results indicate that the quiver $X_7$ shares many
properties with the quivers arising from triangulations,
in the sense of~\cite{FST08},
of closed oriented surfaces of positive genus with one puncture.
Among these properties are
the existence of at least two non-degenerate potentials whose
Jacobian algebras are pairwise not isomorphic,
the existence of a non-degenerate potential whose Jacobian
algebra is symmetric~\cite{Ladkani12} as well as another 
one whose Jacobian algebra is infinite-dimensional,
the non-existence of reddening sequences and the non-connectivity
of the exchange graph, see~\cite{Ladkani13} for details.

However,
the quiver $X_7$ reveals also a new phenomenon that is not present in the
other quivers of finite mutation type, even not those arising from
triangulations of closed surfaces with one puncture.
To the best of our knowledge, the quiver with potential $(X_7,W_0)$
seems to be the first instance of a quiver with a
non-degenerate potential whose Jacobian algebra
behaves differently
according to the characteristic of the ground field.
Previously, there were known
examples of quivers possessing several non-degenerate potentials, some of
them with finite-dimensional Jacobian algebras and some with
infinite-dimensional ones~\cite{GLS16,Ladkani13},
but this property did not depend on the ground field.

In the remaining part of the introduction, let us outline the
structure of the paper.
Due to the inherent symmetries of the quiver $X_7$, when considering
the non-degeneracy of a potential on it, we need to consider three kinds
of mutations; one at a side vertex $k \neq 0$; another at the central
vertex $0$ leading to the other quiver $X'_7$ in the mutation class
(shown in Figure~\ref{fig:X7});
and finally at a side vertex of $X'_7$.
Instead of giving ad-hoc calculations, we prefer to take on a more
systematic approach, which we develop in Section~\ref{sec:QP:mut}.

Observe that when forming the mutation at a vertex $k$ of a quiver $Q$
without loops and $2$-cycles, we first replace any path $\alpha \beta$
(where $\alpha$ is an arrow ending at $k$ and $\beta$ is an arrow starting
at $k$)
by an arrow $[\alpha \beta]$ with the same endpoints, then reverse the
arrows $\alpha$, $\beta$ to form new arrows $\alpha^*$, $\beta^*$
and finally remove a maximal set of $2$-cycles which may be created
during this process. This last step corresponds to a choice of
a matching with maximal possible cardinality between 
arrows $\gamma$ of $Q$ and corresponding
pairs of arrows $(\alpha_\gamma,\beta_\gamma)$ (ending and starting at $k$,
respectively)
such that $\alpha_\gamma \beta_\gamma \gamma$ are $3$-cycles in $Q$
(so that the $2$-cycles we remove while performing the mutation
are exactly $[\alpha_\gamma \beta_\gamma] \gamma$).
This leads us to introduce and study the notion of~\emph{matching}
(see Definitions~\ref{def:matching} and~\ref{def:matching:max}).

Let $\Gamma'$ be the subset of arrows in a maximal matching and $C'$ be
the subset of the corresponding pairs.
If $S$ a potential on $Q$ such that no term starts at $k$ and none of the
paths $\alpha \beta$ for $(\alpha, \beta) \in C'$ appears in $S$
(we call such potentials \emph{admissible},
cf.\ Definition~\ref{def:pot:admiss}), we can
form a ``dual'' potential $S^*$ on the mutation $\mu_k(Q)$ using a very
simple recipe; replace
each path $\alpha \beta$ (necessarily not in $C'$) occurring in $S$ by
the arrow $[\alpha \beta]$ and
each occurrence of an arrow $\gamma \in \Gamma'$ by the signed path
$- \beta^*_\gamma \alpha^*_\gamma$.
Our main computational result is expressed in the next theorem.
For a more elaborate statement, see Theorem~\ref{t:QPmut}.
\begin{thm}
In the situation described in the preceding paragraph,
\[
\mu_k(Q, \sum_{\gamma \in \Gamma'} \alpha_\gamma \beta_\gamma \gamma
+ S)
=
\bigl(\mu_k(Q), \sum_{(\alpha,\beta) \not \in C'}
[\alpha \beta] \beta^* \alpha^* + S^* \bigr).
\]
\end{thm}

In Section~\ref{sec:X7:mut} we deal with particular mutations.
We introduce the family of quivers $Q_n$, which are obtained by
gluing $n$ copies of the non-acyclic quiver in the mutation class
of the extended Dynkin quiver $\wt{A}_2$ at a common central
vertex, and their mutations $Q'_n$ at the central vertex.
The quivers $X_7$ and $X'_7$ are members of these families
(for the value $n=3$).
In Section~\ref{ssec:mut:side} we consider mutations at side vertices of
such gluing and apply this to mutations of potentials 
at side vertices on $Q_n$
in Section~\ref{ssec:Qn:mut:side},
while in Section~\ref{ssec:Qn:mut:center} we consider mutation at the
central vertex of $Q_n$.
Mutations at side vertices of the quiver $X'_7$ are dealt with in
Section~\ref{ssec:X7:side}, where we show that
if $W$ is a potential on $X_7$ of the form
in~\eqref{e:X7:W} and $k$ is a side vertex,
then there exists a potential $\wt{W}$ on $X_7$,
not necessarily of the same form, such that $\mu_k \mu_0(X_7, W)$
is right equivalent to $\mu_0(X_7, \wt{W})$
(Proposition~\ref{p:X7mut2:P} and Remark~\ref{rem:X7mutk:P}).
The crucial observation that $W$ and $\wt{W}$ are right equivalent
is proved in Section~\ref{ssec:X7:righteq}, see Lemma~\ref{l:P:BD}.
All these ingredients are combined together
in Section~\ref{ssec:X7:potential} to give a proof of
Theorem~\ref{t:X7:potential} concerning non-degenerate potentials
on $X_7$.

Section~\ref{sec:Jac:fdim} concerns the question of finite-dimensionality
of Jacobian algebras $\gL$ of potentials of the form
$
B_1 + B_2 + \dots + B_n + f(\Delta_1, \Delta_2, \dots, \Delta_n)
$
on the quivers $Q_n$, where $f \in K \llangle x_1, \dots, x_n \rrangle$
is a power series in $n$ non-commuting variables.
The question is reduced to that of the 
idempotent subalgebra $e_0 \gL e_0$ (Lemma~\ref{l:KQn0fd}),
which can be described as a quotient of $K \llangle x_1, \dots, x_n \rrangle$
by an explicit ideal $I_f$ depending on $f$ (Lemma~\ref{l:e0Le0}).
We present in detail an algorithm (Algorithm~\ref{alg:dr})
to compute the dimensions of the quotients $(I+\fm^r)/(I+\fm^{r+1})$
where $\fm$ is the maximal ideal of $K \llangle x_1, \dots, x_n \rrangle$
and $I$ is an ideal generated by finitely many elements.
We have implemented this algorithm using the \textsc{Magma} computer
algebra system~\cite{MAGMA}. By working also with lattices over
$\bZ$, we have been able to give a computer-assisted proof for the
finite-dimensionality of the algebras $\gL_0$ and $\gL_1$
over any field $K$ with $\ch K \ne 2$ and for $\gL_1$ also when
$\ch K = 2$, see~Section~\ref{ssec:computer}.

In Section~\ref{sec:Jac:infdim} we consider quotients of complete
path algebras of positively graded quivers by ideals generated by
homogeneous commutativity-relations (Proposition~\ref{p:KQI:hcomm})
and deduce the infinite-dimensionality of Jacobian algebras for
certain potentials
which are signed sums of homogeneous cycles of positive degree
(Proposition~\ref{p:QP:cycles}). We use these results to prove that the
Jacobian algebra of the potential $W_0$ is infinite-dimensional
over any ground field of characteristic $2$ (Proposition~\ref{p:W0:Keq2}).
The cluster-theoretic implications of this fact are outlined
in Sections~\ref{ssec:X7:reddening} and~\ref{ssec:X7:exchange}.

\section{A method to compute some QP mutations}
%%%%%%%%%%%%%%%%%%%%%%%%%%%%%%%%%%%%%%%%%%%%%%%
\label{sec:QP:mut}

\subsection{Preliminaries}
%%%%%%%%%%%%%%%%%%%%%%%%%%
We recall the basic notions of the theory of quivers with potentials,
following~\cite{DWZ08}.

A \emph{quiver} is a finite directed graph.
More precisely, it is a quadruple $(Q_0, Q_1, s, t)$, where $Q_0$ and $Q_1$
are finite sets (of \emph{vertices} and \emph{arrows}, respectively)
and $s,t \colon Q_1 \to Q_0$ are functions specifying, for each arrow
$\alpha \in Q_1$,
its starting vertex $s(\alpha)$ and terminating vertex $t(\alpha)$.

Let $Q$ be a quiver.
A \emph{path} $p$ is a sequence of arrows $\alpha_1 \alpha_2 \dots \alpha_n$
such that $s(\alpha_{i+1}) = t(\alpha_i)$ for all $1 \leq i < n$.
For a path $p$, denote by $s(p)$ its starting vertex $s(\alpha_1)$
and by $t(p)$ its terminating vertex $t(\alpha_n)$.
A path $p$ is a \emph{cycle} if $s(p)=t(p)$.
Any vertex $i \in Q_0$ gives rise to the \emph{lazy path} $e_i$,
which is a cycle of length zero with $s(e_i)=t(e_i)=i$.
An \emph{$m$-cycle} is a cycle of length $m$. 
A \emph{loop} is a $1$-cycle.
Two cycles $c, c'$ are \emph{rotationally equivalent} if one can
be obtained from the other by rotating the arrows, i.e.\ there 
exist paths $p,q$ such that $c=pq$ and $c'=qp$.

Let $K$ be a field. The \emph{path algebra} $KQ$ is the $K$-algebra
whose underlying vector space has as a basis the set of paths in $Q$
and the product of two paths $p, q$ 
is their concatenation $pq$ if $t(p)=s(q)$ and zero otherwise.
The \emph{complete path algebra} $\wh{KQ}$ is the completion of $KQ$
at the ideal generated by the arrows.
Thus, as a vector space, it is a direct product of copies of $K$, one
for each path in $Q$, and the multiplication is induced from
that on $KQ$. Denote by $\fm$ the two-sided ideal of $\wh{KQ}$
generated by the arrows of $Q$. The family $\{\fm^n\}_{n \geq 1}$ forms
a basis of open neighbourhoods of $0$.

Let $\wh{KQ}_{cyc} = \bigoplus_{i \in Q_0} e_i \wh{KQ} e_i$ be the closed
subspace of $\wh{KQ}$ consisting of the (possibly infinite) linear
combinations of cycles in $Q$. If $W = \sum_{c} \lambda_c c$ is
such a combination, we say that a cycle $c$ \emph{appears} in $W$ if
$\lambda_c \neq 0$.
A \emph{potential} is an element $W$ in $\wh{KQ}_{cyc}$ such that
all the cycles appearing in $W$ have length at least $2$. A potential $W$
is \emph{reduced} if all the cycles appearing in $W$ have length at least
$3$.

It is sometimes convenient to consider the image of a potential in
the \emph{zeroth continuous Hochschild homology} $\HH_0(\wh{KQ})$.
Recall that this is
the quotient of $\wh{KQ}$ by the closure of the subspace spanned by
all the commutators of pairs of elements in $\wh{KQ}$.
The map from $\wh{KQ}$ restricts to a surjective map
$\wh{KQ}_{cyc} \twoheadrightarrow \HH_0(\wh{KQ})$.
Two elements in $\wh{KQ}_{cyc}$ are \emph{cyclically equivalent}
if they have the same image in $\HH_0(\wh{KQ})$
(\cite[Definition~3.2]{DWZ08}).
Since $\HH_0(\wh{KQ})$ has a topological basis given by the equivalence
classes of cycles of $Q$ modulo rotation,
we may assume that the cycles appearing in a potential are
pairwise rotationally inequivalent.

Two potentials $W$ and $W'$ on $Q$ are \emph{right equivalent}
if there exists an automorphism $\vphi \colon \wh{KQ} \to \wh{KQ}$ with
$\vphi(e_i)=e_i$ for all $i \in Q_0$ such that $\vphi(W)$ and $W'$ are
cyclically equivalent. The map $\vphi$ is called
a \emph{right equivalence}, see~\cite[Definition~4.2]{DWZ08}.

Consider the subspace $KQ_1 = \bigoplus_{\alpha \in Q_1} K\alpha$ spanned
by the arrows. There are natural inclusion $\iota \colon KQ_1 \to \wh{KQ}$
and projection $\pi \colon \wh{KQ} \to KQ_1$.
Let $\vphi \colon \wh{KQ} \to \wh{KQ}$ be an endomorphism of $\wh{KQ}$
with $\vphi(e_i)=e_i$ for all $i \in Q_0$ and form the
linear map $\vphi^{(1)} = \pi \vphi \iota \colon KQ_1 \to KQ_1$.
Then $\vphi$ is a right equivalence if and only if $\vphi^{(1)}$ is an
isomorphism~\cite[Proposition~2.4]{DWZ08}.
If $\vphi^{(1)}$ is the identity map, $\vphi$ is called 
\emph{unitriangular}.
If $\vphi(\alpha) = \vphi^{(1)}(\alpha)$ for all $\alpha \in Q_1$,
$\vphi$ is called \emph{change of arrows},
see~\cite[Definition~2.5]{DWZ08}.
 
Let $(Q',W')$ and $(Q'',W'')$ be quivers with potentials having the same
set of vertices $Q'_0 = Q''_0$. The \emph{direct sum}
$(Q',W') \oplus (Q'',W'')$ is defined as
$(Q' \oplus Q'', W'+W'')$, where
$Q' \oplus Q''$ is the quiver whose set of vertices is $Q'_0$
and its set of arrows is the disjoint union of the sets $Q'_1$ and $Q''_1$.

For the next two paragraphs, assume that the quivers have no loops.
A quiver with potential $(Q,W)$ is \emph{trivial}
if the set of arrows of $Q$ consists of $2n$ distinct arrows
$\alpha_1, \dots, \alpha_n, \beta_1, \dots, \beta_n$ such that
each path $\alpha_k \beta_k$ is a $2$-cycle
and $W$ is cyclically equivalent to
$\vphi(\alpha_1 \beta_1 + \dots + \alpha_n \beta_n)$
for some right equivalence $\vphi$ which is a change of arrows, see~\cite[Proposition~4.4]{DWZ08}.

By the \emph{Splitting Theorem}~\cite[Theorem~4.6]{DWZ08},
any quiver with potential $(Q,W)$ is right equivalent to a direct sum
$(Q_{red}, W_{red}) \oplus (Q_{triv}, W_{triv})$,
where $(Q_{red}, W_{red})$ is a reduced quiver with potential
and $(Q_{triv}, W_{triv})$ is a trivial one, both unique
up to right equivalence.
The quiver with potential $(Q_{red}, W_{red})$ is called the 
\emph{reduced part} of $(Q,W)$.

\subsection{Substitutions}
%%%%%%%%%%%%%%%%%%%%%%%%%%

Let $Q$ be a quiver and $K$ be a field.
Since any endomorphism of $\wh{KQ}$ fixing the lazy paths $e_i$
is determined by its values on the arrows of $Q$, we may view
it as a substitution operation, replacing each occurrence of an arrow by
the corresponding value. In this section we take on this
point of view and list some basic properties of substitutions.
Recall that $\fm$ is the ideal of $\wh{KQ}$ generated by the arrows.

\begin{defn}
Let $\gamma$ be an arrow of $Q$ and $T \in \fm$ an element.
\begin{enumerate}%[(a)]
\renewcommand{\theenumi}{\alph{enumi}}
\item
We say that $T$ is \emph{parallel} to $\gamma$ if
$T \in e_{s(\gamma)} \fm e_{t(\gamma)}$.

\item
We say that $T$ is \emph{anti-parallel} to $\gamma$ if
$T \in e_{t(\gamma)} \fm e_{s(\gamma)}$.
\end{enumerate}
\end{defn}

\begin{remark}
If $T$ is anti-parallel to $\gamma$ then $\gamma T$ is a potential on $Q$.
If $\delta$ is an arrow anti-parallel to $\gamma$, then $\gamma \delta$
is a $2$-cycle in $Q$.
\end{remark}

\begin{defn}
Let $\gamma_1, \gamma_2, \dots, \gamma_n$ be distinct arrows of $Q$
and let $T_1, T_2, \dots, T_n \in \fm$ be elements such that
$T_i$ is parallel to $\gamma_i$ for each $1 \leq i \leq n$.

We denote by
$\sub_{\gamma_1 \la T_1, \gamma_2 \la T_2, \dots, \gamma_n \la T_n}$
the continuous algebra endomorphism of $\wh{KQ}$ whose value on
the generators is given by
\begin{align*}
&\sub_{\gamma_1 \la T_1, \dots, \gamma_n \la T_n}(e_v) = e_v ,\\
&\sub_{\gamma_1 \la T_1, \dots, \gamma_n \la T_n}(\alpha) =
\begin{cases}
T_i & \text{if $\alpha=\gamma_i$ for some $1 \leq i \leq n$,} \\
\alpha & \text{otherwise.}
\end{cases}
\end{align*}
for $v \in Q_0$ and $\alpha \in Q_1$.
\end{defn}

In the next lemmas we record some basic properties of the endomorphisms
$\sub_{\gamma \la T}$.

\begin{defn}
Let $T \in \fm$. We say that an arrow $\gamma$ \emph{appears in $T$}
if it appears in at least one of its terms, i.e.\
if $T = \sum_{p} \lambda_p p$ then
$\lambda_p \neq 0$ for some path $p$ containing $\gamma$.
\end{defn}

The following lemma is straightforward from the definitions.
\begin{lemma}
Let $\gamma$ be an arrow of $Q$ and let $T \in \fm$ be an element
parallel to $\gamma$.
Let $S \in \wh{KQ}$ and consider $S' = \sub_{\gamma \la T}(S)$.
\begin{enumerate}
\renewcommand{\theenumi}{\alph{enumi}}
\item
If $\gamma$ does not appear in $S$, then $S'=S$.

\item
If $\gamma$ does not appear in $T$, then it does not appear in $S'$.

\item
If an arrow of $Q$ does not appear in $S$ and in $T$,
then it does not appear in $S'$.
\end{enumerate}
\end{lemma}

\begin{lemma} \label{l:sub:prod}
Let $\gamma, \gamma'$ be distinct arrows of $Q$
and let $T, T' \in \fm$. Assume that:
\begin{enumerate}%[label=\emph{i}]
\renewcommand{\theenumi}{\roman{enumi}}
\item
$T$ is parallel to $\gamma$ and $T'$ parallel to $\gamma'$,

\item
$\gamma$ does not appear in $T'$ and $\gamma'$ does not appear in $T$.
\end{enumerate}

Then the endomorphisms $\sub_{\gamma \la T}$ and $\sub_{\gamma' \la T'}$
commute with each other, and their composition equals $\sub_{\gamma \la T, \gamma' \la T'}$.
\end{lemma}
\begin{proof}
It suffices to consider the action of the compositions on the arrows of $Q$.
It is clear that any arrow $\alpha \neq \gamma, \gamma'$ is mapped to itself.
Now, 
\begin{align*}
\sub_{\gamma \la T} (\sub_{\gamma' \la T'}(\gamma))
= \sub_{\gamma \la T}(\gamma) = T
&,&
\sub_{\gamma' \la T'} (\sub_{\gamma \la T}(\gamma))
= \sub_{\gamma' \la T'}(T) = T, 
\end{align*}
where the right equality holds since $\gamma'$ does not appear in $T$.
Similarly,
\begin{align*}
\sub_{\gamma' \la T'} (\sub_{\gamma \la T}(\gamma'))
= \sub_{\gamma' \la T'}(\gamma') = T'
&,&
\sub_{\gamma \la T} (\sub_{\gamma' \la T'}(\gamma'))
= \sub_{\gamma \la T}(T') = T' 
\end{align*}
since $\gamma$ does not appear in $T'$, and the claim follows.
\end{proof}

We deduce by induction the following statement.
\begin{cor} \label{c:sub:prod}
Let $\gamma_1, \dots,\gamma_n$ be distinct arrows of $Q$ and let
$T_1, \dots, T_n \in \fm$. Assume that for each $1 \leq i \leq $n,
\begin{enumerate}
\renewcommand{\theenumi}{\roman{enumi}}
\item
$T_i$ is parallel to $\gamma_i$,

\item
$\gamma_i$ does not appear in any of the elements $T_j$ for $j \neq i$.
\end{enumerate}
Then the endomorphisms $\sub_{\gamma_i \la T_i}$ 
commute with each other, and their composition
in any order equals $\sub_{\gamma_1 \la T_1, \dots, \gamma_n \la T_n}$.
\end{cor}

We now consider two situations where the substitution endomorphism is
actually an automorphism (and hence a right equivalence).
The first situation is dealt with in the next lemma.

\begin{lemma} \label{l:sub:gTauto}
Let $\gamma$ be an arrow of $Q$ and let $T, T' \in \fm$ be elements
parallel to $\gamma$. Then:
\begin{enumerate}
\renewcommand{\theenumi}{\alph{enumi}}
\item \label{it:gTT'}
If $\gamma$ does not appear in $T'$, then
\[
\sub_{\gamma \la \gamma + T} \circ \sub_{\gamma \la \gamma + T'}
= \sub_{\gamma \la \gamma + T + T'} .
\]

\item \label{it:gTauto}
If $\gamma$ does not appear in $T$, then
$\sub_{\gamma \la \gamma + T}$ is an automorphism of $\wh{KQ}$.
\end{enumerate}
\end{lemma}
\begin{proof}
For part~\eqref{it:gTT'},
it is enough to consider the action of the composition on the arrows of $Q$.
Any arrow $\alpha \neq \gamma$ is mapped to itself, whereas
\[
\sub_{\gamma \la \gamma + T} (\sub_{\gamma \la \gamma + T'}(\gamma)) = 
\sub_{\gamma \la \gamma + T}(\gamma + T') = (\gamma + T) + T'
\]
since $\gamma$ does not appear in $T'$.

For part~\eqref{it:gTauto}, observe that by~\eqref{it:gTT'} we have
\[
\sub_{\gamma \la \gamma - T} \circ \sub_{\gamma \la \gamma + T} =
\sub_{\gamma \la \gamma + T} \circ \sub_{\gamma \la \gamma - T} =
\id_{\wh{KQ}} .
\]
\end{proof}

\begin{remark}
Observe that the proof of Lemma~\ref{l:sub:gTauto} does not rely on
\cite[Proposition~2.4]{DWZ08}, so if we assume in addition that $T \in KQ$,
then the same argument shows that $\sub_{\gamma \la \gamma+T}$, now
viewed as an endomorphism of the path algebra $KQ$, is an automorphism of $KQ$.
\end{remark}

The second situation is dealt with in the next lemma,
where we do need to work within the complete setting.
\begin{lemma} \label{l:sub:gTauto:m2}
Let $T \in \fm$ such that
$T - \sum_{\alpha \in Q_1} \lambda_\alpha \alpha \in \fm^2$
for some scalars $\lambda_\alpha \in K$.
Then $\sub_{\gamma \la T}$ is an automorphism of $\wh{KQ}$
if and only if $\lambda_\gamma \neq 0$.

In particular, $\sub_{\gamma \la \gamma+T}$ is a
(unitriangular) automorphism for any $T \in \fm^2$.
\end{lemma}
\begin{proof}
The set
$(Q_1 \setminus \{\gamma\}) \cup \left\{ \sum \lambda_\alpha \alpha \right\}$
is a basis of the subspace spanned by the arrows if and only if
$\lambda_\gamma \ne 0$. Now use~\cite[Proposition~2.4]{DWZ08}.
\end{proof}

The next lemma relates the value of the endomorphism
$\sub_{\gamma \la \gamma + T}$ applied on a potential to that of
$\sub_{\gamma \la T}$.

\begin{lemma} \label{l:sub:gTcyc}
Let $\gamma$ be an arrow of $Q$, let $T \in \fm$ be an element
parallel to $\gamma$ and let $W$ be a potential on $Q$. Then
there exists $T' \in \fm$ anti-parallel to $\gamma$ such that
the potentials $\sub_{\gamma \la \gamma+T}(W)$ and
$\gamma T' + \sub_{\gamma \la T}(W)$ are cyclically equivalent.
Moreover,
\begin{enumerate}%[(a)]
\renewcommand{\theenumi}{\alph{enumi}}
\item
If $T$ and $W$ belong to $KQ$, then so does $T'$.

\item
If $\gamma$ does not appear in $T$, then it does not appear in
$\sub_{\gamma \la T}(W)$.

\item
If $\delta$ is an arrow of $Q$
that does not appear in $T$ and in $W$, then 
it does not appear in $T'$ and in $\sub_{\gamma \la T}(W)$.
\end{enumerate}
\end{lemma}
\begin{proof}
By linearity and continuity, it is enough to consider the case where
$W$ is a cycle of length at least $2$.
If $\gamma$ does not appear in $W$, then
\[
\sub_{\gamma \la \gamma + T}(W) = W = \sub_{\gamma \la T}(W)
\]
and the claim follows with $T'=0$.
Otherwise, any cycle containing $\gamma$ is cyclically equivalent
to a cycle of the form
$\gamma p_1 \gamma p_2 \dots \gamma p_n$ for some
$n \geq 1$, where each of the paths $p_i$ is anti-parallel to $\gamma$
and does not contain $\gamma$.
Evaluating the substitution, we get
\[
\sub_{\gamma \la \gamma + T}(\gamma p_1 \gamma p_2 \dots \gamma p_n) =
(\gamma + T) p_1 (\gamma + T) p_2 \dots (\gamma + T) p_n .
\]

By expanding the product, the right hand side can be written as a
sum of $2^n$ terms parameterized by $0 \leq e < 2^n$ in such a way
that the term corresponding to $0$ is
\[
T p_1 T p_2 \dots T p_n =
\sub_{\gamma \la T}(\gamma p_1 \gamma p_2 \dots \gamma p_n)
\]
and all the other terms contain at least one occurrence of $\gamma$.
It follows that for any $e \neq 0$, there exists an element $T'_e$
anti-parallel to $\gamma$ such that the corresponding term is cyclically
equivalent to $\gamma T'_e$. Thus
$\sub_{\gamma \la \gamma + T}(\gamma p_1 \gamma p_2 \dots \gamma p_n)$
is cyclically equivalent to
$\gamma T' + \sub_{\gamma \la T}(\gamma p_1 \gamma p_2 \dots \gamma p_n)$
for $T' = \sum_{e \neq 0} T'_e$.

The other claims are clear.
\end{proof}

\subsection{Some right equivalences involving $2$-cycles}
%%%%%%%%%%%%%%%%%%%%%%%%%%%%%%%%%%%%%%%%%%%%%%%%%%%%%%%%%

Right equivalences of potentials involving $2$-cycles
play an important role in the theory of quivers with potentials.
The proof of the Splitting
Theorem~\cite[Theorem~4.6]{DWZ08} relies on
Lemma~4.7 of~\cite{DWZ08}, which states that
if $\gamma_i \delta_i$ are disjoint $2$-cycles,
$T_i \in \fm^2$ is parallel to $\gamma_i$ and
$U_i \in \fm^2$ is parallel to $\delta_i$, then for any
potential $S \in \fm^3$ in which none of the arrows $\gamma_i$, $\delta_i$
appear, the potential
\begin{equation} \label{e:2cyc:potU}
\gamma_1 \delta_1 + \dots + \gamma_n \delta_n +
\gamma_1 U_1 + \dots + \gamma_n U_n +
T_1 \delta_1 + \dots + T_n \delta_n +
S
\end{equation}
is right equivalent to
\begin{equation} \label{e:2cyc:S}
\gamma_1 \delta_1 + \dots + \gamma_n \delta_n + S'
\end{equation}
for some $S' \in \fm^3$. The proof of this result is based on a limiting
process, as described in~\cite[Lemma~4.8]{DWZ08}, see also the detailed
discussion in~\cite[p.\ 807]{Labardini09}.

In this section we present, in a somewhat restricted setup,
an analogous result whose proof does not
require a limiting process, giving a new perspective to
Remark~5 of~\cite{Labardini09}.
In our version, the terms $U_i$ are required to vanish and none of the arrows
$\gamma_i, \delta_i$ is allowed to appear in any of the terms $T_j$,
however the potential $S$ may contain the arrows $\gamma_i$
(but not $\delta_i$).
An advantage of our approach is that it yields an explicit form for the
potential $S'$ appearing in~\eqref{e:2cyc:S} in terms of an appropriate
substitution, see Proposition~\ref{p:2cyc} below.

We start by considering a situation involving only one $2$-cycle.
\begin{lemma} \label{l:2cyc:1}
Let $\gamma$ and $\delta$ be anti-parallel arrows of $Q$,
let $T \in \fm$ and let $S$ be a potential on $Q$.
Assume that:
\begin{enumerate}
\renewcommand{\theenumi}{\roman{enumi}}
\item
$T$ is parallel to $\gamma$,

\item
$\gamma$, $\delta$ do not appear in $T$,

\item
$\delta$ does not appear in $S$.
\end{enumerate}

Then:
\begin{enumerate}
\renewcommand{\theenumi}{\alph{enumi}}
\item
The potentials $\gamma \delta + T \delta + S$ and 
$\gamma \delta + \sub_{\gamma \la -T}(S)$ are right equivalent.

\item
$\gamma$ and $\delta$ do not appear in $\sub_{\gamma \la -T}(S)$.
\end{enumerate}
\end{lemma}
\begin{proof}
Write $W = \gamma \delta + T \delta + S$.
Since $\gamma$ does not appear in $T$, the endomorphism
$\sub_{\gamma \la \gamma - T}$ is an automorphism of $\wh{KQ}$
by Lemma~\ref{l:sub:gTauto}, hence
$W$ and $\sub_{\gamma \la \gamma -T}(W)$ are right equivalent.
Using again our assumption that $\gamma$ does not appear in $T$,
we get
\[
\sub_{\gamma \la \gamma -T}(W) = 
(\gamma - T) \delta + T \delta + \sub_{\gamma \la \gamma - T}(S)
= \gamma \delta + \sub_{\gamma \la \gamma - T}(S) .
\]

By Lemma~\ref{l:sub:gTcyc} applied for the potential $S$, we get
that $\sub_{\gamma \la \gamma - T}(S)$ is cyclically equivalent to
a potential of the form $\gamma T' + \sub_{\gamma \la -T}(S)$,
where $T'$ is an element anti-parallel to $\gamma$, hence
parallel to $\delta$.
It follows that $\sub_{\gamma \la \gamma -T}(W)$ is cyclically
equivalent to the potential
\[
\gamma \delta + \gamma T' + \sub_{\gamma \la -T}(S).
\]

Moreover, our assumptions on $\delta$ imply that it does not appear in $T'$
nor in $\sub_{\gamma \la -T}(S)$, so by applying the endomorphism
$\sub_{\delta \la \delta - T'}$, we get
\[
\begin{split}
\sub_{\delta \la \delta - T'}(\gamma \delta + \gamma T' + 
\sub_{\gamma \la -T}(S)) &=
\gamma (\delta - T') + \gamma T' +
\sub_{\delta \la \delta - T'}(\sub_{\gamma \la -T}(S)) \\
&= \gamma \delta + \sub_{\gamma \la -T}(S)
\end{split}
\]
By Lemma~\ref{l:sub:gTauto},
$\sub_{\delta \la \delta - T'}$ is an automorphism of $\wh{KQ}$,
so $\gamma \delta + \sub_{\gamma \la -T}(S)$ is right equivalent
to $\sub_{\gamma \la \gamma -T}(W)$ and hence to $W$.
\end{proof}

\begin{prop} \label{p:2cyc}
Let $\gamma_1, \dots, \gamma_n, \delta_1, \dots, \delta_n$
be distinct arrows of $Q$, let $T_1, \dots, T_n \in \fm$
and let $S$ be a potential on $Q$.
Assume that:
\begin{enumerate}
\renewcommand{\theenumi}{\roman{enumi}}
\item
For each $1 \leq i \leq n$, 
$T_i$ is parallel to $\gamma_i$ and
$\delta_i$ is anti-parallel to $\gamma_i$.

\item
The arrows $\gamma_1, \dots, \gamma_n, \delta_1, \dots, \delta_n$
do not appear in any of the elements $T_1, \dots, T_n$.

\item
None of the arrows $\delta_1, \dots, \delta_n$ appears in $S$.
\end{enumerate}
Then:
\begin{enumerate}
\renewcommand{\theenumi}{\alph{enumi}}
\item
The potentials
\begin{equation} \label{e:2cyc:potT}
\gamma_1 \delta_1 + \gamma_2 \delta_2 + \dots + \gamma_n \delta_n
+ T_1 \delta_1 + \dots + T_n \delta_n + S
\end{equation}
and
\[
\gamma_1 \delta_1 + \gamma_2 \delta_2 + \dots + \gamma_n \delta_n + 
\sub_{\gamma_1 \la -T_1, \dots, \gamma_n \la -T_n}(S)
\]
are right equivalent.

\item
None of the arrows $\gamma_1, \dots, \gamma_n, \delta_1, \dots, \delta_n$
appears in
$\sub_{\gamma_1 \la -T_1, \dots, \gamma_n \la -T_n}(S)$.
\end{enumerate}
\end{prop}
\begin{proof}
We prove the claim by induction on $n$, the case $n=1$ being
Lemma~\ref{l:2cyc:1}. We assume the claim for $n-1$ and prove it
for $n$. Write
\[
W = \gamma_1 \delta_1 + \gamma_2 \delta_2 + \dots + \gamma_n \delta_n
+ T_1 \delta_1 + \dots + T_n \delta_n + S
= \gamma_n \delta_n + T_n \delta_n + S'
\]
where
$S' = \sum_{i=1}^{n-1} \gamma_i \delta_i + \sum_{i=1}^{n-1} T_i \delta_i + S$.

By assumption, the arrows $\gamma_n$, $\delta_n$ do not appear in $T_n$ and
the arrow
$\delta_n$ does not appear in any of the terms of $T_1, \dots T_{n-1}, S$,
hence it does not appear in $S'$. Applying Lemma~\ref{l:2cyc:1}
for $\gamma_n$, $\delta_n$, $T_n$ and $S'$, we get that $W$ is right
equivalent to the potential
\[
\gamma_n \delta_n + \sub_{\gamma_n \la -T_n}(S').
\]

Since $\gamma_n$ does not appear in any of the terms $T_1, \dots, T_{n-1}$
occurring in $S'$, we get
\begin{align}
\nonumber
\gamma_n \delta_n + \sub_{\gamma_n \la -T_n}(S') &=
\gamma_n \delta_n + \sum_{i=1}^{n-1} \gamma_i \delta_i
+ \sum_{i=1}^{n-1} T_i \delta_i + \sub_{\gamma_n \la -T_n}(S) \\
\label{e:cdS''}
&= \sum_{i=1}^{n-1} \gamma_i \delta_i
+ \sum_{i=1}^{n-1} T_i \delta_i + S''
\end{align}
where $S'' = \gamma_n \delta_n + \sub_{\gamma_n \la -T_n}(S)$.

By assumption, none of the arrows $\delta_1, \dots, \delta_{n-1}$ appear
in $T_n$ or in $S$, hence they do not appear in $S''$, and by the induction
hypothesis the potential in~\eqref{e:cdS''} is right equivalent to
\[
\begin{split}
&\gamma_1 \delta_1 + \dots + \gamma_{n-1} \delta_{n-1} +
\sub_{\gamma_1 \la -T_1, \dots, \gamma_{n-1} \la -T_{n-1}}(S'')
= \\
&\gamma_1 \delta_1 + \dots + \gamma_{n-1} \delta_{n-1} + \gamma_n \delta_n + 
\sub_{\gamma_1 \la -T_1, \dots, \gamma_{n-1} \la -T_{n-1}}
(\sub_{\gamma_n \la -T_n}(S)) = \\
&\gamma_1 \delta_1 + \dots + \gamma_{n-1} \delta_{n-1} + \gamma_n \delta_n + 
\sub_{\gamma_1 \la -T_1, \dots, \gamma_{n-1} \la -T_{n-1},
\gamma_n \la -T_n}(S)
\end{split}
\]
where the last equality is a consequence of Corollary~\ref{c:sub:prod}.
\end{proof}

\subsection{Pre-mutation}
%%%%%%%%%%%%%%%%%%%%%%%%%

In this section we recall the notions of pre-mutation of quivers and
quivers with potentials and establish terminology and notations that
will be used in the sequel.

Let $Q$ be a quiver and fix a vertex $k$ of $Q$.
Assume that
$Q$ has \emph{no loops at $k$} and \emph{no $2$-cycles passing through $k$}.
For a pair of vertices $i, j$, denote by
$\Gamma_{ij}$ the set of arrows from $j$ to $i$ and let
\[
C_{ij} = \left\{ (\alpha, \beta) \,:\, i \xrightarrow{\alpha} k,
k \xrightarrow{\beta} j \right\}.
\]
One could think of $C_{ij}$ as the set of paths of length $2$ from $i$
to $j$ passing through $k$. These sets obviously depend on the vertex
$k$, but since it is fixed throughout, we omit it in order to simplify
the notation. Write $C = \bigcup_{i,j \ne k} C_{ij}$.

\begin{defn} \label{def:premut}
The \emph{pre-mutation} $\wt{\mu}_k(Q)$ of the quiver $Q$ at the vertex $k$
is obtained from $Q$ by performing the following steps:
\begin{enumerate}
\renewcommand{\labelenumi}{\theenumi.}
\item
For any $i,j \neq k$ and each pair $(\alpha,\beta) \in C_{ij}$,
add an arrow $i \xrightarrow{[\alpha \beta]} j$;

\item
Replace each arrow $i \xrightarrow{\alpha} k$ by an arrow
$k \xrightarrow{\alpha^*} i$ in the opposite direction;
similarly, replace each arrow $k \xrightarrow{\beta} j$ by an arrow
$j \xrightarrow{\beta^*} k$.
\end{enumerate}

When $Q$ has no loops and $2$-cycles,
the \emph{mutation} $\mu_k(Q)$ is obtained from $\wt{\mu}_k(Q)$
by removing a maximal set of disjoint $2$-cycles.
\end{defn}

\begin{notat}
Let $W$ be a potential on $Q$. By rotating its terms we may and will assume
that none of its terms starts at $k$.
We denote by $[W]$ the potential on $\wt{\mu}_k(Q)$ obtained from $W$ by
replacing each occurrence of a path $\alpha \beta$ with
$(\alpha,\beta) \in C$ by the corresponding arrow $[\alpha \beta]$.
\end{notat}

\begin{defn}[\protect{\cite[Definition~5.5]{DWZ08}},
\protect{\cite[Definition~11]{Labardini09}}]
The \emph{pre-mutation} $\wt{\mu}_k(Q,W)$ of the quiver with potential
$(Q,W)$ at the vertex $k$ is defined as
$(\wt{\mu}_k(Q), \wt{W})$, where
\[
\wt{W} = \sum_{(\alpha,\beta) \in C} [\alpha \beta] \beta^* \alpha^* + [W] .
\]
The \emph{mutation} $\mu_k(Q,W)$ is the reduced part of $\wt{\mu}_k(Q,W)$.
\end{defn}

\subsection{Matchings and their duals}
%%%%%%%%%%%%%%%%%%%%%%%%%%%%%%%%%%%%%%
\label{ssec:match}

The process of removal of $2$-cycles when forming the mutation
involves some choice. Let us better understand this choice;
observe that any pair
$(\alpha,\beta) \in C_{ij}$ together with an arrow $\gamma \in \Gamma_{ij}$
give rise to a $2$-cycle $[\alpha \beta] \gamma$ in the pre-mutation
$\wt{\mu}_k(Q)$, and any $2$-cycle arises in this way. Thus, a
choice of a maximal set of $2$-cycles corresponds to a choice of
bijections $C'_{ij} \xrightarrow{\sim} \Gamma'_{ij}$ between
subsets $C'_{ij} \subseteq C_{ij}$ and $\Gamma'_{ij} \subseteq \Gamma_{ij}$
having the maximal possible cardinality
$|C'_{ij}|=|\Gamma'_{ij}| = \min\{|C_{ij}|, |\Gamma_{ij}|\}$.

This motivates us to introduce the notion of \emph{matching}, to be discussed in
this section.
We keep the assumptions on the quiver $Q$ and the vertex $k$;
no loops at $k$ and no $2$-cycles passing through $k$.

\begin{defn} \label{def:matching}
A \emph{matching} $\rho$ consists of the following data,
for each pair of vertices $i,j \neq k$:
\begin{itemize}
\item
a subset $C'_{ij}$ of $C_{ij}$ and
a subset $\Gamma'_{ij}$ of $\Gamma_{ij}$;

\item
a bijection $C'_{ij} \xrightarrow{\sim} \Gamma'_{ij}$.
\end{itemize}

These bijections can be merged to a bijection $C' \xrightarrow{\sim} \Gamma'$
between the subset $C' = \bigcup_{i,j \ne k} C'_{ij}$ of the disjoint union
$C = \bigcup_{i,j \ne k} C_{ij}$ and the subset
$\Gamma' = \bigcup_{i,j \ne k} \Gamma'_{ij}$ of the disjoint union
$\Gamma = \bigcup_{i,j \ne k} \Gamma_{ij}$, and we denote
$\rho = (C', \Gamma', C' \xrightarrow{\sim} \Gamma')$.
\end{defn}

Given a matching, we denote,
for a pair $(\alpha, \beta) \in C'$, the corresponding arrow
in $\Gamma'$ by $\gamma_{\alpha,\beta}$. Similarly, for an arrow
$\gamma \in \Gamma'$, denote the corresponding pair in $C'$ by
$(\alpha_\gamma, \beta_\gamma)$. Thus any pair $(\alpha,\beta) \in C'$
gives rise to the $3$-cycle $\alpha \beta \gamma_{\alpha,\beta}$
in $Q$, any arrow $\gamma \in \Gamma'$ belongs to the $3$-cycle
$\alpha_\gamma \beta_\gamma \gamma$ and these two sets of $3$-cycles
coincide.

\begin{defn}
The potential \emph{encoding} a matching $\rho$ is the following sum of
$3$-cycles
\[
W_\rho = \sum_{(\alpha,\beta) \in C'} \alpha \beta \gamma_{\alpha,\beta}
= \sum_{\gamma \in \Gamma'} \alpha_\gamma \beta_\gamma \gamma .
\]
\end{defn}

\begin{remark}
The matching $\rho$ can be recovered from the potential $W_\rho$ encoding
it in the following way;
if $\sum_{r=1}^n \alpha_r \beta_r \gamma_r$
is a sum of $3$-cycles in $Q$ passing through the vertex $k$ with 
$(\alpha_r, \beta_r) \in C$ and $\gamma_r \in \Gamma$
such that the arrows $\gamma_1, \dots, \gamma_n$ are pairwise distinct
and the elements $(\alpha_1,\beta_1), \dots, (\alpha_n, \beta_n)$ are
pairwise distinct,
then 
$C' = \left\{ (\alpha_1,\beta_1), \dots, (\alpha_n, \beta_n) \right\}$
and $\Gamma' = \{ \gamma_1, \dots, \gamma_n\}$, with the bijection
sending each $(\alpha_r,\beta_r)$ to the corresponding $\gamma_r$
for $1 \leq r \leq n$.
\end{remark}

\begin{example} \label{ex:matching}
As a running example to illustrate the material of this section,
let $Q$ be any quiver containing the leftmost quiver
in Figure~\ref{fig:Qmatch}
as full subquiver, so that in particular
no arrows other than $\alpha_1, \alpha_2, \beta_1, \beta_2$
start or end at the vertex~$k$.
Mutations of potentials on such quivers arising from triangulations of
surfaces were considered in~\cite{GLM20,Labardini09}
and~\cite[\S9.4]{Ladkani17}.

The sets $C$ and $\Gamma$ are given by
$C = \{(\alpha_1, \beta_1), (\alpha_1, \beta_2), (\alpha_2, \beta_1),
(\alpha_2, \beta_2)\}$, 
$\Gamma = \{ \gamma_1, \gamma_2 \}$,
so each arrow $\gamma_t$ can be matched to the pair $(\alpha_t, \beta_t)$.
Hence there are four possible matchings $\rho_{\varnothing}, \rho_1,
\rho_2, \rho_{12}$ which are encoded by the potentials 
\begin{align*}
W_{\rho_{\varnothing}} = 0 &,&
W_{\rho_1} = \alpha_1 \beta_1 \gamma_1 &,&
W_{\rho_2} =  \alpha_2 \beta_2 \gamma_2 &,&
W_{\rho_{12}} = \alpha_1 \beta_1 \gamma_1 + \alpha_2 \beta_2 \gamma_2.
\end{align*}
\end{example}

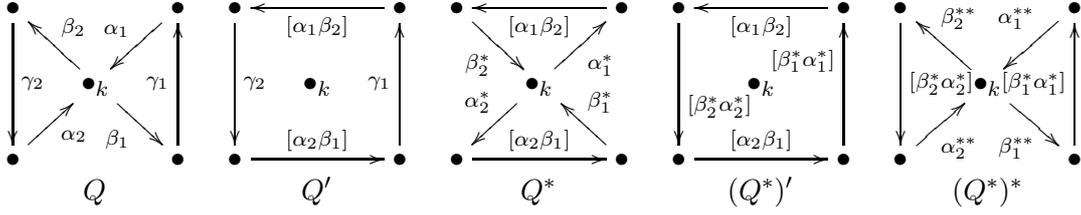
\begin{figure}
\[
\begin{array}{ccccc}
\xymatrix@=1.36pc{
{\bullet} \ar[dd]^{\gamma_2} & & {\bullet} \ar[dl]_{\alpha_1} \\
& {\bullet_k} \ar[dr]_{\beta_1} \ar[ul]_{\beta_2} \\
{\bullet} \ar[ur]_{\alpha_2} & & {\bullet} \ar[uu]^{\gamma_1}
}
\hspace{-0.06em} &
\xymatrix@=1.36pc{
{\bullet} \ar[dd]^{\gamma_2} && {\bullet} \ar[ll]^{[\alpha_1 \beta_2]} \\
& {\bullet_k} \\
{\bullet} \ar[rr]^{[\alpha_2 \beta_1]} && {\bullet} \ar[uu]^{\gamma_1}
}
\hspace{-0.06em} &
\xymatrix@=1.36pc{
{\bullet} \ar[dr]_{\beta^*_2} && {\bullet} \ar[ll]^{[\alpha_1 \beta_2]} \\
& {\bullet_k} \ar[ur]_{\alpha^*_1} \ar[dl]_{\alpha^*_2} \\
{\bullet} \ar[rr]^{[\alpha_2 \beta_1]} && {\bullet} \ar[ul]_{\beta^*_1}
}
\hspace{-0.06em} &
\xymatrix@=1.36pc{
{\bullet} \ar[dd]^(0.65){[\beta^*_2 \alpha^*_2]}
&& {\bullet} \ar[ll]^{[\alpha_1 \beta_2]} \\
& {\bullet_k} \\
{\bullet} \ar[rr]^{[\alpha_2 \beta_1]}
&& {\bullet} \ar[uu]^(0.65){[\beta^*_1 \alpha^*_1]}
}
\hspace{-0.06em} &
\xymatrix@R=1.36pc@C=1.5pc{
{\bullet} \ar[dd]^{[\beta^*_2 \alpha^*_2]}
&& {\bullet} \ar[dl]_(0.4){\alpha_1^{**}} \\
& {\bullet_k} \ar[dr]_(0.6){\beta_1^{**}} \ar[ul]_(0.6){\beta_2^{**}} \\
{\bullet} \ar[ur]_(0.4){\alpha_2^{**}}
&& {\bullet} \ar[uu]^{[\beta^*_1 \alpha^*_1]}
}
\\
Q & Q' & Q^* & (Q^*)' & (Q^*)^*
\end{array}
\]
\caption{Example of quivers built from a matching;
parts of the original quiver $Q$ and the quivers $Q', Q^*, (Q^*)', (Q^*)^*$
around the vertex $k$, computed with respect to the matching $\rho_{12}$ of
Example~\ref{ex:matching}.}
\label{fig:Qmatch}
\end{figure}

Given a matching $\rho$ on $Q$, we introduce two quivers built
from the data of $\rho$; the quiver $Q_\rho$ of the matching and
the dual quiver $Q^*$.

\begin{defn} \label{def:Qrho}
The \emph{quiver of the matching $\rho$}, denoted $Q_\rho$,
is the quiver whose set of vertices is that of $Q$ and its arrows
are of two kinds:
\begin{enumerate}
\renewcommand{\labelenumi}{\theenumi.}
\item
The arrows in $\Gamma'$;

\item
An arrow $i \xrightarrow{[\alpha \beta]} j$ for any $i,j \neq k$ and
$(\alpha,\beta) \in C'_{ij}$.
\end{enumerate}
\end{defn}

\begin{defn} \label{def:Qdual}
Let $\rho = (C', \Gamma', C' \xrightarrow{\sim} \Gamma')$ be a matching.
The \emph{dual quiver} $Q^*$ of $Q$ with respect to $\rho$ is the
quiver obtained from $Q$ by performing the following three steps:
\begin{enumerate}
\renewcommand{\labelenumi}{\theenumi.}
\item
For any $i,j \neq k$ and each pair $(\alpha,\beta) \in C_{ij} \setminus C'_{ij}$,
add an arrow $i \xrightarrow{[\alpha \beta]} j$;

\item
Replace each arrow $i \xrightarrow{\alpha} k$ by an arrow
$k \xrightarrow{\alpha^*} i$ in the opposite direction;
similarly, replace each arrow $k \xrightarrow{\beta} j$ by an arrow
$j \xrightarrow{\beta^*} k$;

\item
Remove all the arrows in $\Gamma'$.
\end{enumerate}
\end{defn}

\begin{remark}
The quiver of the matching and the dual quiver depend only on the subsets
$C'$ and $\Gamma'$ and not on the particular bijection
$C' \xrightarrow{\sim} \Gamma'$.
\end{remark}

\begin{lemma} \label{l:Qrho}
We have:
\begin{enumerate}
\renewcommand{\theenumi}{\alph{enumi}}
\item \label{it:Q*Qrho}
$\wt{\mu}_k(Q) = Q^* \oplus Q_\rho$

\item
$(Q_\rho, [W_\rho])$ is a trivial quiver with potential.
\end{enumerate}
\end{lemma}
\begin{proof}
Comparing Definitions~\ref{def:premut}, \ref{def:Qrho} and \ref{def:Qdual},
we see that $Q^*$ and $Q_\rho$ are subquivers of $\wt{\mu}_k(Q)$ with
disjoint sets of arrows whose union equals the set of arrows of
$\wt{\mu}_k(Q)$. This proves the first claim.
Now $[W_\rho] = \sum_{(\alpha,\beta) \in C'}
[\alpha \beta] \gamma_{\alpha,\beta}$ is a sum of pairwise disjoint
$2$-cycles in $Q_\rho$, proving the second claim.
\end{proof}

\begin{example}
There is always the \emph{empty matching} $\rho_\varnothing$ whose
subsets $C'$ and $\Gamma'$ are the empty sets.
For this matching, $Q_{\rho_\varnothing}$ is a quiver without arrows
and $Q^*=\wt{\mu}_k(Q)$.
\end{example}

Let $i,j \neq k$. Denote by $C^*_{ji}$ the set of all paths $j \to k \to i$
in $Q^*$ and by $\Gamma^*_{ji}$ the set of the arrows from $i$ to $j$
in $Q^*$. Then
$C^*_{ji} = \left\{ (\beta^*, \alpha^*) \,:\, (\alpha, \beta) \in C_{ij}
\right\}$
and $\Gamma^*_{ji}$ contains the set
$\bigl\{ [\alpha \beta] \,:\, (\alpha, \beta) \in C_{ij} \setminus C'_{ij}
\bigr\}$.
Denote by $C^*$ and $\Gamma^*$ the disjoint unions
$\bigcup_{i,j \neq k} C^*_{ji}$ and $\bigcup_{i,j \neq k} \Gamma^*_{ji}$,
respectively.

\begin{defn}
Let $\rho = (C', \Gamma', C' \xrightarrow{\sim} \Gamma')$ be a matching.
The \emph{dual matching} 
$\rho^* = ((C^*)', (\Gamma^*)', (C^*)' \xrightarrow{\sim} (\Gamma^*)')$
on $Q^*$ is defined by
\begin{align*}
(C^*)' &= \left\{ (\beta^*, \alpha^*) \,:\,
(\alpha, \beta) \in C \setminus C' \right\}\\
(\Gamma^*)' &= \left\{ [\alpha \beta] \,:\,
(\alpha, \beta) \in C \setminus C' \right\}
\end{align*}
with the bijection
sending each pair $(\beta^*, \alpha^*)$ to the arrow $[\alpha \beta]$.
\end{defn}

\begin{example}
For the matching $\rho_{12}$ of Example~\ref{ex:matching},
the parts around the vertex $k$ of the dual quiver $Q^*$ and its dual
$(Q^*)^*$ with respect to $\rho^*_{12}$ are shown in Figure~\ref{fig:Qmatch}.
\end{example}

\begin{lemma} \label{l:rho**}
The quivers $Q$ and $(Q^*)^*$ are isomorphic and under this isomorphism
the matching $\rho$ coincides with $(\rho^*)^*$.
\end{lemma}
\begin{proof}
Going over Definition~\ref{def:Qdual} twice, the quiver $(Q^*)^*$ is
obtained from $Q$ as follows:
\begin{itemize}
\item
Replace any arrow $\alpha$ ending at $k$ by $\alpha^*$ and then by
$(\alpha^*)^*$. \\
Do the same for any arrow $\beta$ starting at $k$.

\item
Add the arrows $[\alpha \beta]$ for $(\alpha, \beta) \in C \setminus C'$.

\item
Remove the arrows of $\Gamma'$.

\item
Add the arrows $[\beta^* \alpha^*]$ for
$(\alpha, \beta) \in C'$.

\item
Remove the arrows $[\alpha \beta]$ for
$(\alpha, \beta) \in C \setminus C'$.
\end{itemize}

We see that the processes of adding arrows and then removing arrows are
dual to each other and $Q \simeq (Q^*)^*$ by the isomorphism $\vphi$ sending
\begin{align*}
\alpha \mapsto (\alpha^*)^* &,&
\beta \mapsto (\beta^*)^* &,&
\gamma \mapsto [\beta^*_\gamma \alpha^*_\gamma]
\quad (\gamma \in \Gamma'),
\end{align*}
fixing all the vertices and the other arrows.

Finally, observe that $((\Gamma^*)^*)'$ is the set of arrows
$[\beta^* \alpha^*]$ with $(\alpha, \beta) \in C'$, which gets
identified with $\Gamma'$ by $\vphi$.
\end{proof}

\begin{defn}
Let $\rho$ be a matching.
An element $S$ in $\wh{KQ}$ is \emph{$\rho$-admissible} if:
\begin{enumerate}
\renewcommand{\theenumi}{\roman{enumi}}
\item
None of the terms of $S$ starts or ends at $k$,

\item
For any pair $(\alpha, \beta) \in C'$, the path $\alpha \beta$ does
not appear in any term of $S$.
\end{enumerate}
\end{defn}

\begin{remark}
The property of being $\rho$-admissible depends only on the subset $C'$
and not on $\Gamma'$ or the particular bijection.
\end{remark}

\begin{lemma} \label{l:adm:Q'}
The $\rho$-admissible elements form  a subalgebra of $(1-e_k)\wh{KQ}(1-e_k)$
which is isomorphic to the complete path algebra of the quiver $Q'$
obtained from $Q$ by:
\begin{enumerate}
\renewcommand{\labelenumi}{\theenumi.}
\item
Removing all the arrows starting or ending at $k$; and

\item
Adding an arrow $i \xrightarrow{[\alpha \beta]} j$ 
for any pair $(\alpha, \beta) \in C_{ij} \setminus C'_{ij}$.
\end{enumerate}
\end{lemma}
\begin{proof}
This is clear; the isomorphism is given by the map $S \mapsto [S]$ replacing
any occurrence of a path $\alpha \beta$ of length $2$ such that
$(\alpha, \beta) \in C \setminus C'$ by the arrow $[\alpha \beta]$ in $Q'$.
\end{proof}

\begin{example}
For the matching $\rho_{12}$ of Example~\ref{ex:matching}, parts of the
quivers $Q'$ and $(Q^*)'$ around the vertex $k$ are also shown in
Figure~\ref{fig:Qmatch}.
\end{example}

\begin{lemma} \label{l:Q'isoQ*'}
The quivers $Q'$ and $(Q^*)'$ are isomorphic. 
\end{lemma}
\begin{proof}
Going over the definitions as in the proof of Lemma~\ref{l:rho**},
we see that the quiver $(Q^*)'$ is obtained from $Q'$ by removing all the
arrows in $\Gamma'$ and then adding the arrows $[\beta^* \alpha^*]$ for
$(\alpha, \beta) \in C'$.
Hence $Q' \simeq (Q^*)'$ by the isomorphism $\vphi'$
taking each arrow $\gamma \in \Gamma'$ in $Q'$ to the arrow
$[\beta^*_\gamma \alpha^*_\gamma]$ in $(Q^*)'$, fixing all the vertices
and the other arrows.
\end{proof}

\begin{prop} \label{p:rho:admissible}
The map
\begin{equation} \label{e:Sdual}
S \mapsto
S^* = \sub_{\{\gamma \la
             -\beta^{*}_\gamma \alpha^{*}_\gamma\}_{\gamma \in \Gamma'}}([S])
\end{equation}
is an isomorphism from the algebra
of $\rho$-admissible elements of $\wh{KQ}$
to the algebra of $\rho^*$-admissible elements of $\wh{KQ^*}$.
\end{prop}
\begin{proof}
We start by verifying that $S^*$ is a well-defined $\rho^*$-admissible
element of $\wh{KQ^*}$. The map
$\psi = \sub_{\{\gamma \la
             -\beta^{*}_\gamma \alpha^{*}_\gamma\}_{\gamma \in \Gamma'}}$
is an endomorphism of the complete path algebra of $\wt{\mu}_k(Q)$.
Since $S$ is $\rho$-admissible, the expression $[S]$ is a well-defined
element of that algebra that does not contain any of the arrows
$[\alpha \beta]$ for $(\alpha, \beta) \in C'$. 
After applying the substitution,
$S^* = \psi(S)$ will not contain any of the arrows $\gamma \in \Gamma'$.
Hence $S^*$ can be seen as an element of $\wh{KQ^*}$ (embedded in the
natural way, as $Q^*$ is a subquiver of $\wt{\mu}_k(Q)$).
Moreover, since the only paths of the form $\beta^* \alpha^*$
which may appear in $S^*$ are $\beta^*_\gamma \alpha^*_\gamma$ for
$\gamma \in \Gamma'$, the element $S^*$ is $\rho^*$-admissible.

In view of Lemma~\ref{l:adm:Q'}, it suffices to show that
$\sub_{\{\gamma \la -[\beta^*_\gamma \alpha^*_\gamma]\}_{\gamma \in \Gamma'}}$
(formally speaking, an endomorphism of the complete path algebra of an
ambient quiver whose vertex set equals that of $Q'$ and $(Q^*)'$ and
whose arrow set contains the union of those of $Q'$ and $(Q^*)'$;
for instance, $\wt{\mu}_k \wt{\mu}_k(Q)$ is such quiver)
restricts to an isomorphism from the complete path algebra of $Q'$
to that of $(Q^*)'$.
This is clear from Lemma~\ref{l:Q'isoQ*'} and the isomorphism $\vphi'$
of the quivers $Q'$ and $(Q^*)'$ constructed in its proof;
by sending each $\gamma \in \Gamma'$ to $-\vphi'(\gamma)$ and fixing all the
vertices and the other arrows we get the required isomorphism of complete
path algebras.
\end{proof}

We note that the signs in~\eqref{e:Sdual}, which at first sight might
seem artificial, 
will be needed later when we consider mutations.

\begin{defn} \label{def:pot:admiss}
A potential $S$ on $Q$ is \emph{$\rho$-admissible} if it is 
$\rho$-admissible as an element of $\wh{KQ}$. In other words, 
\begin{enumerate}
\renewcommand{\theenumi}{\roman{enumi}}
\item \label{it:start:k}
None of the terms of $S$ starts at $k$,

\item
For any pair $(\alpha, \beta) \in C'$, the path $\alpha \beta$ does
not appear in $S$.
\end{enumerate}
\end{defn}

We note that by appropriately rotating its terms, any potential is
cyclically equivalent to a potential satisfying the
requirement~\eqref{it:start:k}.

As a consequence of Proposition~\ref{p:rho:admissible}, we obtain:
\begin{cor}
The map $S \mapsto S^*$ gives a linear bijection from
the space of $\rho$-admissible potentials on $Q$
to the space of $\rho^*$-admissible potentials on $Q^*$.
\end{cor}

\subsection{Mutation}
%%%%%%%%%%%%%%%%%%%%%
%\label{ssec:mut}

We now have all the ingredients needed to express the (pre-)mutations
of potentials of the form $W_\rho + S$, where $\rho$ is a (maximal)
matching and $S$ is a $\rho$-admissible potential,
in terms of the dual matching.

\begin{defn} \label{def:matching:max}
We say that a matching $\rho = (C', \Gamma', C' \xrightarrow{\sim} \Gamma')$
is \emph{maximal} if it has the maximal possible cardinality.
That is,
$|C'_{ij}| = |\Gamma'_{ij}| = \min \{|C_{ij}|, |\Gamma_{ij}|\}$
for all $i,j \neq k$. 
\end{defn}

\begin{example}
In Example~\ref{ex:matching}, 
the only maximal matching is $\rho_{12}$.
\end{example}

\begin{lemma} \label{l:Q*mut}
Assume that $Q$ has no loops or $2$-cycles.
Then the following conditions are equivalent for a matching $\rho$:
\begin{enumerate}
\renewcommand{\theenumi}{\alph{enumi}}
\item \label{it:max}
The matching is maximal.

\item \label{it:Q*no2cyc}
There are no $2$-cycles in the dual quiver $Q^*$.

\item \label{it:Q*mut}
The dual quiver $Q^*$ equals the mutation $\mu_k(Q)$.
\end{enumerate}
\end{lemma}
\begin{proof}
Since there are no $2$-cycles in $Q$, the only $2$-cycles in $Q^*$
are of the form
\[
\xymatrix{i \ar@<0.5ex>[r]^{[\alpha \beta]} & j \ar@<0.5ex>[l]^{\gamma}}
\]
for $(\alpha, \beta) \in C_{ij} \setminus C'_{ij}$ and
$\gamma \in \Gamma_{ij} \setminus \Gamma'_{ij}$.
Thus, if such a $2$-cycle exists, we can enlarge the matching by
adding $(\alpha, \beta)$ to $C'$ and $\gamma$ to $\Gamma'$, mapping
$(\alpha, \beta)$ to $\gamma$. This proves the equivalence of
\eqref{it:max} and~\eqref{it:Q*no2cyc}.

By Lemma~\ref{l:Qrho}(\ref{it:Q*Qrho}), $Q^*$ is obtained from
$\wt{\mu}_k(Q)$ by removing the $2$-cycles comprising the arrows
of the quiver $Q_\rho$. Thus, $Q^* = \mu_k(Q)$ if and only if
there are no $2$-cycles left in $Q^*$. This proves the equivalence
of \eqref{it:Q*no2cyc} and \eqref{it:Q*mut}.
\end{proof}

We are now ready to state and prove the main result of this section.

\begin{theorem} \label{t:QPmut}
Let $Q$ be a quiver without loops, let $k$ be a vertex of $Q$ such that
no $2$-cycle passes through $k$ and
let $\rho = (C', \Gamma', C' \xrightarrow{\sim} \Gamma')$ be a matching.

Let $S$ be a $\rho$-admissible potential on $Q$ and let
\[
S^* = \sub_{\{\gamma \la 
             -\beta^{*}_\gamma \alpha^{*}_\gamma\}_{\gamma \in \Gamma'}}([S])
\]
be the corresponding $\rho^*$-admissible potential on $Q^*$.

Consider the potential $W$ on $Q$ given by
\begin{align}
\nonumber
W = W_\rho + S &=
\sum_{(\alpha,\beta) \in C'} \alpha \beta \gamma_{\alpha,\beta} + S
\intertext{and the potential $W^*$ on $Q^*$ given by}
\label{e:W*}
W^* = W_{\rho^*} + S^* &=
\sum_{(\alpha, \beta) \in C \setminus C'} [\alpha \beta] \beta^* \alpha^* 
+ \sub_{\{\gamma \la
            -\beta^{*}_\gamma \alpha^{*}_\gamma\}_{\gamma \in \Gamma'}}([S]).
\end{align}

Then:
\begin{enumerate}
\renewcommand{\theenumi}{\alph{enumi}}
\item \label{it:QWpre}
The pre-mutation $\wt{\mu}_k(Q,W)$ is right equivalent to
$(Q^*, W^*) \oplus (Q_\rho, [W_\rho])$.

\item \label{it:QWmut}
If $Q$ does not have any $2$-cycles and
the matching $\rho$ is maximal, then
$Q^*=\mu_k(Q)$ and the mutation $\mu_k(Q,W)$ is right equivalent to
$(Q^*,W^*)$.
\end{enumerate}
\end{theorem}
\begin{proof}
In order to prove part~\eqref{it:QWpre}, we start by writing the
pre-mutation of $(Q,W)$ as
\[
\begin{split}
\wt{\mu}_k(Q,W) &=
\sum_{(\alpha,\beta) \in C} [\alpha \beta] \beta^* \alpha^* + [W] =
\sum_{(\alpha,\beta) \in C} [\alpha \beta] \beta^* \alpha^* +
\sum_{(\alpha,\beta) \in C'} [\alpha \beta] \gamma_{\alpha,\beta} + [S]
\\
&=
\sum_{(\alpha,\beta) \in C'} [\alpha \beta] \gamma_{\alpha,\beta} +
\sum_{(\alpha,\beta) \in C'} [\alpha \beta] \beta^* \alpha^* +
\biggl(
\sum_{(\alpha,\beta) \in C \setminus C'} [\alpha \beta] \beta^* \alpha^* +
[S]
\biggr) .
\end{split}
\]

Now set $\delta_{\alpha,\beta} = [\alpha \beta]$ and
$T_{\alpha,\beta} = \beta^* \alpha^*$ for $(\alpha,\beta) \in C'$.
Observe that
$\gamma_{\alpha,\beta}$ does not appear in any of the elements
$T_{\alpha',\beta'}$ and the arrow $\delta_{\alpha,\beta}$ does not appear
in any of the elements $T_{\alpha',\beta'}$ or in $[S]$.
Thus we can apply Proposition~\ref{p:2cyc} and deduce that $\wt{\mu}_k(Q,W)$
is right equivalent to the potential
\[
\begin{split}
&\sum_{(\alpha,\beta) \in C'} [\alpha \beta] \gamma_{\alpha,\beta} +
\sub_{\{\gamma_{\alpha,\beta} \la -T_{\alpha,\beta}\}_{(\alpha,\beta) \in C'}}
\Bigl(
\sum_{(\alpha,\beta) \in C \setminus C'} [\alpha \beta] \beta^* \alpha^* +
[S]
\Bigr) = \\
&\sum_{(\alpha,\beta) \in C'} [\alpha \beta] \gamma_{\alpha,\beta} +
\sum_{(\alpha,\beta) \in C \setminus C'} [\alpha \beta] \beta^* \alpha^* +
\sub_{\{\gamma \la -\beta^*_\gamma \alpha^*_\gamma \}_{\gamma \in \Gamma'}}
([S])
\end{split}
\]
which is equal to $(Q^*,W^*) \oplus (Q_\rho,[W_\rho])$.

Note that $(Q_\rho, [W_\rho])$ is a trivial QP by Lemma~\ref{l:Qrho}.
If the assumptions
in part~\eqref{it:QWmut} hold, then $Q^*$ does not contain any $2$-cycles and
hence $(Q^*,W^*)$ is a reduced QP and so by part~\eqref{it:QWpre} equals the
reduced part of a QP which is right equivalent to the pre-mutation
$\wt{\mu}_k(Q,W)$. It follows that the mutation $\mu_k(Q,W)$, being the
reduced part of the pre-mutation $\wt{\mu}_k(Q,W)$, is right equivalent
to $(Q^*,W^*)$.
\end{proof}

\begin{remark}
If the potential $W$ is a signed sum of cycles (i.e.\ each cycle appears
with coefficient equal to $1$ or $-1$), then so is $W^*$.
This makes the mutation operation purely combinatorial so
in principle it can be implemented on a computer.
The main challenge in applying Theorem~\ref{t:QPmut} in practice is,
given a potential $W$ on a quiver $Q$ and a vertex $k$,
to find a decomposition $W = W_0 + S$ such that $W_0$ encodes a
maximal matching with respect to which the potential $S$ is admissible.
\end{remark}

\begin{remark}
A special case of Theorem~\ref{t:QPmut} is that if $Q$ has no loops or
$2$-cycles and the matching $\rho$ is maximal,
then $\mu_k(Q,W_\rho) = (Q^*, W_{\rho^*})$, a result which could
also be verified directly. The theorem tells us that we can deform $W_\rho$
by adding any $\rho$-admissible potential $S$ and then the
result of the mutation is correspondingly deformed to $W_{\rho^*} + S^*$.
\end{remark}

In Section~\ref{sec:X7:mut} we will apply Theorem~\ref{t:QPmut} in order to
compute three kinds of mutations, see Proposition~\ref{p:glue:mut1}, 
Proposition~\ref{p:Qn:mut0} and Proposition~\ref{p:X7mut2}.

\subsection{Signs}
%%%%%%%%%%%%%%%%%%
The formula~\eqref{e:W*} for $W^*$ involves minus signs;
we have to replace each occurrence of an arrow $\gamma \in \Gamma'$ by
$-\beta^*_\gamma \alpha ^*_\gamma$. We would like to know when it is
possible to eliminate these signs. To this end, we consider 
the corresponding potential on $Q^*$ without any signs,
\[
W^*_+ =
\sum_{(\alpha, \beta) \in C \setminus C'} [\alpha \beta] \beta^* \alpha^* 
+ \sub_{\{\gamma \la
        \beta^{*}_\gamma \alpha^{*}_\gamma\}_{\gamma \in \Gamma'}}([S]),
\]
and find
sufficient conditions that $(Q^*,W^*)$ is right equivalent to $(Q^*,W^*_+)$.

Note that
$W^*_+ = W_{\rho^*} + (S_+)^*$ for the potential
$S_+ = \sub_{\{\gamma \la -\gamma\}_{\gamma \in \Gamma'}}(S)$ on $Q$, so
one could look for right equivalence of $W_\rho+S$ and $W_\rho+S_+$ on $Q$.
We will consider $Q^*$ instead, as the right equivalences we construct
are independent of $\Gamma'$.

Obviously, $W^*=W^*_+$ if the following condition $(\star)$ on $S$ holds:
\begin{itemize}
\item
the total number of occurrences of the arrows of $\Gamma'$ in any term of $S$
is even;

\item
in particular, if no arrow of $\Gamma'$ appears in $S$.
\end{itemize}

Denote by $A$ the set of arrows in $Q$ ending at $k$ and by $B$
the set of arrows in $Q$ starting at $k$, so that $C = A \times B$.
Let $A'$ be a subset of $A$ and $B'$ be a subset of $B$.
Denote by $A' \star B'$ the subset
$(A' \times (B \setminus B')) \cup ((A \setminus A') \times B')$ of $C$.

Let $\eps_{A',B'}$ be the right equivalence on the complete path algebra of
$\wt{\mu}_k(Q)$ defined by sending
\begin{align*}
\alpha^* \mapsto -\alpha^* &,&
\beta^* \mapsto -\beta^* &,&
[\alpha' \beta'] \mapsto -[\alpha' \beta']
\end{align*}
for $\alpha \in A'$, $\beta \in B'$ and 
$(\alpha',\beta') \in A' \star B'$, respectively,
keeping all the other arrows intact.
For any matching $\rho$, the restriction of $\eps_{A',B'}$ to $\wh{KQ^*}$
gives a right equivalence which will be denoted by the same notation.

\begin{lemma} \label{l:eAB:Wrho}
$\eps_{A',B'}(W_{\rho^*}) = W_{\rho^*}$ for any matching $\rho$.
\end{lemma}
\begin{proof}
Our choice of signs ensures that $\eps_{A',B'}$ maps each term
$[\alpha \beta] \beta^* \alpha^*$ to itself.
\end{proof}

Fix a matching $\rho = (C', \Gamma', C' \xrightarrow{\sim} \Gamma')$.
For a cycle $c$ of $\wt{\mu}_k(Q)$, let 
\begin{align*}
N_{c,1} &= \text{the total number of occurrences in $c$ of arrows
$[\alpha \beta]$ with $(\alpha, \beta) \in A' \star B'$}, \\
N_{c,2} &= \text{the total number of occurrences in $c$ of arrows
$\gamma \in \Gamma'$ with
$(\alpha_\gamma, \beta_\gamma) \not \in A' \star B'$}.
\end{align*}

\begin{lemma} \label{l:eAB:c}
We have
\[
\eps_{A',B'}
\bigl( \sub_{\{\gamma \la
        -\beta^{*}_\gamma \alpha^{*}_\gamma\}_{\gamma \in \Gamma'}}(c)
\bigr)
=
(-1)^{N_{c,1}+N_{c,2}}
\sub_{\{\gamma \la
        \beta^{*}_\gamma \alpha^{*}_\gamma\}_{\gamma \in \Gamma'}}(c).
\]
\end{lemma}
\begin{proof}
Evaluating $\eps_{A',B'}$ on
$\sub_{\{\gamma \la
        -\beta^{*}_\gamma \alpha^{*}_\gamma\}_{\gamma \in \Gamma'}}(c)$,
we see that each occurrence of an arrow $[\alpha \beta]$
contributes a factor of $-1$ if $(\alpha, \beta) \in A' \star B'$
and $1$ otherwise.
Similarly, if $\gamma \in \Gamma'$ then
$\eps_{A',B'}(-\beta^*_\gamma \alpha^*_\gamma) = 
\beta^*_\gamma \alpha^*_\gamma$
if $(\alpha_\gamma, \beta_\gamma) \in A' \star B'$
and 
$\eps_{A',B'}(-\beta^*_\gamma \alpha^*_\gamma) = 
-\beta^*_\gamma \alpha^*_\gamma$ otherwise. The claim now follows.
\end{proof}

We deduce the following sufficient condition on a potential $S$
for $\eps_{A',B'}(W^*) = W^*_+$ to hold, generalizing the
condition~$(\star)$.

\begin{lemma} \label{l:eAB:S}
Let $\rho$ be a matching and let $S$ be a potential on $Q$, with terms
rotated such that they do not start at the vertex $k$.
If $N_{c,1} + N_{c,2}$ is even for each term $c$ of $[S]$, then
\[
\eps_{A',B'}
\left( \sub_{\{\gamma \la
        -\beta^{*}_\gamma \alpha^{*}_\gamma\}_{\gamma \in \Gamma'}}([S])
\right)
=
\sub_{\{\gamma \la
        \beta^{*}_\gamma \alpha^{*}_\gamma\}_{\gamma \in \Gamma'}}([S])
\]
and hence
$\eps_{A',B'}(W^*) = W^*_+$
\end{lemma}
\begin{proof}
Use Lemma~\ref{l:eAB:Wrho} and Lemma~\ref{l:eAB:c}.
\end{proof}

We now give a condition on the matching that implies,
for any admissible potential, that
$W^*$ is right equivalent to $W^*_+$.

\begin{prop} \label{p:CABsign}
Let $\rho = (C', \Gamma', C' \xrightarrow{\sim} \Gamma')$ be a matching.
If $C' = A' \star B'$ for some $A' \subseteq A$ and $B' \subseteq B$,
then for any $\rho$-admissible potential $S$ on $Q$
we have $\eps_{A',B'}(W^*)=W^*_+$,
hence $(Q^*,W^*)$ is right equivalent to $(Q^*,W^*_+)$.
\end{prop}
\begin{proof}
If $C' = A' \star B'$ and $S$ is admissible with respect to $C'$, then
the number $N_{c,1}$ vanishes for any term $c$ of $[S]$.
Since $(\alpha_\gamma,\beta_\gamma) \in C'$ for any $\gamma \in \Gamma'$,
the numbers $N_{c,2}$ always vanish and the result is a consequence of
Lemma~\ref{l:eAB:S}.
\end{proof}

We list a few specific cases:

\begin{cor}
If $|C_{ij}| \leq |\Gamma_{ij}|$ for any $i,j \neq k$ then the conclusion
of Proposition~\ref{p:CABsign} holds for any maximal matching $\rho$.
\end{cor}
\begin{proof}
In this case, $\rho$ being maximal is equivalent to the condition $C'=C$,
so one can take $A'=A$ and $B'=\varnothing$ (or $A'=\varnothing$ and $B'=B$).
\end{proof}

\begin{cor} \label{c:onearrow}
If the vertex $k$ has only one incoming arrow or only one outgoing arrow,
then the conclusion of Proposition~\ref{p:CABsign} holds for any matching
$\rho$.
\end{cor}
\begin{proof}
Assume that $k$ has only one incoming arrow $\alpha$, the other case being
dual.
In this case we can write any subset $C' \subseteq C$  as $C' = A' \star B'$
for $A' = A = \{\alpha\}$ and
$B' = \{ \beta \in B \,:\, (\alpha,\beta) \in C \setminus C' \}$.
\end{proof}

\begin{example}
Let $Q$ be a quiver as in Example~\ref{ex:matching} and assume that it has
no loops or $2$-cycles. For the maximal matching $\rho_{12}$, the subset
$C' = \{(\alpha_1, \beta_1), (\alpha_2, \beta_2)\}$ can be written as
$A' \star B'$ for $A' = \{\alpha_1\}$ and $B' = \{\beta_2\}$.
Hence if $S$ is any potential on~$Q$ that does not contain the paths
$\alpha_1 \beta_1$ and $\alpha_2 \beta_2$, then the mutation of
the potential $\alpha_1 \beta_1 \gamma_1 + \alpha_2 \beta_2 \gamma_2 + S$
at the vertex $k$ is right equivalent to the potential
\[
[\alpha_1 \beta_2] \beta^*_2 \alpha^*_1 + 
[\alpha_2 \beta_1] \beta^*_1 \alpha^*_2 +
\sub_{\gamma_1 \la \beta^*_1 \alpha^*_1, \,
\gamma_2 \la \beta^*_2 \alpha^*_2}([S])
\]
on the quiver $Q^*$, part of which is shown in Figure~\ref{fig:Qmatch}.
\end{example}

\section{Mutations of potentials on $X_7$ and related quivers}
%%%%%%%%%%%%%%%%%%%%%%%%%%%%%%%%%%%%%%%%%%%%%%%%%%%%%%%%%%%%%%
\label{sec:X7:mut}

\subsection{Mutation at a side vertex}
%%%%%%%%%%%%%%%%%%%%%%%%%%%%%%%%%%%%%%
\label{ssec:mut:side}

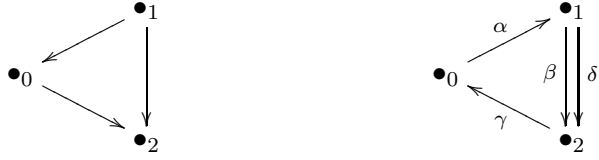
\begin{figure}
\begin{align*}
\xymatrix@=1pc{
&& {\bullet_1} \ar[dll] \ar[dd] \\
{\bullet_0} \ar[drr] \\
&& {\bullet_2}
}
&&
\xymatrix@=1pc{
&& {\bullet_1} \ar@<-0.2pc>[dd]_{\beta} \ar@<0.2pc>[dd]^{\delta} \\
{\bullet_0} \ar[urr]^{\alpha} \\
&& {\bullet_2} \ar[ull]^{\gamma}
}
\end{align*}
\caption{The quivers $\wt{A}_2$ (left) and $T_3$ (right).}
\label{fig:A2}
\end{figure}

The mutation class of the extended Dynkin quiver $\wt{A}_2$ consists
of the two quivers $\wt{A}_2$ and $T_3$ shown in Figure~\ref{fig:A2}.
In this section we consider quivers formed by gluing $T_3$ into
another quiver. This is done formally as follows;
a \emph{pointed quiver} is a pair $(Q,v)$ consisting of
a quiver $Q$ and a vertex $v$ of $Q$.
If $(Q',v')$ and $(Q'',v'')$ are pointed quivers, we can form their
\emph{gluing} by taking the disjoint union of the quivers $Q'$ and $Q''$
and then identifying the vertices $v'$ and $v''$.

Let $Q$ be a quiver without loops and $2$-cycles and let $v$ be a vertex
of $Q$. Consider the quiver $\wt{Q}$ which is the gluing of $(Q,v)$ and
$(T_3,0)$ identifying the vertices $v$ and $0$,
as in the following picture:
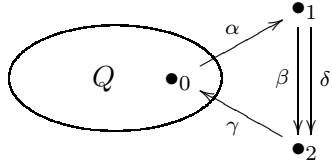
\[
\xymatrix@=1pc{
&&& {\bullet_1} \ar@<-0.2pc>[dd]_{\beta} \ar@<0.2pc>[dd]^{\delta} \\
Q {\xy \ellipse(14,7){} \endxy} &
{\bullet_0} \ar[urr]^{\alpha} \\
&&& {\bullet_2} \ar[ull]^{\gamma}
}
\]

Denote by $B = \alpha \beta \gamma$ and $\Delta = \alpha \delta \gamma$
the two $3$-cycles passing through the vertex $1$.

\begin{prop} \label{p:glue:mut1}
Let $S$ be a potential on $\wt{Q}$ and assume that $\beta$ does not appear
in $S$.
Then the mutation $\mu_1(\wt{Q},B+S)$ is isomorphic to $(\wt{Q},B+S)$.
\end{prop}
\begin{proof}
In the notations of Section~\ref{ssec:match},
the sets $C$ and $\Gamma$ are given by
$C = \{(\alpha,\beta), (\alpha,\delta) \}$ and $\Gamma = \{\gamma\}$.
Choose a maximal matching $\rho$ with
$C' = \{(\alpha,\beta)\}$ and $\Gamma'= \{\gamma\} = \Gamma$,
so that $W_\rho = B$.
Since there is only one incoming arrow to the vertex $1$,
a potential $S$ on $\wt{Q}$ is $\rho$-admissible if and only if
$\beta$ does not appear in $S$.

Computing the dual quiver $\wt{Q}^*$ with respect to $\rho$, we get
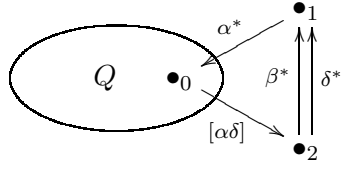
\[
\xymatrix@=1pc{
&&& {\bullet_1} \ar[dll]_{\alpha^*} \\
Q {\xy \ellipse(14,7){} \endxy} &
{\bullet_0} \ar[drr]_{[\alpha \delta]} \\
&&& {\bullet_2} \ar@<0.2pc>[uu]^{\beta^*} \ar@<-0.2pc>[uu]_{\delta^*}
}
\]
with the arrows of $Q$ unchanged, and the subsets defining the dual
matching $\rho^*$ are given by
$(C^*)' = \{(\delta^*, \alpha^*)\}$ and $(\Gamma^*)' = \{[\alpha \delta]\}$.

By Theorem~\ref{t:QPmut} and Corollary~\ref{c:onearrow},
we deduce that the mutation $\mu_1(\wt{Q}, W)$
of the potential $W = B + S$ on $\wt{Q}$
is given by $(\wt{Q}^*, W^*_+)$,
where $W^*_+ = [\alpha \delta] \delta^* \alpha^*
+ \sub_{\gamma \la \beta^* \alpha^*}([S])$.
Thus, in order to complete the proof, we have to
construct a continuous isomorphism of the complete path algebras
of $\wt{Q}$ and $\wt{Q^*}$ sending $W$ to $W^*_+$.

Indeed, the quivers $\wt{Q}$ and $\wt{Q}^*$ are isomorphic by sending
\begin{align*}
\alpha \mapsto [\alpha \delta] &,&
\beta \mapsto \delta^* &,& 
\gamma \mapsto \alpha^* &,&
\delta \mapsto \beta^* ,
\end{align*}
fixing all the other arrows, exchanging the vertices $1$ and $2$ and
fixing the other vertices. This isomorphism
induces an isomorphism between the respective complete path algebras,
which we will denote by $\psi$.

If $w$ is a cycle inside $Q$, then none of the arrows
$\alpha, \beta, \gamma, \delta$ appears in $w$ and we trivially have
$\psi(w) = w = \sub_{\gamma \la \beta^* \alpha^*}([w])$.
We also have
\[
\psi(\Delta) = [\alpha \delta] \beta^* \alpha^* = 
\sub_{\gamma \la \beta^* \alpha^*}([\Delta]),
\]
so if $\beta$ does not appear in $S$, then any term of $S$ is 
rotationally equivalent to a concatenation of cycles which are either
inside $Q$ or equal to $\Delta$, hence
$\psi(S) = \sub_{\gamma \la \beta^* \alpha^*}([S])$.

Finally, $\psi(B) = [\alpha \delta] \delta^* \alpha^*$,
so
\[
\psi(W) = \psi(B+S) = [\alpha \delta] \delta^* \alpha^* +
\sub_{\gamma \la \beta^* \alpha^*}([S]) 
= W^*_+
\]
as required.
\end{proof}

We record also the dual statement whose proof is similar and thus omitted.
\begin{prop} \label{p:glue:mut2}
Let $S$ be a potential on $\wt{Q}$ and assume that $\beta$ does not appear
in $S$. Then the mutation $\mu_2(\wt{Q},B+S)$ is isomorphic to $(\wt{Q},B+S)$.
\end{prop}

The next proposition shows that if $\beta$ does not appear in a potential
$S$, we can deform $B+S$ by adding arbitrary terms containing at
least two consecutive cycles inside $T_3$.

\begin{prop} \label{p:glue:deform}
Let $S$ be a potential on $\wt{Q}$ and assume that
$\beta$ does not appear in $S$.
Then for any scalars $\lambda, \lambda' \in K$ with $\lambda \neq 0$ and
any four potentials $S', S'', S''', S''''$ in $e_0 \wh{K\wt{Q}} e_0$,
the potentials $B + S$ and
$
\lambda B + \lambda' \Delta + S
+ B^2 S' + B \Delta S'' + \Delta B S''' + \Delta^2 S''''
$
are right equivalent.
\end{prop}
\begin{proof}
The element
$T = \lambda \beta + \lambda' \delta + \beta \gamma S' \alpha \beta
+ \delta \gamma S'' \alpha \beta + \beta \gamma S''' \alpha \delta
+ \delta \gamma S'''' \alpha \delta$
is parallel to $\beta$.
The endomorphism
$\sub_{\beta \la T}$ is a right equivalence by Lemma~\ref{l:sub:gTauto:m2}.
Applying it on $B + S = \alpha \beta \gamma + S$, 
observing that $\beta$ appears only in $B$, we get
\[
\lambda B + \lambda' \Delta
+ B S' B + \Delta S'' B + B S''' \Delta + \Delta S'''' \Delta
+ S
\]
which is cyclically equivalent to the required potential.
\end{proof}

\subsection{The quivers $Q_n$ and mutations of potentials on them}
%%%%%%%%%%%%%%%%%%%%%%%%%%%%%%%%%%%%%%%%%%%%%%%%%%%%%%%%%%%%%%%%%%
\label{ssec:Qn:mut:side}

In this section we introduce a sequence of quivers $Q_n$ and
investigate mutations of certain potentials on them.

\begin{notat}
Let $n \geq 1$. Denote by $Q_n$ the quiver on the $2n+1$ vertices
$0, 1, 2, \dots, 2n$ with $4n$ arrows
$\alpha_i, \beta_i, \gamma_i, \delta_i$ for $1 \leq i \leq n$ given by
\begin{align*}
0 \xrightarrow{\alpha_i} 2i-1 &,&
2i-1 \xrightarrow{\beta_i} 2i &,&
2i \xrightarrow{\gamma_i} 0 &,&
2i-1 \xrightarrow{\delta_i} 2i &&
(1 \leq i \leq n).
\end{align*}
\end{notat}

The quivers $Q_n$ for $n=1,2$ are shown in Figure~\ref{fig:Qn}.
Observe that the initial quiver $Q_1$ is the quiver $T_3$
of Figure~\ref{fig:A2}, and
$Q_{n+1}$ is
obtained from $Q_n$ by gluing a copy of $T_3$ at the
common vertex $0$. It follows that $Q_n$ can be viewed as iterated gluing
of $n$ copies of the quiver $T_3$ at the common vertex $0$
and that each quiver $Q_n$ is a full subquiver
of the next quiver $Q_{n+1}$ in the sequence.

The quiver $Q_3$ is precisely the quiver $X_7$ of
Derksen-Owen~\cite{DerksenOwen08} shown in Figure~\ref{fig:X7}.
It follows that the mutation class of $Q_n$ is finite if and only if
$n \leq 3$.
The quivers $Q_1$ and $Q_2$ arise from triangulations of
marked surfaces~\cite{FST08};
the surface corresponding to $Q_1$ is an annulus with
two marked points on one boundary component and one marked point
on the other boundary component;
the surface corresponding to $Q_2$ is
a punctured annulus with one marked point on each of the two boundary
components.

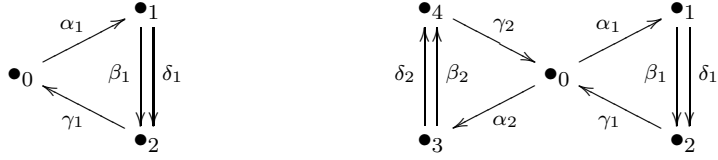
\begin{figure}
\begin{align*}
\xymatrix@=1pc{
&& {\bullet_1} \ar@<-0.2pc>[dd]_{\beta_1} \ar@<0.2pc>[dd]^{\delta_1} \\
{\bullet_0} \ar[urr]^{\alpha_1} \\
&& {\bullet_2} \ar[ull]^{\gamma_1}
}
&&
\xymatrix@=1pc{
{\bullet_4} \ar[drr]^{\gamma_2} &&
&& {\bullet_1} \ar@<-0.2pc>[dd]_{\beta_1} \ar@<0.2pc>[dd]^{\delta_1} \\
&& {\bullet_0} \ar[urr]^{\alpha_1} \ar[dll]^{\alpha_2} \\
{\bullet_3} \ar@<-0.2pc>[uu]_{\beta_2} \ar@<0.2pc>[uu]^{\delta_2} &&
&& {\bullet_2} \ar[ull]^{\gamma_1}
}
\end{align*}
\caption{The quivers $Q_n$ for $n=1,2$.}
\label{fig:Qn}
\end{figure}

For $1 \leq i \leq n$,
denote by $B_i = \alpha_i \beta_i \gamma_i$ and 
$\Delta_i = \alpha_i \delta_i \gamma_i$ the two $3$-cycles passing through
the vertices $2i-1$ and $2i$.

We start by recording the following observations:
\begin{lemma} \label{l:Qn:comb}
Let $S$ be a potential on $Q_n$ and let $1 \leq i \leq n$.
\renewcommand{\theenumi}{\alph{enumi}}
\begin{enumerate}
\item \label{it:acyclic}
The full subquiver of $Q_n$ on the set of vertices $1,2,\dots,2n$
is acyclic.

\item \label{it:words}
$S$ is cyclically equivalent to a (possibly infinite)
linear combination of words in the letters 
$B_1, \dots, B_n, \Delta_1, \dots, \Delta_n$.

\item
The arrow $\beta_i$ does not appear in $S$ if and only if
none the words in~\eqref{it:words} contains the letter $B_i$.

\item \label{it:admiss}
Rotate the terms of $S$ such that they do not start at $0$.
Then the path $\gamma_i \alpha_i$ does not appear in $S$
if and only if none of the words in~\eqref{it:words} is
$B_i$, $\Delta_i$, or up to rotation, contains as subword any of the words
$B_i B_i$, $B_i \Delta_i$, $\Delta_i B_i$ or $\Delta_i \Delta_i$.
\end{enumerate}
\end{lemma}
\begin{proof}
{\ }
\renewcommand{\theenumi}{\alph{enumi}}
\begin{enumerate}
\item
Deleting the vertex $0$ and all the arrows incident to it, we are left
with $n$ disjoint copies of the Kronecker quiver.

\item
By the previous claim, any cycle of $Q_n$ is rotationally equivalent to
a cycle starting at the vertex $0$. Each such cycle is a concatenation of
cycles $B_j$ and $\Delta_j$ for $1 \leq j \leq n$.

\item
The only letter containing $\beta_i$ is $B_i$.

\item
If $(i_1, i_2, \dots, i_r)$ is the sequence of indices of the letters
comprising a word, then the paths of the form $\gamma_i \alpha_j$
appearing in any of its rotations are $\gamma_{i_r} \alpha_{i_1}$ and
$\gamma_{i_s} \alpha_{i_{s+1}}$ for $1 \leq s < r$.
Hence $\gamma_i \alpha_i$ does not appear if and only if
$(i_r,i_1) \neq (i,i)$ and $(i_s,i_{s+1}) \neq (i,i)$ for all $1 \leq s < r$.
If $r=1$, these conditions reduce to $i_1 \neq i$.
\end{enumerate}
\end{proof}

We now deal with mutations at the side vertices.
\begin{prop} \label{p:Qn:mutk}
Let $S$ be a potential on $Q_n$.
Assume that none of the arrows $\beta_1,\dots,\beta_n$ appears in $S$,
or equivalently, all the terms of $S$ are scalar multiples of words in
the letters $\Delta_i$.
Let $W = B_1 + \dots + B_n + S$.
Then $\mu_k(Q_n,W)$ is isomorphic to $(Q_n,W)$
for any vertex $k \neq 0$ of $Q_n$.
\end{prop}
\begin{proof}
Let $i = \lfloor (k+1)/2 \rfloor$ and let $Q$ be the full subquiver of
$Q_n$ on the set of vertices $\{0,1,2,\dots,2n+1\} \setminus \{2i-1, 2i\}$.
Then $Q_n \simeq \wt{Q}$ in the notations of the previous section, 
so we can apply Proposition~\ref{p:glue:mut1} or Proposition~\ref{p:glue:mut2}
for the potential $\sum_{j \neq i} B_j + S$ in which the arrow $\beta_i$
does not appear.
\end{proof}

The next statement concerns deformations of a potential
$\sum_{i=1}^n B_i + S$ by terms containing
at least two consecutive cycles inside some
of the $n$ blocks glued together to form the quiver~$Q_n$.

\begin{prop} \label{p:Qn:deform}
Let $S$ be a potential on $Q_n$ and let $I \subseteq \{1,2,\dots,n\}$.
Assume that no arrow
$\beta_i$ for $i \in I$ appears in $S$.
For each $i \in I$, let $\lambda_i, \lambda'_i$ be scalars such that
$\lambda_i \neq 0$
and let $S'_i, S''_i, S'''_i, S''''_i$ be potentials
in $e_0 \wh{KQ_n} e_0$.

Then the potential $B_1 + \dots + B_n + S$ is right equivalent to
\[
\sum_{i \not \in I} B_i + 
\sum_{i \in I} \lambda_i B_i + S + \sum_{i \in I} \lambda'_i \Delta_i
+ \sum_{i \in I} B_i^2 S'_i 
+ \sum_{i \in I} B_i \Delta_i S''_i
+ \sum_{i \in I} \Delta_i B_i S'''_i
+ \sum_{i \in I} \Delta_i^2 S''''_i
\]
\end{prop}
\begin{proof}
We argue as in the proof of Proposition~\ref{p:glue:deform}
(which cannot be directly applied since after deforming at the first
index $i$, the potential may now have terms other than $B_j$ containing
the arrows $\beta_j$ for some $j \in I$).

For each $i \in I$, let
\[
T_i = \lambda_i \beta_i + \lambda'_i \delta_i
+ \beta_i \gamma_i S'_i \alpha_i \beta_i
+ \delta \gamma S'' \alpha \beta + \beta \gamma S'''_i \alpha_i \delta_i
+ \delta \gamma S'''' \alpha \delta
\]
Then $T_i$ is parallel to $\beta_i$ and the endomorphism 
$\sub_{\{\beta_i \la T_i\}_{i \in I}}$
defined by sending each $\beta_i$ to $T_i$ while
fixing all other arrows and vertices
is a right equivalence by an argument similar to that
of Lemma~\ref{l:sub:gTauto:m2} relying on~\cite[Proposition~2.4]{DWZ08}.
We get the result by evaluating $\sub_{\{\beta_i \la T_i\}_{i \in I}}$
on $B_1 + \dots + B_n + S$.
\end{proof}

The case where the subset $I$ consists of just one block is a special case
of Proposition~\ref{p:glue:deform} and will be used in the proof of
Lemma~\ref{l:FhG:jk}.
Another case, where $I$ consists of all the $n$ blocks, will be used in the
proof of Proposition~\ref{p:Qn:Wf}.

\subsection{Mutations at the central vertex and the quivers $Q'_n$}
%%%%%%%%%%%%%%%%%%%%%%%%%%%%%%%%%%%%%%%%%%%%%%%%%%%%%%%%%%%%%%%%%%%
\label{ssec:Qn:mut:center}

We turn our attention to mutation at the central vertex $0$.
In the notations of Section~\ref{ssec:match}, the sets $C$ and $\Gamma$ are
given by $C = \{(\gamma_i,\alpha_j)\}_{1 \leq i, j \leq n}$ and
$\Gamma = \{ \beta_i, \delta_i \}_{1 \leq i \leq n}$.

\begin{notat}
Throughout this section,
fix a matching $\rho$ with $C' = \{(\gamma_i,\alpha_i)\}_{1 \leq i \leq n}$
and $\Gamma' = \{\beta_i\}_{1 \leq i \leq n}$, sending each
pair $(\gamma_i, \alpha_i)$ to $\beta_i$.
Observe that the matching $\rho$ is maximal.
\end{notat}

\begin{notat}
Let $n \geq 1$. Denote by $Q'_n$ the quiver on the $2n+1$ vertices
$0,1,2,\dots,2n$ with arrows
$a_i, c_i, d_i$ ($1 \leq i \leq n$) and $b_{ij}$ ($1 \leq i \ne j \leq n$)
given by
\begin{align*}
0 \xrightarrow{a_i} 2i &,&
2i \xrightarrow{b_{ij}} 2j-1 &,&
2i-1 \xrightarrow{c_i} 0 &,&
2i-1 \xrightarrow{d_i} 2i .
\end{align*}
Observe that the quiver $Q'_n$ has $n(n+2)$ arrows.
Call the arrows $a_i$ and $c_i$ \emph{radial}, as they are
incident to the central vertex $0$.
\end{notat}

The quivers $Q'_n$ for $n=1,2$ are shown in Figure~\ref{fig:Q'n}.
Note that $Q'_1$ is the extended Dynkin quiver $\wt{A_2}$ and
the quiver $Q'_3$ equals the other member $X'_7$ in the mutation class
of $X_7$, shown in Figure~\ref{fig:X7}.

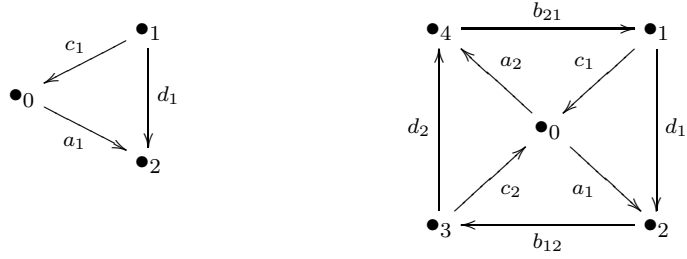
\begin{figure}
\begin{align*}
\begin{array}{c}
\xymatrix@=1pc{
&& {\bullet_1} \ar[dll]_{c_1} \ar[dd]^{d_1} \\
{\bullet_0} \ar[drr]_{a_1} \\
&& {\bullet_2}
}
\end{array}
&&
\begin{array}{c}
\xymatrix{
{\bullet_4} \ar[rr]^{b_{21}} &
& {\bullet_1} \ar[dl]_{c_1} \ar[dd]^{d_1} \\
& {\bullet_0} \ar[dr]_{a_1} \ar[ul]_{a_2} \\
{\bullet_3} \ar[ur]_{c_2} \ar[uu]^{d_2} &
& {\bullet_2} \ar[ll]^{b_{12}}
}
\end{array}
\end{align*}
\caption{The quivers $Q'_n$ for $n=1,2$.}
\label{fig:Q'n}
\end{figure}

\begin{lemma} \label{l:Qn:mut0}
$\mu_0(Q_n) \simeq Q'_n$.
\end{lemma}
\begin{proof}
By Lemma~\ref{l:Q*mut}, the mutation $\mu_0(Q_n)$ equals the dual quiver 
$Q^*_n$ with respect to $\rho$, whose arrows are
$\alpha^*_i, \gamma^*_i, \delta_i$ for $1 \leq i \leq n$
and $[\gamma_i \alpha_j]$ for $1 \leq i \neq j \leq n$ , given by
\begin{align*}
2i-1 \xrightarrow{\alpha_i^*} 0 &,&
0 \xrightarrow{\gamma_i^*} 2i &,&
2i-1 \xrightarrow{\delta_i} 2i &,&
2i \xrightarrow{[\gamma_i \alpha_j]} 2j-1.
\end{align*}
The quivers $Q^*_n$ and $Q'_n$ are clearly seen to be isomorphic, by
sending
\begin{align}
\label{e:Qn:mut0}
\alpha^*_i \mapsto c_i &,&
\gamma^*_i \mapsto a_i &,&
\delta_i \mapsto d_i &,&
[\gamma_i \alpha_j] \mapsto b_{ij} .
\end{align}
\end{proof}

\begin{lemma} \label{l:Qn:Wrho}
For the matching $\rho$,
\begin{enumerate}
\renewcommand{\theenumi}{\alph{enumi}}
\item
$W_\rho$ is cyclically equivalent to $B_1 + B_2 + \dots + B_n$.

\item
Under the isomorphism of~\eqref{e:Qn:mut0},
$W_{\rho^*}$ is cyclically equivalent to
$\sum_{1 \leq i \neq j \leq n} a_i b_{ij} c_j$.
\end{enumerate}
\end{lemma}
\begin{proof}
We have $W_\rho = \sum_{i=1}^{n} \gamma_i \alpha_i \beta_i$ in $Q_n$
and
$W_{\rho^*} = \sum_{1 \leq i \neq j \leq n}
[\gamma_i \alpha_j] \alpha^*_j \gamma^*_i$ in $Q_n^*$.
\end{proof}

Before dealing with potentials, we record a few observations concerning
symmetries of the quivers $Q_n$ and $Q'_n$. Any permutation
$\sigma \in S_n$ defines an automorphism of the quiver $Q_n$ by
permuting its $n$ blocks; namely, 
fixing the vertex $0$, acting on the other vertices by
\begin{align*}
2i-1 \mapsto 2\sigma(i)-1 &,&
2i \mapsto 2\sigma(i) &,& (1 \leq i \leq n)
\end{align*}
and sending each of the arrows $\alpha_i$, $\beta_i$, $\gamma_i$, $\delta_i$
to the corresponding arrow $\alpha_{\sigma(i)}$, $\beta_{\sigma(i)}$,
$\gamma_{\sigma(i)}$, $\delta_{\sigma(i)}$, respectively.
It induces an automorphism of the mutated quiver $Q'_n$ by the same action
on the vertices, mapping each of the arrows $a_i$, $c_i$, $d_i$ to
$a_{\sigma(i)}$, $c_{\sigma(i)}$, $d_{\sigma(i)}$ respectively, and
each arrow $b_{ij}$ to $b_{\sigma(i)\sigma(j)}$.

Recall that the \emph{opposite} quiver $Q^{op}$ of a quiver $Q$ is
the quiver obtained from $Q$ by reversing all its arrows. Formally, if
$Q=(Q_0, Q_1, s, t)$, then $Q^{op}=(Q_0^*, Q_1^*, s^*, t^*)$
where $Q_0^*=Q_0$, $Q_1^* = \{ \alpha^* \,:\, \alpha \in Q_1 \}$ and
$s^*(\alpha^*) = t(\alpha)$, $t^*(\alpha^*)=s(\alpha)$ for any
$\alpha \in Q_1$.

\begin{lemma} \label{l:Qn:opp}
There are isomorphisms $Q_n \simeq Q_n^{op}$ and $Q'_n \simeq (Q'_n)^{op}$.
\end{lemma}
\begin{proof}
The quiver $Q_n$ is isomorphic to its opposite by
the isomorphism fixing the vertex $0$, exchanging the odd-numbered vertices
$2i-1$ and the even-numbered ones $2i$ for all $1 \leq i \leq n$
and mapping each of the arrows
$\alpha_i$, $\beta_i$, $\gamma_i$, $\delta_i$
to
$\gamma_i^*$, $\beta^*_i$, $\alpha^*_i$, $\delta^*_i$ respectively.
Similarly,
the quiver $Q'_n$ is isomorphic to its opposite by the isomorphism having
the same action on the vertices, mapping each of the arrows
$a_i$, $c_i$, $d_i$ to $c^*_i$, $a^*_i$, $d^*_i$ respectively
and each arrow $b_{ij}$ to $b^*_{ji}$.
\end{proof}

Let $R = K \llangle x_1, \dots, x_n, y_1, \dots, y_n \rrangle$ be the ring of
power series in $2n$ non-commuting variables with coefficients in $K$.
Define a continuous ring homomorphism $\Phi \colon R \to e_0 \wh{KQ_n} e_0$
by setting its values on the generators as
\begin{align*}
\Phi(1) = e_0 &,& \Phi(x_i) = \Delta_i &,& \Phi(y_i) = B_i &&
& (1 \leq i \leq n).
\end{align*}

In other words, if
$F \in K \llangle x_1, \dots, x_n, y_1, \dots, y_n \rrangle$ then
\[
\Phi(F) = F(\Delta_1, \dots, \Delta_n, B_1, \dots, B_n) .
\]
In particular, the image of any monomial in $R$ is a corresponding word in the
letters $B_i$ and $\Delta_i$.

\begin{lemma}
The map $\Phi \colon R \to e_0 \wh{KQ_n} e_0$ is an isomorphism.
\end{lemma}
\begin{proof}
This is a consequence of Lemma~\ref{l:Qn:comb}\eqref{it:words}.
\end{proof}

Similarly to $\Phi$, we would like to define a map $\Psi$ to the space of
potentials on $Q'_n$. However, its domain of definition will not be entire $R$
but rather a subspace of it consisting of the admissible power series, to be
defined next.
To do this, consider the ring $A = \llangle x_1, \dots, x_n \rrangle$ of 
power series in $n$ non-commuting variables. We have a natural
homomorphism $\pi \colon R \twoheadrightarrow A$ defined by
$\pi(x_i) = \pi(y_i) = x_i$ for all $1 \leq i \leq n$.

\begin{defn} \label{def:admiss}
A monomial $x_{i_1} x_{i_2} \dots x_{i_r} \in A$ is \emph{admissible}
if it has degree at least $2$ and any two consecutive variables (in
cyclic order) are different. In other words,
$r>1$, $i_s \neq i_{s+1}$ for all $1 \leq s < r$, and
$i_r \neq i_1$.

A monomial $w \in R$ is \emph{admissible} if $\pi(w)$ is admissible
as a monomial in $A$. In other words, it has degree at least $2$ and
no word of the form $x_i x_i, x_i y_i, y_i x_i, y_i y_i$ appears
as subword in $w$ or its rotations.

A power series in $R$ is \emph{admissible} if 
any monomial occurring with non-zero coefficient is
admissible.
\end{defn}

Strictly speaking,
in order to be able to mutate a potential at the vertex $0$, its
terms should not start at $0$. However, in our description as $\Phi(F)$
for $F \in R$ all the terms start at $0$.
As usual, one addresses this by appropriately rotating the terms.

\begin{notat}
Let $\omega$ be a cycle in $Q_n$ starting at the vertex $0$.
We denote by $\rot(\omega)$ the cycle obtained from $\omega$ by rotating
it one step so that its first arrow is the second arrow of $\omega$ and its
last arrow is the first arrow of $\omega$.

Extend this by linearity and continuity, so that
if $S$ is a potential in $e_0 \wh{KQ_n} e_0$, then $\rot(S)$ is the
potential on $Q_n$ obtained by rotating each cycle occurring in $S$.
\end{notat}

\begin{lemma} \label{l:Qn:Fadmiss}
Let $F \in R$. Then $\rot(\Phi(F))$ is $\rho$-admissible
if and only if $F$ is admissible.
\end{lemma}
\begin{proof}
This is a reformulation of Lemma~\ref{l:Qn:comb}\eqref{it:admiss}
using Definition~\ref{def:admiss}.
\end{proof}

We will now define a continuous linear map $\Psi$ from the subspace
of admissible power series in $R$ to $\wh{KQ'_n}$. 
It suffices to define $\Psi$ for admissible monomials.

\begin{defn} \label{def:Psi}
Let $w \in R$ be an admissible monomial. Write
$\pi(w) = x_{i_1} x_{i_2} \dots x_{i_r}$ so that $r>1$, $i_s \neq i_{s+1}$
for $1 \leq s < r$ and $i_r \neq i_1$.
For $1 \leq s \leq r$,
define an element $p_s$ in $\wh{KQ'_n}$ which is a scalar multiple
of a path from the vertex $2i_s-1$ to $2i_s$, by
\[
p_s = \begin{cases}
d_{i_s} & \text{if the $s$-th letter in $w$ equals $x_{i_s}$,} \\
-c_{i_s} a_{i_s} & \text{if the $s$-th letter in $w$ equals $y_{i_s}$.}
\end{cases}
\]
Now define
\begin{equation} \label{e:Psi}
\Psi(w) = 
p_1 b_{i_1 i_2} p_2 \dots
b_{i_{r-1} i_r} p_r b_{i_r i_1} .
\end{equation}
\end{defn}

\begin{lemma} \label{l:Psi}
Let $F \in R$ be admissible. Then the isomorphism~\eqref{e:Qn:mut0}
identifies $\Psi(F)$ with
$\sub_{\{\beta_i \la -\alpha^*_i \gamma^*_i\}_{i=1}^n}([\rot(\Phi(F))])$.
\end{lemma}
\begin{proof}
It suffices to consider the case that $F$ is an admissible monomial.
In the notations of Definition~\ref{def:Psi}, $\rot(\Phi(F))$ is then a
concatenation of $r$ paths of the form
$\delta_{i_s} \gamma_{i_s} \alpha_{i_{s+1}}$
or $\beta_{i_s} \gamma_{i_s} \alpha_{i_{s+1}}$,
where we set $i_{r+1}=i_1$.

Since $F$ is admissible, in $[\rot(\Phi(F)]$ each path
$\gamma_{i_s} \alpha_{i_{s+1}}$ is replaced by the arrow
$[\gamma_{i_s} \alpha_{i_{s+1}}]$, so after applying the substitutions,
the path $\delta_{i_s} \gamma_{i_s} \alpha_{i_{s+1}}$
becomes $\delta_{i_s} [\gamma_{i_s} \alpha_{i_{s+1}}]$
and $\beta_{i_s} \gamma_{i_s} \alpha_{i_{s+1}}$
becomes $-\alpha^*_{i_s} \gamma^*_{i_s} [\gamma_{i_s} \alpha_{i_{s+1}}]$.
Applying the isomorphism~\eqref{e:Qn:mut0}, these become
identified with $d_{i_s} b_{i_s i_{s+1}}$ and
$-c_{i_s} a_{i_s} b_{i_s i_{s+1}}$, respectively.
Concatenating the $r$ paths,
we get the expression in~\eqref{e:Psi}, as required.
\end{proof}

\begin{prop} \label{p:Qn:mut0}
Let $F \in K \llangle x_1, \dots, x_n, y_1, \dots, y_n \rrangle$ be an 
admissible power series.
Then
\[
\mu_0(Q_n, B_1 + \dots + B_n + \Phi(F)) \simeq
\Bigl( Q'_n, \sum_{1 \leq i \neq j \leq n} a_i b_{ij} c_j + \Psi(F) \Bigl).
\]
\end{prop}
\begin{proof}
Replace $\Phi(F)$ by the cyclically equivalent potential $S=\rot(\Phi(F))$,
which is $\rho$-admissible by Lemma~\ref{l:Qn:Fadmiss}.
By Theorem~\ref{t:QPmut}, $\mu_0(Q_n, W_\rho + S)$ is right
equivalent to $(Q^*_n, W_{\rho*} +
\sub_{\{\beta_i \la -\alpha_i^* \gamma^*_i\}}([S]))$.
The result is now a consequence of Lemma~\ref{l:Qn:mut0},
Lemma~\ref{l:Qn:Wrho} and Lemma~\ref{l:Psi}.
\end{proof}

\begin{cor}
Let $F \in K \llangle x_1, \dots, x_n, y_1, \dots, y_n \rrangle$ be an 
admissible power series. If the potential $B_1 + \dots + B_n + \Phi(F)$
on $Q_n$ is non-degenerate, then for any $i<j$
the sum of the coefficients of the monomials $x_i x_j$ 
and $x_j x_i$ in $F$ is not zero.
\end{cor}
\begin{proof}
The full subquiver of $Q'_n$ on the set of vertices $2i-1, 2i, 2j-1, 2j$
is a chordless cycle of length $4$.
Any non-degenerate potential on it must contain the cycle
$d_i b_{ij} d_j b_{ji}$ or its rotations.

If a potential is non-degenerate, then so are its mutations and any
restriction to a full subquiver~\cite[Corollary~22]{Labardini09}.
The claim now follows by
performing a mutation at $0$ and then restricting to the above subquiver,
since $\Psi(x_i x_j) = d_i b_{ij} d_j b_{ji}$
and $\Psi(x_j x_i) = d_j b_{ji} d_i b_{ij}$.
\end{proof}

To illustrate the results of this section, we construct a family of
potentials on $Q_n$ for which all the mutations at vertices of $Q_n$ are
completely understood.

Call a power series $f(x_1, \dots, x_n) \in A$ \emph{reduced}
if all its terms are admissible.
The \emph{reduction} $\bar{f}$ of a power series $f$ is defined
by deleting all the non-admissible monomials occurring in $f$.

Viewing $A$ as a subspace of $R$ in the natural way, the maps
$\Phi$ and $\Psi$ restrict to $A$ and its subspace of reduced
power series, respectively. We have
$\Phi(f) = f(\Delta_1, \dots, \Delta_n)$ for any power series $f \in A$,
while 
\begin{equation} \label{e:PsiA}
\Psi(x_{i_1} x_{i_2} \dots x_{i_r})
= d_{i_1} b_{i_1 i_2} d_{i_2} b_{i_2 i_3} \dots d_{i_r} b_{i_r i_1}
\end{equation}
for any admissible monomial $x_{i_1} x_{i_2} \dots x_{i_r}$ in $A$.

\begin{remark}
Note that the expression in~\eqref{e:PsiA} is a cycle in $Q'_n$ without
any radial arrows, hence for any reduced power series $f$,
no radial arrows appear in $\Psi(f)$.

Conversely, any cycle in $Q'_n$ starting at an odd-numbered vertex
without radial arrows is of the form in~\eqref{e:PsiA}.
Thus the restriction of $\Psi$ induces an isomorphism between the subspace
of reduced power series in $A$ and the space
$\bigoplus_{i=1}^n
e_{2i-1} \! \left(\wh{KQ'_n}/\wh{KQ'_n}e_0\wh{KQ'_n}\right) \! e_{2i-1}$.
\end{remark}

\begin{prop} \label{p:Qn:Wf}
Let $f(x_1,\dots,x_n)$ be a power series without constant term and
consider the potential
$W = B_1 + \dots + B_n + f(\Delta_1, \dots, \Delta_n)$ on $Q_n$.
Then:
\begin{enumerate}
\renewcommand{\theenumi}{\alph{enumi}}
\item \label{it:Qn:Wf:mutk}
$\mu_k(Q_n, W) \simeq (Q_n, W)$ for any vertex $1 \leq k \leq 2n$ of $Q_n$.

\item \label{it:Qn:Wf:red}
$(Q_n, W)$ is right equivalent to $(Q_n, \overline{W})$, where
\[
\overline{W} = B_1 + \dots + B_n + \bar{f}(\Delta_1, \dots, \Delta_n)
\]
is the potential corresponding to the reduction $\bar{f}$ of $f$.

\item \label{it:Qn:Wf:mut0}
$
\mu_0(Q_n, W) \simeq
\left( Q'_n, \sum_{1 \leq i \neq j \leq n} a_i b_{ij} c_j + \Psi(\bar{f}) \right)
$
\end{enumerate}
\end{prop}
\begin{proof}
Claim~\eqref{it:Qn:Wf:mutk} is a consequence of Proposition~\ref{p:Qn:mutk}.
To show claim~\eqref{it:Qn:Wf:red}, let $1 \leq i \leq n$ and consider
all the terms in $f(\Delta_1, \dots, \Delta_n)$ containing 
(up to rotation) $\Delta_i^2$
but no $\Delta_j^2$ for any $j<i$. Rotate each of these terms such that it
starts in $\Delta_i^2$, so the sum of the rotated terms can be written
as $\Delta_i^2 S_i$ for some element $S_i$ in $e_0 \wh{KQ_n} e_0$.
Now the claim follows from Proposition~\ref{p:Qn:deform}
for $I = \{1,2,\dots,n\}$,
noting that $W$ is cyclically equivalent to
\[
\overline{W} + \sum_{i=1}^n \lambda_i \Delta_i + \sum_{i=1}^n \Delta_i^2 S_i
\]
where $\lambda_i$ is the coefficient of the monomial $x_i$ in $f$.

Finally, to prove claim~\eqref{it:Qn:Wf:mut0}, replace $W$ by the right
equivalent potential $\overline{W}$ and then use Proposition~\ref{p:Qn:mut0}.
\end{proof}

\subsection{Mutations at a side vertex of the quiver $X'_7$}
%%%%%%%%%%%%%%%%%%%%%%%%%%%%%%%%%%%%%%%%%%%%%%%%%%%%%%%%%%%%
\label{ssec:X7:side}

The quiver $X_7$ was discovered by Derksen-Owen~\cite{DerksenOwen08} as a new
quiver whose mutation class is finite, consisting of two members
$X_7$ and $X'_7$ which are shown in Figure~\ref{fig:X7}.
These are the quivers $Q_n$ and $Q'_n$, respectively,
of the previous sections with $n=3$.
In particular, Proposition~\ref{p:Qn:mutk} applies to the mutations
$\mu_k$ of certain potentials on the quiver
$X_7$ when $1 \leq k \leq 6$ (whose resulting quiver is isomorphic to $X_7$),
while Proposition~\ref{p:Qn:mut0} applies to their mutation $\mu_0$ taking
$X_7$ to $X'_7$.

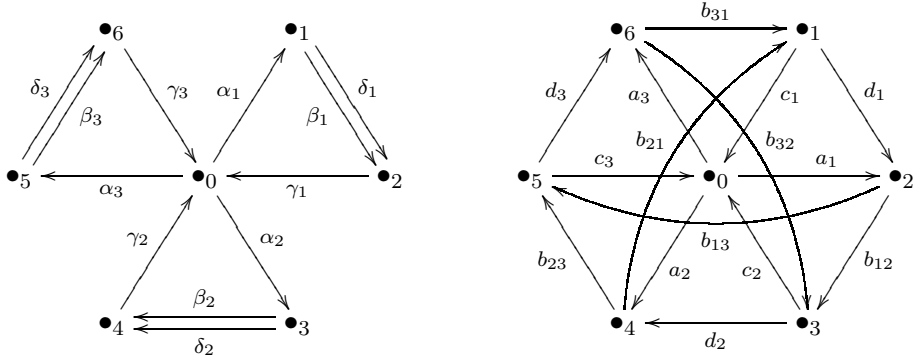
\begin{figure}
\begin{align*}
\xymatrix@=1.5pc{
& {\bullet_6} \ar[ddr]^{\gamma_3} & &
{\bullet_1} \ar@<0.5ex>[ddr]^{\delta_1} \ar@<-0.5ex>[ddr]_{\beta_1} \\
\\
{\bullet_5} \ar@<0.5ex>[uur]^{\delta_3} \ar@<-0.5ex>[uur]_{\beta_3} & &
{\bullet_0} \ar[ll]^{\alpha_3} \ar[uur]^{\alpha_1} \ar[ddr]^{\alpha_2} & &
{\bullet_2} \ar[ll]^{\gamma_1} \\
\\
& {\bullet_4} \ar[uur]^{\gamma_2} & & {\bullet_3} \ar@<0.5ex>[ll]^{\delta_2}
\ar@<-0.5ex>[ll]_{\beta_2}
}
&&
\xymatrix@=1.5pc{
& {\bullet_6} \ar[rr]^{b_{31}} \ar@/^1.5pc/[ddddrr]^{b_{32}} & &
{\bullet_1} \ar[ddr]^{d_1} \ar[ddl]^(0.4){c_1} \\
\\
{\bullet_5} \ar[uur]^{d_3} \ar[rr]^(0.4){c_3} & &
{\bullet_0} \ar[rr]^(0.6){a_1} \ar[ddl]^(0.6){a_2} \ar[uul]^(0.6){a_3} & &
{\bullet_2} \ar[ddl]^{b_{12}} \ar@/^1.5pc/[llll]^{b_{13}} \\
\\
& {\bullet_4} \ar[uul]^{b_{23}} \ar@/^1.5pc/[uuuurr]^{b_{21}} & &
{\bullet_3} \ar[ll]^{d_2} \ar[uul]^(0.4){c_2}
}
\end{align*}
\caption{The two quivers $X_7$ (left) and $X'_7$ (right)
in the mutation class of $X_7$.}
\label{fig:X7}
\end{figure}

It remains to consider the mutations $\mu_k$ of $X'_7$ for $1 \leq k \leq 6$.
While in this case $\mu_k(X'_7) \simeq X'_7$, the isomorphism sends
radial arrows to non-radial ones and vice versa, so its effect on a potential
is more complicated, as we shall see below.

Taking into account the symmetries of Lemma~\ref{l:Qn:opp}
and the discussion preceding it, it suffices to consider mutation at
one of the side vertices.
Without loss of generality, we consider mutation at the vertex $2$.
In this case,
the sets $C$ and $\Gamma$ are given by
$C = \{(a_1, b_{12}), (a_1, b_{13}), (d_1, b_{12}), (d_1, b_{13})\}$
and $\Gamma = \{c_2, c_3\}$,
and there is a unique maximal matching $\rho = (C', \Gamma',
C' \xrightarrow{\sim} \Gamma')$ with
$C' = \{(a_1, b_{12}), (a_1, b_{13})\}$ and $\Gamma'=\Gamma$,
encoded by the potential
\[
W_\rho = a_1 b_{12} c_2 + a_1 b_{13} c_3.
\]

The dual quiver $(X'_7)^*$ with respect to the matching $\rho$ is shown
in Figure~\ref{fig:X7mut2}, arranged in such a way to illustrate its
isomorphism with $X'_7$, to be discussed next.
The potential encoding the dual matching is
\[
W_{\rho^*} = [d_1 b_{12}] b_{12}^* d_1^* + [d_1 b_{13}] b_{13}^* d_1^*.
\]

\begin{figure}
\begin{align*}
\xymatrix@=1.5pc{
& {\bullet_5} \ar[rr]^{b^*_{13}} \ar@/^1.5pc/[ddddrr]^{d_3} & &
{\bullet_2} \ar[ddr]^{a_1^*} \ar[ddl]^(0.4){d^*_1} \\
\\
{\bullet_4} \ar[uur]^{b_{23}} \ar[rr]^(0.4){b_{21}} & &
{\bullet_1} \ar[rr]^(0.6){c_1} \ar[ddl]^(0.7){[d_1 b_{12}]}
\ar[uul]^(0.55){[d_1 b_{13}]} & &
{\bullet_0} \ar[ddl]^{a_3} \ar@/^1.5pc/[llll]^{a_2} \\
\\
& {\bullet_3} \ar[uul]^{d_2} \ar@/^1.5pc/[uuuurr]^{b^*_{12}} & &
{\bullet_6} \ar[ll]^{b_{32}} \ar[uul]^(0.3){b_{31}}
}
\end{align*}
\caption{The dual quiver $(X'_7)^*$ with respect to the maximal
matching $\{c_2 \leftrightarrow (a_1, b_{12}), 
c_3 \leftrightarrow (a_1, b_{13})\}$
corresponding to mutation at the vertex~$2$.}
\label{fig:X7mut2}
\end{figure}
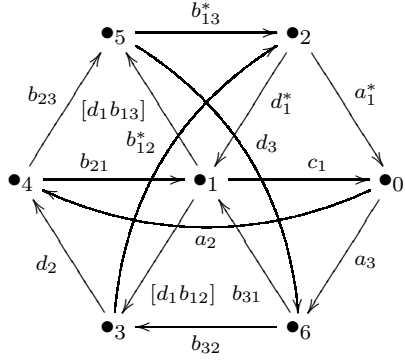

Comparing Figure~\ref{fig:X7} and Figure~\ref{fig:X7mut2}, we see:
\begin{lemma} \label{l:X7mut2:iso}
There is an isomorphism $\bar{\theta} \colon (X'_7)^* \to X'_7$
which is defined on the arrows as in Table~\ref{tab:X7:iso}.
\end{lemma}

\begin{table}
\begin{align*}
a_2 \mapsto b_{13} &&
b_{21} \mapsto c_3 &&
c_1 \mapsto a_1 &&
d_2 \mapsto b_{23} &&
a_1^* \mapsto d_1 &&
b_{12}^* \mapsto b_{21}
\\
a_3 \mapsto b_{12} &&
b_{23} \mapsto d_3 &&
[d_1 b_{12}] \mapsto a_2 &&
d_3 \mapsto b_{32} &&
d_1^* \mapsto c_1 &&
b_{13}^* \mapsto b_{31}
\\
&& b_{31} \mapsto c_2 &&
[d_1 b_{13}] \mapsto a_3
\\
&& b_{32} \mapsto d_2
\end{align*}
\caption{An isomorphism
$\bar{\theta} \colon (X'_7)^* \xrightarrow{\sim} X'_7$.}
\label{tab:X7:iso}
\end{table}

The map $\bar{\theta}$ defines an isomorphism from the complete
path algebra of $(X'_7)^*$ to that of $X'_7$, and
from now on we will identify these algebras via $\bar{\theta}$.
We will also consider the homomorphism $\theta$ from the complete path
algebra of the quiver $(X'_7)'$ (cf. Lemma~\ref{l:adm:Q'}) to that of
$X'_7$ which agrees with $\bar{\theta}$ on 
the arrows of $(X'_7)^*$ and its values on the remaining arrows are
\begin{align*}
\theta(c_2) = b_{21} d_1 &,& \theta(c_3) = b_{31} d_1 .
\end{align*}

\begin{prop} \label{p:X7mut2}
Let $S$ be a potential on $X'_7$ and assume that the arrow $a_1$ does
not appear in $S$. Let
$W = a_1 b_{12} c_2 + a_1 b_{13} c_3 + S$.
Then the mutation $\mu_2(X'_7, W)$ is isomorphic to $(X'_7, W')$
with
$W' = a_2 b_{21} c_1 + a_3 b_{31} c_1 + \theta([S])$.
\end{prop}
\begin{proof}
The condition that $a_1$ does not appear in $S$ is equivalent
to $S$ being $\rho$-admissible. Moreover, $C'=A' \star B'$
for the sets $A'=\{a_1\}$ and $B'=\varnothing$.
Hence Theorem~\ref{t:QPmut} and Proposition~\ref{p:CABsign} apply,
and the mutation $\mu_2(X'_7,W)$ is right equivalent to the potential
\[
[d_1 b_{12}] b^*_{12} d^*_1 + [d_1 b_{13}] b^*_{13} d^*_1 +
\sub_{\{c_2 \la b^*_{12} a^*_1, \, c_3 \la b^*_{13} a^*_1\}}([S])
\]
on $(X'_7)^*$. Now the claim follows by applying $\bar{\theta}$.
\end{proof}

Let $R=K\llangle x_1, x_2, x_3, y_1, y_2, y_3 \rrangle$ and let
$R' \subseteq R$ be its subspace of admissible power series,
cf.\ Definition~\ref{def:admiss}.
Consider an admissible power series $F \in R'$ and let
\begin{equation} \label{e:W:PsiF}
W_F = \sum_{1 \leq i \neq j \leq 3} a_i b_{ij} c_j + \Psi(F)
\end{equation}
be its associated potential on $X'_7$,
where $\Psi$ is the map defined in~\eqref{e:Psi}, cf.\ Definition~\ref{def:Psi}.

Before stating the next lemma,
we list a few values of $\theta$ for certain cycles of $(X'_7)'$:
\begin{align*}
&\theta(a_2 b_{21} c_1) = b_{13} c_3 a_1 ,&&
\theta(a_2 b_{23} c_3) = b_{13} d_3 b_{31} d_1 \\
&\theta(a_3 b_{31} c_1) = b_{12} c_2 a_1 ,&&
\theta(a_3 b_{32} c_2) = b_{12} d_2 b_{21} d_1 \\
&\theta([d_1 b_{12}] d_2 b_{21}) = a_2 b_{23} c_3 ,&&
\theta([d_1 b_{13}] d_3 b_{31}) = a_3 b_{32} c_2 ,&&
\theta(d_2 b_{23} d_3 b_{32}) = b_{23} d_3 b_{32} d_2
\end{align*}

\begin{lemma} \label{l:X7mut2:W0}
Let $F=x_1 x_2 + x_2 x_3 + x_3 x_1$. Then
$\mu_2(X'_7, W_F) \simeq (X'_7, W_F)$.
\end{lemma}
\begin{proof}
This is clear from Proposition~\ref{p:X7mut2} and the values of $\theta$
listed above.
\end{proof}

The rest of this section is devoted to understanding the effect of $\mu_2$
on potentials of the form $W_F$ for general admissible power series $F$.
We start by stating the conditions on $F$
for the resulting potential to be of the same form.

\begin{lemma} \label{l:WF:admiss}
$W_F$ is $\rho$-admissible if and only if
the letter $y_1$ does not appear in any term of $F$.
\end{lemma}
\begin{proof}
$W_F$ is $\rho$-admissible if and only if the arrow $a_1$ does not appear
in $\Psi(F)$. By the definition of $\Psi$,
any occurrence of $a_1$ in $\Psi(F)$ arises from that of the letter $y_1$
in $F$.
\end{proof}

\begin{defn}
Let $w$ be an admissible monomial of $R$ in which the letter $y_1$ does
not appear. 
We way that $w$ \emph{avoids $i1i$} if none of the words below
\begin{align*} %\tag{$\star$} \label{e:212313}
x_2 x_1 x_2 && x_2 x_1 y_2 && y_2 x_1 x_2 && y_2 x_1 y_2 &&
x_3 x_1 x_3 && x_3 x_1 y_3 && y_3 x_1 x_3 && y_3 x_1 y_3
\end{align*}
appears as a subword of $w$ or its rotations.
In other words, if $(i_1, i_2, \dots, i_r)$ is the sequence of indices of
the letters in $w$ then $(2,1,2)$ and $(3,1,3)$ do not appear in any of its rotations.

We say that an admissible power series \emph{avoids $i1i$} if
all its terms avoid $i1i$.
\end{defn}

\begin{lemma} \label{l:212313}
Assume that $W_F$ is $\rho$-admissible. Then
$\mu_2(X'_7, W_F)$ is isomorphic to a potential of the form $W_{F'}$
for some admissible power series $F'$
if and only if the following conditions hold:
\begin{enumerate}
\renewcommand{\theenumi}{\roman{enumi}}
\item
$F$ contains the terms $x_1 x_2$ and $x_3 x_1$ (up to rotation);

\item \label{it:212313}
Any term of $F$ avoids $i1i$.
\end{enumerate}
\end{lemma}
\begin{proof}
From Proposition~\ref{p:X7mut2} and the values of $\theta$ listed
above we see that in order for
all the triangles $a_i b_{ij} c_j$ to appear in $\mu_2(X'_7,W_F)$,
the terms
$d_1 b_{12} d_2 b_{21}$ and $d_1 b_{13} d_3 b_{31}$ must occur in $\Psi(F)$,
that is, $x_1 x_2$ and $x_3 x_1$ occur in $F$.

Now $\mu_2(X'_2, W_F)$ will be of the form $W_{F'}$ if any path
$c_j a_i$ of radial arrows with $i \neq j$ does not appear in any term
apart from the triangle $a_i b_{ij} c_j$ it belongs to. Looking back
at the definition of $\theta$, this means that the paths
$b_{31} c_1$ (whose image under $\theta$ is $c_2 a_1$),
$b_{31}[d_1 b_{13}]$ (whose image is $c_2 a_3$),
$b_{21} c_1$ (whose image is $c_3 a_1$)
and $b_{21}[d_1 b_{12}]$ (whose image is $c_3 a_2$)
do not appear in $[\Psi(F)]$.

Since $W_F$ is $\rho$-admissible, $F$ does not contain the letter $y_1$
and hence the arrow $c_1$ does not appear in $\Psi(F)$. We are left
to consider the paths $b_{21} d_1 b_{12}$ and $b_{31} d_1 b_{13}$ in $\Psi(F)$.
They will not appear exactly when the condition~\eqref{it:212313} holds.
\end{proof}

\begin{defn}
Let $w$ be an admissible monomial in $R'$. Assume that $y_1$ does not 
appear in $w$ and that $w$ avoids $i1i$.
Construct a new monomial $\theta_2(w)$ as follows.
Let $i_1, \dots, i_r$ be the sequence of indices of the letters in $w$.
For any place $1 \leq s \leq r$ such that $i_s \ne 1$, let $s'$ be the
previous place (in cyclic order) such that $i_{s'} \ne 1$. Then
either $s-s'=1$ or $s-s'=2$ (modulo $r$), since $w$ is admissible.

For each such $s$, let $i=i_s$. If $s-s'=1$, replace the $s$-th
letter of $w$ by a word of one or two letters according to the following:
\begin{align*}
x_i \mapsto x_i && y_i \mapsto -x_i x_1
\end{align*}
If $s-s'=2$, replace the $s$-th and its previous letter (which must
be $x_1$) as follows:
\begin{align*}
x_1 x_i \mapsto -y_i && x_1 y_i \mapsto y_i x_1
\end{align*}

Define $\theta_2(w)$ as the (signed) monomial obtained from $w$ after
going over all the places $s$ with $i_s \neq 1$ and performing these
replacements.
\end{defn}

\begin{example}
If $w=x_2 x_3$ then we need to go over the places $1,2$ and
$\theta_2(w) = x_2 x_3$.
If $w=x_1 x_2 x_3$ then we need to go over the places $2,3$ and
$\theta_2(w) = -y_2 x_3$.
Similarly, $\theta_2(x_3 x_2 x_1) = -y_3 x_2$
(going over the places $1,2$).
\end{example}

\begin{lemma}
$\theta_2(w)$ is always admissible and avoids $i1i$.
\end{lemma}
\begin{proof}
This follows from the fact that we always have $i_s \neq i_{s'}$.
If $s-s'=1$ this is true since $w$ is admissible, while if $s-s'=2$ this
holds since $w$ avoids $i1i$.
\end{proof}

\begin{lemma} \label{l:Psi:theta}
$\Psi(\theta_2(w)) = \theta([\Psi(w)])$.
\end{lemma}
\begin{proof}
We will consider $\Psi(w)$ as a concatenation of certain (signed) paths
not starting or ending at the vertex $2$
and evaluate $\theta$ after applying $[-]$ on each such path.

Let $i_1, \dots, i_r$ be the sequence of indices of the letters of $w$,
let $s, s'$ be as in the construction of $\theta_2(w)$ and
let $i=i_s$, $i'=i_{s'}$. Then by our assumptions $i \neq i'$
and they both differ from $1$.

If $s-s'=1$ then in $\Psi(w)$ we either have the path $b_{i'i} d_i$
(arising if the $s$-th letter of $w$ is $x_i$) or the path
$-b_{i'i} c_i a_i$ (arising if it is $y_i$). Then
\begin{align*}
\theta(b_{i'i} d_i) = d_i b_{ii'} &&
\theta(-b_{i'i} c_i a_i) = - d_i b_{i1} d_1 b_{1i'}
\end{align*}
corresponding to the words $x_i$ and $-x_i x_1$, respectively.

If $s-s'=2$ then in $\Psi(w)$ we either have the path
$b_{i'1} d_1 b_{1 i} d_i$
(arising from the word $x_1 x_i$ in $w$) or the path
$-b_{i'i} d_1 b_{1 i} c_i a_i$ (arising from the word $x_1 y_i$). Then
\begin{align*}
\theta(b_{i'1} [d_1 b_{1 i}] d_i) = c_i a_i b_{ii'} &&
\theta(-b_{i'1} [d_1 b_{1 i}] c_i a_i) = - c_i a_i b_{i1} d_1 b_{1i'}
\end{align*}
corresponding to the words $-y_i$ and $y_i x_1$, respectively.
\end{proof}

We can extend the map $\theta_2$ to the space of admissible power series
avoiding $i1i$ using linearity and continuity. The next proposition
expresses the result of the mutation $\mu_2$ on a potential $W_F$
in terms of the map $\theta_2$.

\begin{prop} \label{p:X7mut2:G}
Let $G$ be an admissible power series which avoids $i1i$
and let $F = x_1 x_2 + x_3 x_1 + G$.
%let $F' = x_1 x_2 + x_3 x_1 + \theta_2(G)$.
Then $\mu_2(X'_7, W_F) \simeq (X'_7, W_{F'})$
for $F' = x_1 x_2 + x_3 x_1 + \theta_2(G)$.
\end{prop}
\begin{proof}
By Lemma~\ref{l:WF:admiss}, $W_F$ is $\rho$-admissible, so by
arguing as in the proof of Lemma~\ref{l:212313}
using Proposition~\ref{p:X7mut2} and the list of values of $\theta$
preceding Lemma~\ref{l:X7mut2:W0}, we see that
$\mu_2(X'_7, W_F) = (X'_7, W_{F'})$ for $F'= x_1 x_2 + x_3 x_1 + G'$,
where $G'$ is an admissible power series such that
$\Psi(G') = \theta([\Psi(G)])$.
The result now follows from Lemma~\ref{l:Psi:theta}.
\end{proof}

The preceding proposition can be used to construct a family of potentials
on $X'_7$ whose mutations at all the side vertices are completely understood.
For convenience, the result is stated in terms of potentials on the quiver
$X_7$.

\begin{prop} \label{p:X7mut2:P}
Let $P_{123}, P_{321} \in K[[x]]$ be two power series without constant
term and consider the potential on $X_7$ given by
\[
W = B_1 + B_2 + B_3 +
\Delta_1 \Delta_2 + \Delta_2 \Delta_3 + \Delta_3 \Delta_1 +
P_{123}(\Delta_1 \Delta_2 \Delta_3) + P_{321}(\Delta_3 \Delta_2 \Delta_1).
\]

Then $\mu_2 \mu_0(X_7, W) \simeq \mu_0(X_7, \wt{W})$
for the potential $\wt{W}$ on $X_7$ given by
\[
\wt{W} = B_1 + B_2 + B_3 +
\Delta_1 \Delta_2 + \Delta_2 \Delta_3 + \Delta_3 \Delta_1 +
P_{123}(-B_2 \Delta_3) + P_{321}(-B_3 \Delta_2).
\]
\end{prop}
\begin{proof}
We can write $W = B_1 + B_2 + B_3 + \Phi(f)$ for the power series
$f(x_1, x_2, x_3) = x_1 x_2 + x_3 x_1 + g(x_1, x_2, x_3)$,
where
\[
g(x_1, x_2, x_3) = x_2 x_3 + P_{123}(x_1 x_2 x_3)
+ P_{321}(x_3 x_2 x_1) .
\]

Write also $\wt{W} = B_1 + B_2 + B_3 + \Phi(F)$ for the power series
$F(x_1, x_2, x_3, y_1, y_2, y_3) = x_1 x_2 + x_3 x_1 +
G(x_1, x_2, x_3, y_1, y_2, y_3)$, where
\[
G(x_1, x_2, x_3, y_1, y_2, y_3) = x_2 x_3 + P_{123}(-y_2 x_3)
+ P_{321}(-y_3 x_2) .
\]

Since each of the monomials $(x_1 x_2 x_3)^n$ and $(x_3 x_2 x_1)^n$ avoids
$i1i$, so does $g$, and moreover
$\theta_2(x_2 x_3) = x_2 x_3$, $\theta_2((x_1 x_2 x_3)^n) = (-y_2 x_3)^n$
and $\theta_2((x_3 x_2 x_1)^n) = (-y_3 x_2)^n$, hence
$\theta_2(g) = G$.

By Proposition~\ref{p:Qn:Wf},
$\mu_0(X_7, W) = (X'_7, \sum_{1 \leq i \ne j \leq 3} a_i b_{ij} c_j
+ \Psi(f))$. Applying Proposition~\ref{p:X7mut2:G} and using
$\theta_2(g)=G$, we get
$\mu_2 \mu_0(X_7, W) = (X'_7, \sum_{1 \leq i \ne j \leq 3} a_i b_{ij} c_j
+ \Psi(F))$,
which equals $\mu_0(X_7, \wt{W})$ by
Proposition~\ref{p:Qn:mut0}.
\end{proof}

\begin{remark} \label{rem:X7mutk:P}
When performing a mutation at a general side vertex, we
need to replace the terms $-B_2 \Delta_3$ and $-B_3 \Delta_2$
in Proposition~\ref{p:X7mut2:P}
by $-B_i \Delta_j$ and $-B_j \Delta_i$, respectively,
for some $1 \leq i \neq j \leq 3$ (there are $6$ possibilities
corresponding to the $6$ vertices).
\end{remark}

\begin{remark}
In the special case where both $P_{123}$ and $P_{321}$ are equal to zero
we have $\wt{W}=W$, which is also a direct consequence of
Lemma~\ref{l:X7mut2:W0}.
\end{remark}

\subsection{Some right equivalences of potentials on $X_7$}
%%%%%%%%%%%%%%%%%%%%%%%%%%%%%%%%%%%%%%%%%%%%%%%%%%%%%%%%%%%
\label{ssec:X7:righteq}

In this section we consider some right equivalences between
potentials on $X_7$ of the form
\[
B_1 + B_2 + B_3 + f(\Delta_1, \Delta_2, \Delta_3)
\]
having ``pure'' terms which are monomials in the $\Delta_i$ only,
and potentials of the form
\[
B_1 + B_2 + B_3 + F(\Delta_1, \Delta_2, \Delta_3, B_1, B_2, B_3)
\]
having ``mixed'' terms involving both $B_i$ and $\Delta_j$.

We record a few useful observations concerning evaluations of power series.
Let $Q$ be a quiver and $k$ a vertex. Denote by $\fm$ the ideal of $\wh{KQ}$
generated by the arrows.
If $F(x_1,\dots,x_n) \in K \llangle x_1,\dots,x_n \rrangle$ is a power series
without constant term, we can substitute any sequence of potentials
$\omega_1, \dots, \omega_n \in e_k \fm^2 e_k$ and
get a potential $F(\omega_1, \dots, \omega_n)$.
If $F$ and $F'$ are two power series whose difference
$F-F'$ lies in the closure of the subspace spanned by the commutators,
then the potentials $F(\omega_1, \dots, \omega_n)$ and
$F'(\omega_1, \dots, \omega_n)$ are cyclically equivalent.
If $S_1, \dots, S_n, \omega \in e_k \fm e_k$ and
$\omega'_i = \omega_i + S_i \omega$ for all $1 \leq i \leq n$,
then there exists $S \in e_k \fm e_k$ such that
$F(\omega'_1, \dots, \omega'_n)$ and $F(\omega_1, \dots, \omega_n) + S \omega$
are cyclically equivalent.

\begin{notat}
In this section, $i, j, k$ will denote three different indices
from the set $\{1,2,3\}$.
\end{notat}

\begin{lemma} \label{l:FhG:jk}
Let $F(x,y) \in K \llangle x, y \rrangle$ be a power series without constant
term. Let $h(x,y) \in K \llangle x, y \rrangle$ and set
$G(x,y) = F(x,y) + (y-x)h(x,y)$.
Then the potentials
\begin{align}
\label{e:potF}
B_1 + B_2 + B_3 +
\Delta_1 \Delta_2 + \Delta_2 \Delta_3 + \Delta_3 \Delta_1 +
F(-B_j \Delta_k, \Delta_i \Delta_j \Delta_k)
\intertext{and}
\label{e:potG}
B_1 + B_2 + B_3 +
\Delta_1 \Delta_2 + \Delta_2 \Delta_3 + \Delta_3 \Delta_1 +
G(-B_j \Delta_k, \Delta_i \Delta_j \Delta_k)
\end{align}
are right equivalent.
\end{lemma}
\begin{proof}
Apply the unitriangular right equivalence
$\vphi = \sub_{\gamma_j \la \gamma_j +
\gamma_j \Delta_k h(-B_j \Delta_k, \Delta_i \Delta_j \Delta_k)}$
taking $\gamma_j$ to
$ \gamma_j + \gamma_j \Delta_k h(-B_j \Delta_k, \Delta_i \Delta_j \Delta_k)$
and fixing all the other arrows, cf.\ Lemma~\ref{l:sub:gTauto:m2}.
Then
\begin{align*}
\vphi(B_j) &= B_j + B_j \Delta_k h(-B_j \Delta_k, \Delta_i \Delta_j \Delta_k)
\\
\vphi(\Delta_j) &= 
\Delta_j + \Delta_j \Delta_k h(-B_j \Delta_k, \Delta_i \Delta_j \Delta_k) .
\end{align*}

Observe that every term of
$\Delta_k h(-B_j \Delta_k, \Delta_i \Delta_j \Delta_k)$ ends in
the letter $\Delta_k$. To verify this, consider the monomials
appearing in $h$. If a monomial has positive degree, then substituting
$-B_j \Delta_k$ and $\Delta_i \Delta_j \Delta_k$ results in a word ending
at $\Delta_k$, while for the monomial of degree zero, the corresponding
term is a scalar multiple of $\Delta_k$.
Hence there exist potentials $S$ and $S'$ such that
we can write
\begin{align*}
\vphi(B_j \Delta_k) = B_j \Delta_k + S \Delta_k^2 &,&
\vphi(\Delta_j \Delta_k) = \Delta_j \Delta_k + S' \Delta_k^2 ,
\end{align*}
so
$\vphi(F(-B_j \Delta_k, \Delta_i \Delta_j \Delta_k)) =
F(\vphi(-B_j \Delta_k), \vphi(\Delta_i \Delta_j \Delta_k))$
is cyclically equivalent to
\[
F(-B_j \Delta_k, \Delta_i \Delta_j \Delta_k) + S'' \Delta_k^2
\]
for a suitable potential $S''$,
by the preceding remarks on evaluations.

Evaluating $\vphi$ on the potential in~\eqref{e:potF}, we therefore get, up
to cyclic equivalence,
\begin{align*}
\vphi(&B_k + B_i + B_j +
\Delta_k \Delta_i + \Delta_j \Delta_k + \Delta_i \Delta_j +
F(-B_j \Delta_k, \Delta_i \Delta_j \Delta_k)) = \\
& B_k + B_i + B_j + B_j \Delta_k  h(-B_j \Delta_k, \Delta_i \Delta_j \Delta_k)
+ \Delta_k \Delta_i
+ \Delta_j \Delta_k + S' \Delta_k^2 \\
&+ \Delta_i \Delta_j
+ \Delta_i \Delta_j \Delta_k  h(-B_j \Delta_k, \Delta_i \Delta_j \Delta_k) 
+ F(-B_j \Delta_k, \Delta_i \Delta_j \Delta_k) + S'' \Delta_k^2
\end{align*}
which is right equivalent to the potential in~\eqref{e:potG}
by Proposition~\ref{p:Qn:deform} with $I=\{k\}$, since the arrow $\beta_k$
does not appear in any term of~\eqref{e:potG} except $B_k$.
\end{proof}

We also have the following dual statement.
\begin{lemma} \label{l:FhG:ij}
In the notations of Lemma~\ref{l:FhG:jk}, the potentials
\begin{align*}
B_1 + B_2 + B_3 +
\Delta_1 \Delta_2 + \Delta_2 \Delta_3 + \Delta_3 \Delta_1 +
F(-\Delta_i B_j, \Delta_i \Delta_j \Delta_k)
\intertext{and}
B_1 + B_2 + B_3 +
\Delta_1 \Delta_2 + \Delta_2 \Delta_3 + \Delta_3 \Delta_1 +
G(-\Delta_i B_j, \Delta_i \Delta_j \Delta_k)
\end{align*}
are right equivalent.
\end{lemma}
\begin{proof}
The proof is completely analogous to that of the previous lemma,
this time using the right equivalence sending $\alpha_j$
to $\alpha_j + h(-\Delta_i B_j, \Delta_i \Delta_j \Delta_k) \Delta_i \alpha_j$
and fixing all other arrows.
Alternatively, this can also be deduced from the previous lemma
by exchanging the roles of $i$ and $k$ and working over the opposite quiver.
\end{proof}

Let $P(x) \in K[[x]]$ be a power series.
We can consider two power series of two variables built from $P$,
one is $F(x,y) = P(x)$ (so $y$ is a dummy variable) and the other is
$G(x,y) = P(y)$ (so $x$ is a dummy variable).

\begin{lemma} \label{l:P:xy}
There exists $h \in K \llangle x,y \rrangle$ such that
$P(y) - P(x) - (y-x)h(x,y)$ lies in the closure of the subspace 
of $K \llangle x,y \rrangle$ spanned by the commutators.
\end{lemma}
\begin{proof}
Write $P(x) = \sum_{n \geq 1} c_n x^n$, let
$h(x,y) = \sum_{i,j \geq 0} c_{i+j+1} y^j x^i \in K \llangle x,y \rrangle$
and arrange the terms of $h$ according to their total degree
$h(x,y) = \sum_{n \geq 1} c_n \sum_{i=1}^{n} y^{n-i} x^{i-1}$.

Since for any $n \geq 1$ we have the equality
\[
(y-x)(y^{n-1} + y^{n-2} x + \dots + y x^{n-2} + x^{n-1})
= y^n + \sum_{i=1}^{n-1} y^{n-i} x^i
- \sum_{i=1}^{n-1} x y^{n-i} x^{i-1} - x^n
\]
and the right hand side equals $y^n-x^n$ 
modulo the subspace spanned by the commutators,
we deduce that $P(y) - P(x)$ is equal to $(y-x)h(x,y)$
modulo the closure of that subspace.
\end{proof}

\begin{lemma} \label{l:P:BD}
Let $P(x)$ be a power series without constant term.
Then the potentials
\begin{align*}
&B_1 + B_2 + B_3 +
\Delta_1 \Delta_2 + \Delta_2 \Delta_3 + \Delta_3 \Delta_1 +
P(\Delta_i \Delta_j \Delta_k)
\\
&B_1 + B_2 + B_3 +
\Delta_1 \Delta_2 + \Delta_2 \Delta_3 + \Delta_3 \Delta_1 +
P(-B_j \Delta_k)
\\
&B_1 + B_2 + B_3 +
\Delta_1 \Delta_2 + \Delta_2 \Delta_3 + \Delta_3 \Delta_1 +
P(-\Delta_i B_j)
\end{align*}
are right equivalent.
\end{lemma}
\begin{proof}
Combine Lemma~\ref{l:FhG:jk}, Lemma~\ref{l:FhG:ij} and Lemma~\ref{l:P:xy}.
\end{proof}

As a consequence we get the following right equivalence.
\begin{cor}
The potentials
\begin{equation} \label{e:P123}
B_1 + B_2 + B_3 +
\Delta_1 \Delta_2 + \Delta_2 \Delta_3 + \Delta_3 \Delta_1 +
P(\Delta_1 \Delta_2 \Delta_3)
\end{equation}
and
\begin{equation} \label{e:P321}
B_1 + B_2 + B_3 +
\Delta_1 \Delta_2 + \Delta_2 \Delta_3 + \Delta_3 \Delta_1 +
P(\Delta_3 \Delta_2 \Delta_1)
\end{equation}
are right equivalent.
\end{cor}
\begin{proof}
By Lemma~\ref{l:P:BD} with $(i,j,k)=(1,2,3)$,
the potential~\eqref{e:P123} is right equivalent to
\[
B_1 + B_2 + B_3 +
\Delta_1 \Delta_2 + \Delta_2 \Delta_3 + \Delta_3 \Delta_1 +
P(-B_2 \Delta_3)
\]
and by the same lemma with $(i,j,k)=(3,2,1)$,
the potential~\eqref{e:P321} is right equivalent to
\[
B_1 + B_2 + B_3 +
\Delta_1 \Delta_2 + \Delta_2 \Delta_3 + \Delta_3 \Delta_1 +
P(-\Delta_3 B_2).
\]

These potentials are cyclically equivalent since every
monomial $(-B_2 \Delta_3)^n$ is rotationally equivalent to
the corresponding monomial $(-\Delta_3 B_2)^n$.
\end{proof}

\subsection{Explicit non-degenerate potentials on $X_7$}
%%%%%%%%%%%%%%%%%%%%%%%%%%%%%%%%%%%%%%%%%%%%%%%%%%%%%%%%
\label{ssec:X7:potential}

We are now ready to state and prove the main result of this section.
We keep the notations from the previous section.

\begin{theorem} \label{t:X7:potential}
Let $P(x) \in K[[x]]$ be a power series without constant term.
Then the potential on the quiver $X_7$ given by
\[
W = B_1 + B_2 + B_3
+ \Delta_1 \Delta_2 + \Delta_2 \Delta_3 + \Delta_3 \Delta_1
+ P(\Delta_1 \Delta_2 \Delta_3)
\]
is non-degenerate.
\end{theorem}
\begin{proof}
We will prove that there is a potential $W'$ on the quiver $X'_7$ such
that
\begin{align*}
\mu_0(X_7, W) \simeq (X'_7, W') &,&
\mu_k(X_7, W) \simeq (X_7, W) &,&
\mu_k(X'_7, W') \simeq (X'_7, W')
\end{align*}
for any vertex $1 \leq k \leq 6$.
Since each of the quivers $X_7$ and $X'_7$ does not have any $2$-cycle,
this will imply that $W$ is non-degenerate.

The potential $W$ can be written as
$B_1 + B_2 + B_3 + f(\Delta_1, \Delta_2, \Delta_3)$
for the power series
$f(x_1,x_2,x_3) = x_1 x_2 + x_2 x_3 + x_3 x_1 + P(x_1 x_2 x_3)$.
Since $f$ is reduced, we can apply
Proposition~\ref{p:Qn:Wf} and deduce that
$\mu_k(X_7, W) \simeq (X_7, W)$ for any vertex $1 \leq k \leq 6$,
while $\mu_0(X_7, W) = (X'_7, W')$ for the potential $W'$ given by
\[
W' = \sum_{1 \leq i \neq j \leq 3} a_i b_{ij} c_j +
\sum_{1 \leq i < j \leq 3} d_i b_{ij} d_j b_{ji}
+ P(d_1 b_{12} d_2 b_{23} d_3 b_{31}).
\]

It remains to consider the mutations $\mu_k(X'_7,W')$ for $1 \leq k \leq 6$.
By Proposition~\ref{p:X7mut2:P} and the remark following it,
$\mu_k(X'_7, W') \simeq \mu_0(X_7, \wt{W})$ for a potential $\wt{W}$
of the form
\[
\wt{W} = B_1 + B_2 + B_3
+ \Delta_1 \Delta_2 + \Delta_2 \Delta_3 + \Delta_3 \Delta_1
+ P(-B_i \Delta_j)
\]
for some $1 \leq i \ne j \leq 3$ (depending on $k$). By Lemma~\ref{l:P:BD},
the potentials $W$ and $\wt{W}$ are right equivalent, hence their
mutations at the vertex $0$ are also right equivalent. It follows that
$\mu_k(X'_7, W')$ is right equivalent to $(X'_7, W')$.
\end{proof}

\section{Potentials on $X_7$ whose Jacobian algebras are finite-dimensional}
%%%%%%%%%%%%%%%%%%%%%%%%%%%%%%%%%%%%%%%%%%%%%%%%%%%%
\label{sec:Jac:fdim}

\subsection{Preliminaries}
%%%%%%%%%%%%%%%%%%%%%%%%%%

Let $Q$ be a quiver. For any arrow $\alpha$ of $Q$ there is a 
\emph{cyclic derivative} map
$\partial_\alpha \colon \HH_0(\wh{KQ}) \to \wh{KQ}$
which is the unique continuous linear map whose value on each
cycle $\alpha_1 \alpha_2 \dots \alpha_n$ is given by
\[
\partial_{\alpha}(\alpha_1 \alpha_2 \dots \alpha_n) =
\sum_{i \,:\, \alpha_i = \alpha}
\alpha_{i+1} \dots \alpha_n \alpha_1 \dots \alpha_{i-1}
\]
where the sum goes over all the indices $1 \leq i \leq n$
such that $\alpha_i=\alpha$.

\begin{defn}[\protect{\cite{DWZ08}}]
The \emph{Jacobian algebra} $\cP(Q,W)$ of a quiver with potential $(Q,W)$
is the quotient of the complete path algebra of $Q$ by the 
\emph{Jacobian ideal} $\cJ(W)$, which is the closure of the
(two-sided) ideal generated by the cyclic derivatives of $W$ with respect
to the arrows of $Q$,
\[
\cP(Q,W) = \wh{KQ} / \cJ(W) = 
\wh{KQ}/ \overline{(\partial_\alpha(W) \,:\, \alpha \in Q_1)} .
\]
\end{defn}

We recall two useful facts concerning Jacobian algebras:
\begin{itemize}
\item
Right equivalent potentials have isomorphic Jacobian
algebras~\cite[Proposition~3.7]{DWZ08}

\item
The finite-dimensionality of Jacobian algebras is preserved under mutations,
i.e.\ if $\cP(Q,W)$ is finite-dimensional, then so is
$\cP(\mu_k(Q,W))$~\cite[Corollary~6.6]{DWZ08}.
\end{itemize}
Thus, in order to show that two potentials on a given quiver are not
right equivalent, it suffices to show that their Jacobian algebras are not
isomorphic.

\subsection{Reduction to the central vertex}
%%%%%%%%%%%%%%%%%%%%%%%%%%%%%%%%%%%%%%%%%%%%

Let $Q$ be a quiver. If $I \subseteq Q_0$ is a set of vertices, 
recall that the \emph{full subquiver of $Q$ on $I$} is the quiver whose
set of vertices is $I$ and its set of arrows consists of the
arrows in $Q$ whose both endpoints lie in $I$.
Denote by $e_I = \sum_{i \in I} e_i$ the corresponding idempotent
in $\wh{KQ}$. By abuse of notation, $e_I$ will also denote the
image of this idempotent in any quotient of $\wh{KQ}$.

\begin{lemma} \label{l:eIfd}
Let $Q$ be a quiver and $I \subseteq Q_0$ be a subset of its vertices
such that the full subquiver of $Q$ on $Q_0 \setminus I$
is acyclic.
Then for any quotient $\gL$ of $\wh{KQ}$, we have 
$\dim_K \gL < \infty$ if and only if
$\dim_K e_I \gL e_I < \infty$.
\end{lemma}
\begin{proof}
It is clear that if $\dim_K \gL  < \infty$ then
$\dim_K e_I \gL e_I < \infty$. For the converse implication,
it suffices to find a finite set of paths in $Q$ whose images span $\gL$.

For a path $p=\alpha_1 \alpha_2 \dots \alpha_n$ in $Q$, let
$(i_0, i_1, \dots, i_n)$ denote the sequence of vertices it traverses,
so that $\alpha_j$ is an arrow $i_{j-1} \to i_j$ for every
$1 \leq j \leq n$. Define four subsets of paths 
$\cP_I, \cP^-, \cP^+, \cP'$ as follows:
\begin{align*}
& p \in \cP_I \quad \text{ if $i_0, i_n \in I$,} \\
& p \in \cP^+ \quad \text{if $i_0 \in I$ but $i_j \not \in I$
for any $0 < j \leq n$,} \\
& p \in \cP^- \quad \text{if $i_n \in I$ but $i_j \not \in I$ 
for any $0 \leq j < n$,} \\
& p \in \cP' \quad \text{ if $i_j \not \in I$ for every $0 \leq j \leq n$.}
\end{align*}

Since the full subquiver of $Q$ on $Q_0 \setminus I$ is acyclic, the sets
$\cP'$, $\cP^+$ and $\cP^-$ are finite. Moreover, since $e_I \gL e_I$ is
finite-dimensional, there exists a finite subset $\cP'_I \subseteq \cP_I$
such that the image in $\gL$ of any path in $\cP_I$ is a linear combination
of images of paths in $\cP'_I$.
Any path in $Q$ is either in $\cP'$ or can be written as 
$p_- u p_+$ for some $p_- \in \cP^-$, $u \in \cP_I$ and $p_+ \in \cP^+$.
Hence $\gL$ is spanned by the images of the paths in the finite set
\[
\cP' \cup
\left\{ \text{paths of the form $p_- u p_+$ with
$p_- \in \cP^-$, $p_+ \in \cP^+$, $u \in \cP'_I$} \right\} .
\]
\end{proof}

Observe that if $I = Q_0 \setminus \{k\}$ is the complement of a singleton,
we recover Lemma~6.5 of~\cite{DWZ08}, as in this case the condition
on $Q_0 \setminus I$ in Lemma~\ref{l:eIfd} means that there are no loops at $k$.
We will be interested in the other extreme, namely when $I$ consists of a
single vertex. This is recorded in the next lemma. 

\begin{lemma} \label{l:ekfd}
Let $Q$ be a quiver and let $k \in Q_0$ such that the full subquiver on
$Q_0 \setminus \{k\}$ is acyclic. Then for any quotient $\gL$
of $\wh{KQ}$, $\dim_K \gL < \infty$ if and only if
$\dim_K e_k \gL e_k < \infty$.
\end{lemma}
\begin{proof}
Take $I = Q_0 \setminus \{k\}$ in Lemma~\ref{l:eIfd}.
\end{proof}

Applying this to the quiver $Q_n$ and its central vertex $0$, we get:
\begin{lemma} \label{l:KQn0fd}
For any quotient $\gL$ of $\wh{KQ_n}$,
$\dim_K \gL < \infty$ if and only if $\dim_K e_0 \gL e_0 < \infty$.
\end{lemma}
\begin{proof}
Combine Lemma~\ref{l:Qn:comb}\eqref{it:acyclic} and Lemma~\ref{l:ekfd}.
\end{proof}

\subsection{Calculation of $e_0 \gL e_0$}
%%%%%%%%%%%%%%%%%%%%%%%%%%%%%%%%%%%%%%%%%
Let $f(x_1,x_2,\dots,x_n)$ be a power series without constant term
in $n$ non-commuting variables $x_1, x_2, \dots, x_n$.
In this section we compute $e_0 \gL e_0$ for the Jacobian algebra
$
\gL = \cP(Q_n, B_1 + \dots + B_n + f(\Delta_1, \dots, \Delta_n)).
$
Note that $f$ can be viewed as a potential on a quiver with one vertex
and $n$ loops $x_1, \dots, x_n$.

\begin{lemma} \label{l:Wder}
Let $W = \sum_{i=1}^n B_i + f(\Delta_1, \dots, \Delta_n)$. Then
\begin{align*}
\partial_{\alpha_i}(W) &= 
\beta_i \gamma_i +
\delta_i \gamma_i (\partial_{x_i} f)(\Delta_1, \dots, \Delta_n),
&
\partial_{\beta_i}(W) &= \gamma_i \alpha_i ,
\\
\partial_{\gamma_i}(W) &= 
\alpha_i \beta_i +
(\partial_{x_i} f)(\Delta_1, \dots, \Delta_n) \alpha_i \delta_i,
&
\partial_{\delta_i}(W) &=
\gamma_i (\partial_{x_i} f)(\Delta_1, \dots, \Delta_n) \alpha_i
\end{align*}
for any $1 \leq i \leq n$.
\end{lemma}
\begin{proof}
We illustrate for the cyclic derivatives with respect to the arrows
$\alpha_i$, the other cases being similar.
By linearity and continuity, it is enough to compute these derivatives
for the cycles $B_j$ and the monomials
$\Delta_{i_1} \Delta_{i_2} \dots \Delta_{i_r}$.

Indeed, $\partial_{\alpha_i}(B_j)$ equals $\beta_i \gamma_i$ if $j=i$
and vanishes otherwise, and
\begin{align*}
\partial_{\alpha_i}(\Delta_{i_1} \Delta_{i_2} \dots \Delta_{i_r})
&= \sum_{s \,:\, i_s = i} \delta_i \gamma_i \Delta_{i_{s+1}} \dots
\Delta_{i_r} \Delta_{i_1} \dots \Delta_{i_{s-1}}
\\
&= \delta_i \gamma_i \cdot
\partial_{x_i} (x_{i_1} x_{i_2} \dots x_{i_r})(\Delta_1, \dots, \Delta_n) .
\end{align*}
\end{proof}

\begin{lemma} \label{l:e0Le0}
Let $W = \sum_{i=1}^n B_i + f(\Delta_1, \dots, \Delta_n)$
and consider $\gL = \cP(Q_n, W)$. Then
\begin{equation} \label{e:e0Le0}
e_0 \gL e_0 \simeq
K \langle \! \langle x_1, \dots, x_n \rangle \! \rangle / 
\overline{\bigl( x_i^2, [x_i, \partial_{x_i} f] \bigr)}_{1 \leq i \leq n} .
\end{equation}
\end{lemma}
\begin{proof}
The algebra $e_0 \gL e_0$ is the quotient of the subalgebra
$e_0 \wh{KQ_n} e_0$ of $\wh{KQ_n}$
by the ideal $J = \cJ(W) \cap e_0 \wh{KQ_n} e_0$.
The space $e_0 \wh{KQ_n} e_0$ is the closure of the subspace of
$\wh{KQ_n}$ spanned by all the cycles starting at $0$.
By Lemma~\ref{l:Qn:comb}\eqref{it:words},
these cycles are precisely the words in the letters $B_i$, $\Delta_i$.

The ideal $J$ is (topologically) generated by the elements
\begin{align*}
& \alpha_i \cdot \partial_{\alpha_i}(W), \\
& \partial_{\gamma_i}(W) \cdot \gamma_i, \\
& \alpha_i \delta_i \cdot \partial_{\beta_i}(W) \cdot \delta_i \gamma_i, &&
\alpha_i \delta_i \cdot \partial_{\beta_i}(W) \cdot \beta_i \gamma_i, &&
\alpha_i \beta_i \cdot \partial_{\beta_i}(W) \cdot \delta_i \gamma_i, &&
\alpha_i \beta_i \cdot \partial_{\beta_i}(W) \cdot \beta_i \gamma_i
\\
& \alpha_i \delta_i \cdot \partial_{\delta_i}(W) \cdot \delta_i \gamma_i, &&
\alpha_i \delta_i \cdot \partial_{\delta_i}(W) \cdot \beta_i \gamma_i, &&
\alpha_i \beta_i \cdot \partial_{\delta_i}(W) \cdot \delta_i \gamma_i, &&
\alpha_i \beta_i \cdot \partial_{\delta_i}(W) \cdot \beta_i \gamma_i
\end{align*}
for $1 \leq i \leq n$.
Since, by Lemma~\ref{l:Wder},
$\alpha_i \beta_i + \omega_i \alpha_i \delta_i$
and $\beta_i \gamma_i + \delta_i \gamma_i \omega_i$ belong to $\cJ(W)$ for
$\omega_i = (\partial_{x_i}f)(\Delta_1, \dots, \Delta_n) \in
e_0 \wh{KQ_n} e_0$,
it suffices to take only the first element in each of the last two rows
of generators.

Computing using Lemma~\ref{l:Wder}
the generators appearing in the first three rows, we get:
\begin{align*}
\alpha_i \cdot \partial_{\alpha_i}(W)
&= B_i + \Delta_i \cdot (\partial_{x_i} f)(\Delta_1, \dots, \Delta_n)
\\
\partial_{\gamma_i}(W) \cdot \gamma_i
&= B_i + (\partial_{x_i} f)(\Delta_1, \dots, \Delta_n) \cdot \Delta_i
\\
\alpha_i \delta_i \cdot \partial_{\beta_i}(W) \cdot \delta_i \gamma_i
&= \Delta_i \Delta_i ,
\end{align*}
hence 
the image of any $B_i$ modulo $J$ is a power series in 
$\Delta_1, \dots, \Delta_n$, and we can view $e_0 \gL e_0$ as
a quotient of
$K \langle \! \langle \Delta_1, \dots, \Delta_n \rangle \! \rangle$.
Moreover, $\Delta_i^2$ and
$[\Delta_i, (\partial_{x_i} f)(\Delta_1, \dots, \Delta_n)]$ belong
to $J$ for $1 \leq i \leq n$, so they belong to the 
ideal defining that quotient.

In order to show that they generate that ideal, we verify that the
remaining generators already lie in the ideal generated by $\Delta_i^2$
and $[\Delta_i, (\partial_{x_i} f)(\Delta_1, \dots, \Delta_n)]$.
Indeed,
\begin{align*}
\alpha_i \delta_i \cdot \partial_{\delta_i}(W) \cdot \delta_i \gamma_i
& = \Delta_i \cdot (\partial_{x_i} f)(\Delta_1, \dots, \Delta_n) \cdot
\Delta_i
\\
&=
[\Delta_i, (\partial_{x_i} f)(\Delta_1, \dots, \Delta_n)] \cdot \Delta_i 
+ (\partial_{x_i} f)(\Delta_1, \dots, \Delta_n) \cdot \Delta_i \Delta_i .
\end{align*}
\end{proof}

\begin{remark} \label{rem:gen1}
The well-known identity $\sum_{i=1}^n [x_i, \partial_{x_i}f] = 0$ implies
that in the description~\eqref{e:e0Le0}, one can take only $n-1$ of the
generators $[x_i, \partial_{x_i}f]$. To verify this identity, it is
enough to consider monomials and note that
\[
\sum_{i=1}^n x_i \cdot \partial_{x_i}(x_{i_1} x_{i_2} \dots x_{i_r})
= \sum_{s=1}^r x_{i_s} \dots x_{i_r} x_{i_1} \dots x_{i_{s-1}}
= \sum_{i=1}^n \partial_{x_i}(x_{i_1} x_{i_2} \dots x_{i_r}) \cdot x_i .
\]
\end{remark}

\subsection{Finite-dimensionality of quotients of
$K\langle \! \langle x_1, \dots, x_n \rangle \! \rangle$}
%%%%%%%%%%%%%%%%%%%%%%%%%%%%%%%%%%%%%%%%%%%%%%%%%%%%%%%%%

Let $A = K\langle \! \langle x_1, \dots, x_n \rangle \! \rangle$ be the
complete path algebra of the $n$-loop quiver, and let 
$\fm = (x_1, x_2, \dots, x_n)$ be the ideal of $A$ generated by the arrows.
Observe that $\fm$ is a closed ideal.
We shall view a monomial in $A$ as a word $u$ in the letters $x_1, \dots, x_n$
whose \emph{length} $\ell(u)$ is thus equal to the total degree of the monomial.

For any $r \geq 1$, a $K$-basis of the vector space $A/\fm^r$ is given
by all the words in the letters $x_1,\dots,x_n$ whose length is smaller
than $r$. Hence $A/\fm^r$ is finite-dimensional and
$\dim A/\fm^r = n^0 + n^1 + \dots + n^{r-1}$. It follows that for
any ideal $I$ containing $\fm^r$, the $K$-algebra $A/I$ is finite-dimensional.

The next lemma gives a useful criterion for a closed ideal to contain
some power of $\fm$.
\begin{lemma} \label{l:AIfd}
Let $I \trianglelefteq A$ be an ideal.
If $\fm^r \subseteq I + \fm^{r+1}$ for some $r \geq 1$,
then $\fm^r \subseteq \bar{I}$ and hence
$A/\bar{I}$ is finite-dimensional.
\end{lemma}
\begin{proof}
We show by induction on $i \geq 1$ that $\fm^r \subseteq I + \fm^{r+i}$.
Indeed, the claim holds for $i=1$ by assumption. If it is true for $i$,
then
\[
\fm^r \subseteq I + \fm^{r+i} = I + \fm^i \cdot \fm^r \subseteq
I + \fm^i \cdot (I + \fm^{r+1}) = I + \fm^{r+i+1},
\]
so it holds for $i+1$ as well. We conclude that
$\fm^r \subseteq \bigcap_{i \geq 1} (I + \fm^{r+i}) = \bar{I}$.
\end{proof}

We now consider a quantitative version.
\begin{notat}
Let $I \trianglelefteq A$. For $r \geq 0$, let
\[
d_r = d_r(I) = \dim (I + \fm^r) / (I + \fm^{r+1}).
\]
Observe that since $(I + \fm^r)/(I + \fm^{r+1})$ is a sub-quotient
of $A/\fm^{r+1}$, the numbers $d_r$ are well-defined non-negative
integers.
Moreover, 
$\dim A/(I+\fm^r) = d_0 + d_1 + \dots + d_{r-1}$.
\end{notat}

\begin{remark}
Since $I + \fm^i = \bar{I} + \fm^i$ for all $i \geq 0$, we have
$d_r(\bar{I}) = d_r(I)$ for all $r \geq 0$.
\end{remark}

\begin{prop} \label{p:AIfd}
Let $I \trianglelefteq A$ and
assume that $d_r(I) = 0$ for some $r \geq 0$. Then:
\begin{enumerate}
\renewcommand{\theenumi}{\alph{enumi}}
\item
The algebra $A/\bar{I}$ is finite-dimensional and 
$\dim A/\bar{I} = d_0 + d_1 + \dots + d_{r-1}$.

\item
$d_{r+i}(I) = 0$ for any $i \geq 0$.
\end{enumerate}
\end{prop}
\begin{proof}
The assumption $d_r=0$ implies that $I + \fm^r = I + \fm^{r+1}$.
By Lemma~\ref{l:AIfd}, $\fm^r \subseteq \bar{I}$ and the algebra $A/\bar{I}$
is finite-dimensional. The claim on its dimension follows from the equality
$\bar{I} + \fm^r = \bar{I}$.
Moreover, $d_{r+i}(I) = d_{r+i}(\bar{I}) = 0$ since
$\bar{I} + \fm^{r+i} = \bar{I}$ for any $i \geq 0$.
\end{proof}

We conclude this section by presenting an algorithm to compute the numbers
$d_0, \dots, d_r$ for a finitely generated ideal $I = (\rho_1, \dots, \rho_t)$
given by the list of its generators.

Note that from the inclusions
\begin{align*}
\fm^{r+1} \subseteq I + \fm^{r+1} \subseteq I + \fm^r &&
\fm^{r+1} \subseteq \fm^r \subseteq I + \fm^r
\end{align*}
we get by comparing dimensions that
\begin{align*}
\dim (I+\fm^r)/(I+\fm^{r+1}) &=
\dim \fm^r/\fm^{r+1} + \dim (I+\fm^r)/\fm^r - \dim (I+\fm^{r+1})/\fm^{r+1} \\
&= n^r - \dim (I+\fm^{r+1})/\fm^{r+1} + \dim (I+\fm^r)/\fm^r ,
\end{align*}
so it suffices to compute the dimension of each subspace
$(I+\fm^r)/\fm^r$ inside $A/\fm^r$.

We shall identify $A/\fm^r$ with the subspace of $A$ consisting of all
the linear combinations of words whose length is smaller than $r$.
Under this identification, the canonical projection
$A \twoheadrightarrow A/\fm^r$ is realised as the \emph{truncation}
$\tau_r \colon A \to A$, taking an element $a \in A$ and removing all the
words whose length is greater or equal to $r$.
Obviously, $\tau_r(a) - a \in \fm^r$ for any $a \in A$, hence
$u \tau_r(a) v - u a v \in \fm^{r+ \ell(u) + \ell(v)}$
for any two words $u, v \in A$.

\begin{lemma} \label{l:Imrmr}
Assume that $I$ is generated by a sequence $\rho_1, \rho_2, \dots, \rho_t$
of elements in $\fm^e$ for some $e \geq 0$.
Then $(I + \fm^r)/\fm^r$ is spanned by all the elements of the form
\begin{equation} \label{e:Imrmr}
u \cdot \tau_{r - \ell(u) - \ell(v)}(\rho_i) \cdot v ,
\end{equation}
where $1 \leq i \leq t$ and $u$ and $v$ are words such that
$\ell(u) + \ell(v) < r - e$.
\end{lemma}
\begin{proof}
As a vector space, $I$ is spanned by the elements $u \rho_i v$, where $u, v$
are words and $1 \leq i \leq t$. Hence $(I+ \fm^r)/\fm^r$ is spanned by
the corresponding elements $\tau_r(u \rho_i v)$.

Now if $\ell(u) + \ell(v) + e \geq r$, then $u \rho_i v \in \fm^r$
and its image under $\tau_r$ vanishes. Otherwise,
$\tau_r(u \rho_i v) = u \tau_{r - \ell(u) - \ell(v)}(\rho_i) v$ and the
result follows.
\end{proof}

Let $r > e$. 
We can encode the elements in~\eqref{e:Imrmr} in a matrix $M_r$
(with entries in $K$) as follows.
The rows of $M_r$ are indexed by the triples $(u,v,i)$ where $u,v$
are words such that $\ell(u)+\ell(v) < r-e$ and $1 \leq i \leq t$,
whereas its columns are indexed by the words $w$ with
$e \leq \ell(w) < r$. The entries of $M_r$ are given by
\begin{equation} \label{e:Mr}
(M_r)_{(u,v,i),w} = 
\text{the coefficient of $w$ in
$u \cdot \tau_{r-\ell(u)-\ell(v)}(\rho_i) \cdot v$.}
\end{equation}
Viewing the rows of $M_r$ as vectors in $\fm^e/\fm^r$, Lemma~\ref{l:Imrmr}
tells us that the subspace of $\fm^e/\fm^r$ spanned by the rows of $M_r$
equals $(I+\fm^r)/\fm^r$.

Let $p_r$ and $q_r$ denote the numbers of rows and columns in $M_r$,
respectively. Since the number of pairs $(u,v)$ of words with
$\ell(u)+\ell(v)=j$ equals $(j+1)n^j$, we have
\begin{align*}
p_r = t \cdot \sum_{j=0}^{r-e-1} (j+1) n^j &,&
q_r = \sum_{j=e}^{r-1} n^j .
\end{align*}
Moreover, if we order the rows $(u,v,i)$ according to $\ell(u)+\ell(v)$
and the columns $w$ according to $\ell(w)$, then for any $e < s \leq r$,
the upper left submatrix of size $p_s \times q_s$ of $M_r$ equals $M_s$,
since $\tau_s \tau_r = \tau_s$.

The preceding discussion can be summarized by the next algorithm.
\begin{alg} \label{alg:dr}
Given:
\begin{itemize}
\item
an integer $e \geq 0$;
\item
a sequence $\rho_1, \rho_2, \dots, \rho_t$ of elements in $\fm^e$;
\item
an integer $r \geq e$;
\end{itemize}
Output the numbers $d_0(I), d_1(I), \dots, d_r(I)$ for
$I = (\rho_1, \rho_2, \dots, \rho_t)$ by performing the following steps:
\begin{enumerate}
\renewcommand{\labelenumi}{\theenumi.}
\item
Compute the $p_{r+1} \times q_{r+1}$ matrix $M_{r+1}$ according to~\eqref{e:Mr}.

\item
Compute the matrices $M_{s+1}$ as submatrices of $M_{r+1}$
for $e \leq s < r$.

\item
Output
\[
d_s(I) =
\begin{cases}
n^s & \text{if $0 \leq s < e$,} \\
n^e - \rank M_{e+1} & \text{if $s=e$,} \\
n^s - \rank M_{s+1} + \rank M_s & \text{if $e < s \leq r$.}
\end{cases}
\]
\end{enumerate}
\end{alg}

\subsection{A computer-assisted proof of finite-dimensionality}
%%%%%%%%%%%%%%%%%%%%%%%%%%%%%%%%%%%%%%%%%%%%%%%%%%%%%%%%%%%%%
\label{ssec:computer}

We have implemented Algorithm~\ref{alg:dr} using the \textsc{Magma}
computer algebra system~\cite{MAGMA} (in order to be able to work with
matrices over finite fields) for the ideals
$I_f = \bigl( x_i^2, [x_i, \partial_{x_i} f] \bigr)_{1 \leq i \leq 3}$
of $K \llangle x_1, x_2, x_3 \rrangle$,
where $f \in \fm^2$ is a polynomial.
In this case, $e=2$ and in view of Remark~\ref{rem:gen1}, one needs to
take only two commutators, so $t=5$ generators suffice.

For each of the two polynomials $f=f_0$ and $f=f_1$ given by
\begin{align*}
f_0(x_1, x_2, x_3) &= x_1 x_2 + x_2 x_3 + x_3 x_1 \\
f_1(x_1, x_2, x_3) &= x_1 x_2 + x_2 x_3 + x_3 x_1 + x_1 x_2 x_3 ,
\end{align*}
over each of the base fields $K=\bQ$ and $K=\bF_2$,
the numbers $d_r(I_f)$ for $0 \leq r \leq 6$ are shown in Table~\ref{tab:dr}.

\begin{table}
\[
\begin{array}{cc|ccccccc}
K & f & d_0 & d_1 & d_2 & d_3 & d_4 & d_5 & d_6 \\
\hline
\bQ & f_0 &  1 & 3 & 4 & 3 & 1 & 0 & 0 \\
\bQ & f_1 &  1 & 3 & 4 & 3 & 1 & 0 & 0 \\
\bF_2 & f_0 &1 & 3 & 4 & 4 & 4 & 4 & 4 \\
\bF_2 & f_1 &1 & 3 & 4 & 4 & 3 & 1 & 0
\end{array}
\]
\caption{The numbers $d_r$ for some ideals
$\bigl( x_i^2, [x_i, \partial_{x_i} f] \bigr)_{1 \leq i \leq 3}$
of $K \llangle x_1, x_2, x_3 \rrangle$.
}
\label{tab:dr}
\end{table}

Note that in order to compute $d_6$, a matrix with
$2735 = 5 \cdot (1 + 2 \cdot 3 + 3 \cdot 9 + 4 \cdot 27 + 5 \cdot 81)$
rows and $1089 = 9 + 27 + 81 + 243 + 729$ columns has to be constructed
and analysed.

\begin{remark}
It may be possible to use the generators $x_i^2$ in order to better encode
the words by considering only those with different consecutive letters, so
that their growth would be exponential with base $2$ rather than $3$,
thus significantly lowering the sizes of the matrices involved.
We have not implemented this yet.
\end{remark}

Consider now the potentials
\[
W_\eps = B_1 + B_2 + B_3 + \Delta_1 \Delta_2 + \Delta_2 \Delta_3
+ \Delta_3 \Delta_1 + \eps \Delta_1 \Delta_2 \Delta_3
\]
on $X_7$
for $\eps \in \{0,1\}$.
By examining the entries in Table~\ref{tab:dr}, we deduce the following:
\begin{prop} \label{p:W0W1dim} {\ }
\begin{enumerate}
\renewcommand{\theenumi}{\alph{enumi}}
\item
If $\ch K = 0$, the algebras $\cP(X_7, W_0)$ and $\cP(X_7, W_1)$
are finite-dimensional.

\item
If $\ch K = 2$, the algebra $\cP(X_7, W_1)$ is finite-dimensional
and the two potentials $(X_7, W_0)$ and $(X_7, W_1)$ are not right equivalent.
\end{enumerate}
\end{prop}
\begin{proof}
Let $\eps \in \{0,1\}$. By Lemma~\ref{l:KQn0fd}
and Lemma~\ref{l:e0Le0}, the finite-dimensionality of the algebra
$\cP(X_7, W_\eps)$ is equivalent to that of
$K \llangle x_1,x_2,x_3 \rrangle / \overline{I_{f_\eps}}$.
For the latter algebras, use Proposition~\ref{p:AIfd}, observing that
$d_5(I_{f_0})=d_5(I_{f_1})=0$ for $K=\bQ$ and $d_6(I_{f_1})=0$
for $K=\bF_2$.

Any right equivalence of $W_0$ and $W_1$ would imply the isomorphism
of the algebras $e_0 \cP(X_7, W_\eps) e_0$, but when $\ch K = 2$
these algebras cannot have the same dimension by virtue of
Table~\ref{tab:dr}.
\end{proof}

Let us explain how to extend these computational results from the case of
characteristic $0$ to arbitrary characteristic (not equal to $2$)
by working with lattices over $\bZ$.

Since the coefficients of $f_0$ and $f_1$ are integers, the matrices
$M_s$ constructed by the algorithm have integer entries and as such,
one can consider their rank over any finite field $\bF_p$. Moreover,
for any integral matrix $M$ we always have
$\rank_{\bF_p} M \leq \rank_{\bQ} M$.

If $f$ has integer coefficients, denote by $d_r^{\bF_p}$
the number $d_r(I_f)$ computed for the ideal $I_f$
in $\bF_p \llangle x_1, x_2, x_3 \rrangle$, and similarly for
$d_r^{\bQ}$.
The next lemma gives a sufficient condition to deduce the vanishing of
$d_r^{\bF_p}$ given that of $d_r^{\bQ}$.

\begin{lemma} \label{l:dr:p0}
Assume that $d_r^{\bQ} = 0$. If $\rank_{\bF_p} M_{r+1} = \rank_{\bQ} M_{r+1}$
then $d_r^{\bF_p} = 0$.
\end{lemma}
\begin{proof}
We have
\[
d_r^{\bF_p} = n^r - \rank_{\bF_p} M_{r+1} + \rank_{\bF_p} M_r  \leq
n^r - \rank_{\bQ} M_{r+1} + \rank_{\bQ} M_r = d_r^{\bQ} = 0 .
\]
\end{proof}

Our computational tool to verify the equality of ranks is based on the
following lemma.
\begin{lemma} \label{l:prank}
Let $L \subseteq \bZ^N$ be a lattice of rank $N-d$ and let
$v_1, v_2, \dots v_d \in \bZ^N$ be vectors such that
$L' = L + \sum \bZ v_i$ has full rank. 
If $p$ is a prime that does not divide the index $[\bZ^N : L']$, then
$\dim_{\bF_p} (L \otimes_{\bZ} \bF_p) = \dim_{\bQ} (L \otimes_{\bZ} \bQ)$.
\end{lemma}
\begin{proof}
It suffices to prove that $L' \otimes \bF_p = \bF_p^N$, as
\[
\dim_{\bF_p} (L' \otimes \bF_p) - d \leq
\dim_{\bF_p} (L \otimes \bF_p) \leq
\dim_{\bQ} (L \otimes \bQ) = \rank L = N - d .
\]

Indeed,
let $w_1, \dots, w_N$ be a $\bZ$-basis of $L'$ and consider the square
integral matrix whose rows are given by the vectors $w_i$.
The absolute value of its determinant equals the index $[\bZ^N : L']$,
which by assumption is not divisible by $p$.
Hence the images modulo $p$ of the vectors $w_1, \dots, w_n$
form a basis of $\bF_p^N$.
\end{proof}

\begin{prop} \label{p:W0W1:Kne2}
Let $K$ be a field with $\ch K \neq 2$. Then the algebras
$\cP(X_7, W_0)$ and $\cP(X_7, W_1)$ are finite-dimensional over $K$.
\end{prop}
\begin{proof}
The case $\ch K = 0$ has already been considered in
Proposition~\ref{p:W0W1dim}. Thus it suffices to prove that
$d_5^{\bF_p}(I_{f_0}) = d_5^{\bF_p}(I_{f_1}) = 0$ for any prime $p>2$.

From Table~\ref{tab:dr} we already know that
$d_5^{\bQ}(I_{f_0}) = d_5^{\bQ}(I_{f_1}) = 0$, thus by Lemma~\ref{l:dr:p0}
we are done once we prove that in each case the rank of the matrix $M_6$ over
$\bF_p$ is equal to that over $\bQ$.

The matrix $M_6$ has $360=9+27+81+243$ columns. One
verifies that $\rank_{\bQ} M_6 = 352$ (this follows also from the equality
$d_2 + d_3 + d_4 = 8$, see Table~\ref{tab:dr}).

Let $L'$ be the lattice 
in $\bZ^{360}$ spanned by the rows of $M_6$ together with the $8$
standard basis vectors in $\bZ^{360}$ corresponding to the words
\begin{align*}
x_1 x_2, && x_1 x_3, && x_2 x_1, && x_2 x_3, && x_1 x_2 x_1, && x_1 x_3 x_2,
&& x_2 x_1 x_2, && x_1 x_2 x_1 x_2 .
\end{align*}

One verifies using \textsc{Magma}~\cite{MAGMA} that $L'$ has full rank
and that $[\bZ^{360} : L'] = 256$.
By Lemma~\ref{l:prank} (with $N=360$, $d=8$ and
$L$ being the lattice spanned by the rows of $M_6$), 
$\rank_{\bF_p} M_6 = \rank_{\bQ} M_6$ for any prime $p>2$.
\end{proof}

\section{A potential on $X_7$ with infinite-dimensional Jacobian algebra}
%%%%%%%%%%%%%%%%%%%%%%%%%%%%%%%%%%%%%%%%%%%%%%%%%%%%%%%%%%%%%%%%%%%%%%%%%
\label{sec:Jac:infdim}

\subsection{Graded quivers with homogeneous commutativity-relations}
%%%%%%%%%%%%%%%%%%%%%%%%%%%%%%%%%%%%%%%%%%%%%%%%%%%%%%%%%%%%%%%%%%%%

In this section we present a general result about certain
quotients of complete path algebras being infinite-dimensional.

Let $Q$ be a quiver. Recall that a $\emph{grading}$ on $Q$ is a function
$|\cdot| \colon Q_1 \to \bZ$. We say that $Q$ is \emph{positively graded}
if the degree of any arrow is a positive integer.
The degree of a path is the sum of degrees of its arrows. This induces
gradings on the path algebra $KQ$ and its completion $\wh{KQ}$.

Since the grading is positive,
the degree of any path of length $n$ is at least $n$; in other words,
any path of degree $d$ has length at most $d$. This implies that
the subspace $\wh{KQ}_d$ spanned by the paths of
degree $d$ is finite-dimensional and equals $(KQ)_d$.
We have $\wh{KQ} = \prod_{d \geq 0} \wh{KQ}_d$.
Let $\wh{KQ}_{\geq d} = \prod_{d' \geq d} \wh{KQ}_{d'}$.
If $\fm$ denotes the ideal of $\wh{KQ}$ generated by the arrows, then
$\fm^n \subseteq \wh{KQ}_{\geq n}$ for any $n \geq 1$,
hence $\bar{I} \subseteq \bigcap_{d \geq 0} (I + \wh{KQ}_{\geq d})$
for any ideal $I$ of $\wh{KQ}$
(in fact there is an equality but we will not use this).

\begin{lemma} \label{l:KQI:graded}
Let $Q$ be positively graded. If $I$ is an ideal of $\wh{KQ}$
generated by finitely many homogeneous elements, then
$I$ is closed and $\wh{KQ}/I \simeq \prod_{d \geq 0} {KQ}_d/I_d$.
\end{lemma}
\begin{proof}
Write $I = (\rho_1, \dots, \rho_t)$ for homogeneous elements
$\rho_1, \dots_, \rho_t$.
For any $d \geq 0$, let $I_d = I \cap \wh{KQ}_d$ be the subset
consisting of the elements in $I$ having degree $d$ and consider
$I' = \prod_{d \geq 0} I_d$. 
Then $I'$ is an ideal of $\wh{KQ}$ containing $\rho_1, \dots, \rho_t$
and hence it contains $I$. Since
$I' + \wh{KQ}_{\geq d} =
\left(\prod_{d' < d} I_{d'}\right) \times
\bigl(\prod_{d' \geq d} \wh{KQ}_{d'}\bigr)$, the intersection
$\bigcap_{d \geq 0} (I' + \wh{KQ}_{\geq d})$ equals $I'$
and therefore $I'$ is closed. Moreover, from the construction of $I'$
we have $\wh{KQ}/I' \simeq \prod_{d \geq 0} KQ_d/I_d$.
Thus the claim of the lemma would follow once we show that $I'=I$.

Indeed, let $x=(x_d)_{d \geq 0} \in I'$ with $x_d \in I_d$.
For any $d \geq 0$, there exist homogeneous elements
$a_{1,d}, \dots, a_{t,d}$ such that
$x_d = a_{1,d} \rho_1 + \dots + a_{t,d} \rho_t$
and $|a_{i,d}| = d - |\rho_i|$ (if $|\rho_i|>d$ then $a_{i,d}=0$).
Let $a_i = \sum_{d \geq |\rho_i|} a_{i,d}$ for $1 \leq i \leq t$.
Then $x = a_1 \rho_1 + \dots + a_t \rho_t \in I$, as required.
\end{proof}

Recall that two paths in $Q$ are \emph{parallel} if they have the same
starting point and the same ending point.
A \emph{commutativity-relation} is an element in $KQ$ of the form
$p-q$ where $p,q$ are parallel paths. 
A commutativity-relation is \emph{homogeneous} if the paths $p$ and
$q$ have the same degree.

\begin{prop} \label{p:KQI:hcomm}
Let $Q$ be a positively graded quiver and let $I$ be an ideal of $\wh{KQ}$
generated by finitely many homogeneous commutativity-relations.
Then $I$ is closed and $\dim \wh{KQ}/I < \infty$ if and only if
$Q$ is acyclic.
\end{prop}
\begin{proof}
Let $I=(p_1-q_1, \dots, p_t-q_t)$, where 
$(p_i, q_i)$ are pairs of parallel paths of the same degree
$|p_i|=|q_i|$.
By Lemma~\ref{l:KQI:graded}, $I$ is closed and
$\wh{KQ}/I \simeq \prod_{d \geq 0} KQ_d/I_d$.

Define a binary relation on the set of paths of $Q$ by
$u p_i v \sim u q_i v$ for any $1 \leq i \leq t$ and paths $u, v$
such that the concatenation $u p_i v$ (and hence also $u q_i v$) is defined.
Its transitive closure defines an equivalence relation on the set of
paths of $Q$. Since $|p_i|=|q_i|$ for any $1 \leq i \leq t$,
equivalent paths have the same degree and a basis of the space $KQ_d/I_d$
is given by representatives of the equivalence classes of paths of degree $d$.
Thus, $KQ_d/I_d \neq 0$ if there exists at least one path in $Q$ of degree $d$.

Since the grading is positive, if $Q$ is not acyclic then there exist
paths of arbitrarily large degree, and hence infinitely many of the factors
$KQ_d/I_d$ are not zero, so that $\wh{KQ}/I$ is infinite-dimensional.
If $Q$ is acyclic, then $\wh{KQ}$ equals the path algebra $KQ$
which is finite-dimensional over $K$ and so is any of its quotients.
\end{proof}

Let us derive a version of Proposition~\ref{p:KQI:hcomm} for potentials.
Call a cycle $c$ \emph{simple} if no arrow appears in $c$ more than once.

\begin{prop} \label{p:QP:cycles}
Let $Q$ be a positively graded quiver which is not acyclic,
let $\cC$ be a set of simple cycles of the same positive degree
which are pairwise rotationally inequivalent, and
let $\cA$ be the set
of all arrows appearing in at least one of the cycles in $\cC$.

Let $K$ be a field and
assume that there exists a function $\eps \colon \cC \to K$ such that
any arrow $\alpha \in \cA$ appears in exactly two cycles $c \neq c'$ in $\cC$
and moreover $\eps(c)+\eps(c')=0$.

Then the Jacobian algebra $\cP \left(Q, \sum_{c \in \cC} \eps(c) c \right)$
is infinite-dimensional over $K$.
\end{prop}

\begin{remark} \label{rem:QP:cycles2}
If $\ch K = 2$, one can take the constant function $\eps$ defined by
$\eps(c)=1$ for any $c \in \cC$.
If any arrow in $\cA$ appears in exactly two cycles in $\cC$, then the
Jacobian algebra of the potential $\sum_{c \in \cC} c$ on $Q$
is infinite-dimensional over $K$.
\end{remark}

\begin{remark}
When $\ch K \neq 2$, the condition on the function $\eps$ can be rephrased
in combinatorial terms as follows; there is a partition
$\cC = \cC^+ \cup \cC^-$ into two disjoint sets
such that any arrow in $\cA$ belongs to exactly
one cycle in $\cC^+$ and one cycle in $\cC^-$. Then the Jacobian algebra
$\cP(Q, \sum_{c \in \cC^+} c - \sum_{c \in \cC^-} c)$ is infinite-dimensional over
$K$.
\end{remark}

\begin{proof}[Proof of Proposition~\protect{\ref{p:QP:cycles}}]
Observe that if an arrow $\alpha$ appears in a simple cycle $c$, then
$\partial_{\alpha}(c) = p$ for a path $p$ with $|p| = |c| - |\alpha|$.

Let $W = \sum_{c \in \cC} \eps(c) c$, let $\alpha$ be an arrow and consider
the cyclic derivative $\partial_\alpha(W)$.
If $\alpha \not \in \cA$ then $\partial_{\alpha}(W) = 0$.
Otherwise, let $c, c'$ be the two cycles in $\cC$ that $\alpha$ appears in.
Then there exist parallel paths $p$ and $p'$ such that $\partial_\alpha(c)=p$,
$\partial_\alpha(c')=p'$ and $|p| = |c| - |\alpha| = |c'| - |\alpha| = |p'|$,
hence $\partial_\alpha(W) = \eps(c)p + \eps(c')p'$ is a scalar multiple
by $\eps(c)$ of a homogeneous commutativity-relation
due to our assumption that $\eps(c) = - \eps(c')$.

We have shown that the ideal of $\wh{KQ}$ generated by the cyclic
derivatives of $W$ is generated by homogeneous commutativity-relations.
By Proposition~\ref{p:KQI:hcomm}, this ideal is closed and moreover the
quotient, which thus equals the Jacobian algebra, is infinite-dimensional
over $K$.
\end{proof}

\begin{remark}
The conclusion of Proposition~\ref{p:QP:cycles} is false without the
assumption that $Q$ is positively graded and the cycles in $\cC$ are
of the same positive degree. Prominent counterexamples are the potentials
associated to triangulations of closed surfaces with punctures,
defined in~\cite{Labardini09}. 
As shown in~\cite{Ladkani12}, many of these potentials are of the form
$\sum_{c \in \cC} \eps(c) c$ with the property that any arrow belongs to
exactly two cycles $c,c'$ with $\eps(c)+\eps(c')=0$
(called there \emph{$f$-cycle} and \emph{$g$-cycle}),
however their Jacobian algebras are finite-dimensional over any ground field.
\end{remark}

\subsection{Positive grading on $X'_7$ and the Jacobian algebra $\cP(X_7,W_0)$}
%%%%%%%%%%%%%%%%%%%%%%%%%%%%%%%%%%%%%%%%%%%%%%%%%%%%%%%%%%%%%%%%%%%%%%%%%%%%%%%%

Consider the positive grading on the quiver $X'_7$ given by
\begin{align*}
|a_i|=2, && |b_{ij}|=2, && |c_i|=2, && |d_i|=1
\end{align*}
for $1 \leq i \neq j \leq 3$.

\begin{lemma} \label{l:X7:hcycle}
Let $\cC$ be the set consisting of the following $9$ cycles in $X'_7$:
\[
\cC = \left\{ a_i b_{ij} c_j \right\}_{1 \leq i \neq j \leq 3}
\cup
\left\{ d_i b_{ij} d_j b_{ji} \right\}_{1 \leq i < j \leq 3} .
\]
Then any cycle in $\cC$ has degree $6$ and any arrow of $X'_7$ belongs
to exactly two cycles in~$\cC$.
\end{lemma}
\begin{proof}
One verifies that $|a_i b_{ij} c_j|=6$ and $|d_i b_{ij} d_j b_{ji}|=6$.
The table below lists, for each arrow (in the left column),
two cycles rotationally equivalent to
the two cycles in $\cC$ it belongs to. Here, $i,j,k$ indicate different
indices from the set $\{1,2,3\}$.
\begin{align*}
&a_i    && a_i b_{ij} c_j && a_i b_{ik} c_k && (1 \leq i \leq 3) \\
&b_{ij} && a_i b_{ij} c_j && d_i b_{ij} d_j b_{ji} && (1 \leq i \neq j \leq 3) \\
&c_i    && a_j b_{ji} c_i && a_k b_{ki} c_i && (1 \leq i \leq 3) \\
&d_i    && d_i b_{ij} d_j b_{ji} && d_i b_{ik} d_k b_{ki} && (1 \leq i \leq 3)
\end{align*}
\end{proof}

Recall that $W_0 = B_1 + B_2 + B_3 + \Delta_1 \Delta_2 + \Delta_2 \Delta_3
+ \Delta_3 \Delta_1$ on $X_7$.

\begin{prop} \label{p:W0:Keq2}
If $\ch K = 2$,
the Jacobian algebra $\cP(X_7, W_0)$ is infinite dimensional.
\end{prop}
\begin{proof}
Since the finite-dimensionality of Jacobian algebras is preserved under
mutations, it suffices to show that the Jacobian algebra of the mutation
$\mu_0(X_7, W_0)$ is infinite-dimensional.

By Proposition~\ref{p:Qn:Wf}, $\mu_0(X_7, W_0) = (X'_7, W_0')$,
where the potential $W'_0$ is given by 
\[
W'_0 = \sum_{1 \leq i \neq j \leq 3} a_i b_{ij} c_j +
\sum_{1 \leq i < j \leq 3} d_i b_{ij} d_j b_{ji} ,
\]
which is exactly $\sum_{\omega \in \cC} \omega$ for the set of cycles
$\cC$ considered in Lemma~\ref{l:X7:hcycle}.
The result is now a consequence of Proposition~\ref{p:QP:cycles}
(and Remark~\ref{rem:QP:cycles2}).
\end{proof}

\begin{remark}
Since the Jacobian algebra $\cP(X_7, W_0)$ is finite-dimensional
when $\ch K \ne 2$ (cf.\ Proposition~\ref{p:W0W1:Kne2}),
the same holds for $\cP(X'_7, W'_0)$. We see that the conclusion of
Proposition~\ref{p:QP:cycles} is false without the assumption on the
function $\eps \colon \cC \to K$.
\end{remark}

\subsection{No reddening sequences on $X_7$}
%%%%%%%%%%%%%%%%%%%%%%%%%%%%%%%%%%%%%%%%%%%%
\label{ssec:X7:reddening}

We recall the definition of reddening sequences from~\cite{Keller13},
see also the survey~\cite{DemonetKeller20}.

Let $Q$ be a quiver without loops and $2$-cycles.
The \emph{framed quiver} $\wh{Q}$ is obtained from $Q$
by adding, for each vertex $i \in Q_0$, a new vertex $i'$
and a new arrow $i \to i'$. Call the new vertices \emph{frozen},
as we will not mutate at them.

Let $R$ be a quiver obtained from $\wh{Q}$ by applying a finite sequence
of mutations at non-frozen vertices.
A non-frozen vertex $i$ is \emph{green} if there are no arrows $j' \to i$
from any frozen vertex to $i$. It is \emph{red} if there are no arrows
$i \to j'$ from $i$ to any frozen vertex.
By the \emph{sign coherence of $c$-vectors}, established in~\cite{DWZ10}
using the theory of decorated representations of Jacobian algebras of
quivers with potentials,
each non-frozen vertex is either green or red.

Initially, all the non-frozen vertices of $\wh{Q}$ are green.
A sequence of non-frozen vertices $i_1, i_2, \dots, i_m$
is a \emph{reddening sequence} (or \emph{green-to-red} sequence)
for $Q$
if all the non-frozen vertices in the quiver
$\mu_{i_m} \dots \mu_{i_2} \mu_{i_1}(\wh{Q})$ are red.
A reddening sequence is \emph{maximal green} if,
for each $1 < t \leq m$, the vertex $i_t$ is green in
the quiver $\mu_{i_{t-1}} \dots \mu_{i_2} \mu_{i_1}(\wh{Q})$.

In~\cite{Seven14}, Seven gave a combinatorial proof that
there are no maximal green sequences (and presumably no reddening
sequences) on the quiver $X_7$.
We can now present a representation theoretic argument, relying
on~\cite{BDP14}.

\begin{prop} \label{p:X7:reddening}
There are no reddening sequences for the quiver $X_7$.
\end{prop}
\begin{proof}
Over a ground field of characteristic $2$, the potential $W_0$ on $X_7$
is non-degenerate and its Jacobian algebra $\cP(X_7, W_0)$ is not 
finite-dimensional.

The result now follows from
\cite[Proposition~8.1]{BDP14} or~\cite[Corollary~5.2]{DemonetKeller20}.
\end{proof}

\subsection{On the exchange graph of cluster-tilting objects}
%%%%%%%%%%%%%%%%%%%%%%%%%%%%%%%%%%%%%%%%%%%%%%%%%%%%%%%%%%%%%
\label{ssec:X7:exchange}

Recall that
a $K$-linear triangulated category $\cC$ with suspension functor
$\Sigma$ and finite-dimensional morphism spaces is \emph{$2$-Calabi-Yau}
if there are bifunctorial isomorphisms
$\Hom_{\cC}(X, \Sigma Y) \simeq D\Hom_{\cC}(Y, \Sigma X)$
for any $X, Y \in \cC$, where
$D(-)$ denotes the duality of vector spaces over $K$.

Assume that $\cC$ is also Krull-Schmidt.
An object $T \in \cC$ is a \emph{cluster-tilting object} if
$\Hom_{\cC}(T, \Sigma T) = 0$ and moreover any object $X \in \cC$ for
which $\Hom_{\cC}(T, \Sigma X)=0$ is 
a summand of a finite direct sum of copies of $T$.
We will consider \emph{basic} cluster-tilting
objects, i.e.\ the multiplicity in $T$ of any of its
indecomposable summands is equal to one.

Iyama and Yoshino have shown~\cite[Theorem~5.3]{IyamaYoshino08} that
it is possible to mutate (basic) cluster-tilting objects in the following
sense; if 
$T$ is a cluster-tilting object of $\cC$ and $X$ is an indecomposable
summand of $T$, then there 
exists a unique indecomposable object $X'$ of $\cC$ which is not
isomorphic to $X$ such that 
if we write $T = \bar{T} \oplus X$, then
$T' = \bar{T} \oplus X'$ is also a cluster-tilting
object in $\cC$. We say that $T'$ is the \emph{mutation} of $T$ at $X$.

The \emph{exchange graph} of the cluster-tilting objects in $\cC$
has its set of vertices equal to the set of isomorphism classes of the
basic cluster-tilting objects in $\cC$, and any two vertices
$T$ and $T'$ are connected by an edge if $T'$ is a mutation of $T$.

For any quiver with potential $(Q,W)$ whose Jacobian algebra is
finite-dimensional, Amiot~\cite{Amiot09} has constructed
the (generalised) \emph{cluster category} $\cC_{(Q,W)}$,
which is a Krull-Schmidt 2-Calabi-Yau triangulated category with a
cluster-tilting object.

\begin{prop} \label{p:X7:exchange}
Let $W$ be a non-degenerate potential on the quiver $X_7$ such that
the Jacobian algebra $\cP(X_7, W)$ is finite-dimensional.
Then the exchange graph of the cluster-tilting objects in the
cluster category $\cC_{(X_7, W)}$ is not connected.
\end{prop}
\begin{proof}
Let $T$ be the canonical cluster-tilting object in $\cC_{(X_7, W)}$.
Then $\Sigma T$ is also a cluster-tilting object. Any sequence of
mutations from $T$ to $\Sigma T$ would yield a reddening sequence
for $X_7$.
\end{proof}

\begin{remark}
In view of Propositions~\ref{p:W0W1dim} and~\ref{p:W0W1:Kne2},
the statement of Proposition~\ref{p:X7:exchange}
applies to the cluster categories $\cC_{(X_7, W_0)}$ (when the
characteristic of the ground field differs from $2$) and
$\cC_{(X_7, W_1)}$ (over any ground field).
\end{remark}

\bibliographystyle{amsplain}
\bibliography{potX7}

\end{document}